\documentclass[11pt]{article}

\usepackage{amsfonts}
\usepackage{amscd}
\usepackage{amssymb}
\usepackage{amsthm}
\usepackage{amsmath}
\usepackage{stmaryrd}
\usepackage{graphicx}
\usepackage{color}
\usepackage{verbatim}
\usepackage{mathrsfs}
\usepackage{dsfont}
\usepackage{tikz}
\usepackage{subfigure}
\usetikzlibrary{arrows,decorations.pathmorphing,backgrounds,positioning,fit}

\colorlet{Green}{black!20!green}

 \theoremstyle{plain}
\newtheorem{thm}{Theorem}[section]
\newtheorem{lemma}[thm]{Lemma}
\newtheorem{prop}[thm]{Proposition}
\newtheorem{cor}[thm]{Corollary}
\theoremstyle{definition}
\newtheorem{example}[thm]{Example}
\newtheorem{defn}[thm]{Definition}
\newtheorem{remark}[thm]{Remark}
\theoremstyle{remark}

\numberwithin{equation}{section}

\def\cA{\mathcal{A}}
\def\cB{\mathcal{B}}

\def\cD{\mathcal{D}}

\def\cF{\mathcal{F}}
\def\cG{\mathcal{G}}

\def\cN{\mathcal{N}}

\def\cS{\mathcal{S}}
\def\cT{\mathcal{T}}

\def\EE{\mathbb{E}}

\def\HH{\mathbb{H}}

\def\NN{\mathbb{N}}

\def\PP{\mathbb{P}}
\def\QQ{\mathbb{Q}}
\def\RR{\mathbb{R}}
\def\SS{\mathbb{S}}

\def\ZZ{\mathbb{Z}}

\def\vrho{\varrho}

\newcommand{\bT}{\mathbf{T}}
\newcommand{\bW}{\mathbf{W}}

\newcommand{\1}{\mathbf{1}}

\newcommand{\x}{\mathbf{x}}

\colorlet{lightgray}{black!20}

\makeatletter
\renewcommand{\@makefnmark}{\mbox{\textsuperscript{}}}
\makeatother

\title{Limit theorems for random walks on Fuchsian buildings and Kac-Moody groups}
\author{
L. Gilch \and S. M\"{u}ller\footnote{Research partly supported under CIRM Research in Pairs program} \and J. Parkinson\footnote{Research partly supported under the Australian Research Council (ARC) discovery grant DP110103205, Austrian Science Fund W1230 and P24028, and CIRM Research in Pairs program}
}

\begin{document}

\maketitle

\begin{abstract}
In this paper we prove a rate of escape theorem and a central limit theorem for isotropic random walks on Fuchsian buildings, giving formulae for the speed and asymptotic variance. In particular, these results apply to random walks induced by bi-invariant measures on Fuchsian Kac-Moody groups, however they also apply to the case where the building is not associated to any reasonable group structure. Our primary strategy is to construct a renewal structure of the random walk. For this purpose we define cones and cone types for buildings and prove that the corresponding automata in the building and the underlying Coxeter group are strongly connected. The limit theorems are then proven by adapting the techniques in~\cite{HMM:13}. The moments of the renewal times are controlled via the retraction of the walks onto an apartment of the building.
~\\
\newline{\scshape Keywords:} random walk, limit theorems, Fuchsian building, Kac-Moody group, Cannon automaton
\newline \scshape 2010 Mathematics Subject Classification: 60G50, 60F05, 60B15, 20E42, 51E24 
\end{abstract}
\section{Introduction}

Let $(X_{n})_{n\geq 1}$ be i.i.d.~random variables taking values in $\ZZ^{d}$. Under a second moment condition the classical central limit theorem gives  
$$\frac{\sum_{i=1}^{n} X_{i}-nv}{\sqrt n } \stackrel{\mathcal{D}}{\longrightarrow} \mathcal{N}(0,\sigma^{2}),$$  
where $v=\EE[X_{1}]$ is the rate of escape (or drift) and $\sigma^{2}$ the asymptotic variance. A natural and influential question, dating back to Bellman~\cite{bellman} and Furstenberg and Kesten~\cite{FK}, is to what extent this phenomenon generalizes to the situation where $(X_{n})_{n\geq 1}$ takes values in a group, or more generally, the situation where $(X_n)_{n\geq 1}$ is a random walk on a graph.

There are various settings in which central limit theorems have been
established, with key results in the contexts of Lie groups and hyperbolic
groups. In the hyperbolic setting, Sawyer and Steger~\cite{ST} studied the case
of the free group $F_{d}$ with~$d$ standard generators and the corresponding
word distance $d(\cdot,\cdot)$. Under a technical moment condition they show,
using analytic extensions of Green functions, that $(d(e,X_n)-nv)/\sqrt{n}$
converges in law to some non-degenerate Gaussian distribution, where $e$ is the group
identity in $F_{d}$. Another proof was given by Lalley~\cite{La:93} using
algebraic function theory and Perron-Frobenius theory, and a geometric proof
was later presented by Ledrappier~\cite{Le:01}. A generalization to trees with
finitely many cone types can be found in Nagnibeda and Woess~\cite{NW}. Another
generalization to free products of graphs was given by Gilch \cite{gilch:PhD}.

More recently Bj\"{o}rklund \cite{bjorklund} proved a central limit theorem for hyperbolic groups with respect to the Green metric, and this was pushed forward by Benoist and Quint~\cite{BeQu:14} for random walks on hyperbolic groups with respect to the word metric under the optimal second moment condition.

Another approach to the central limit theorem for surface groups has been developed by Haissinski, Mathieu, and M\"{u}ller~\cite{HMM:13}, where the planarity and hyperbolicity of the Cayley graph are employed to develop a renewal theory for random walks on these groups. The resulting central limit theorem comes complete with formulae for the speed and variance of the walk. 

Central limit theorems for semisimple real Lie groups were established by Wehn~\cite{wehn:62}, Tutubalin~\cite{Tutubalin:65}, Virtser~\cite{virtser}, Stroock and Varadhan \cite{Stroock:73}, and Guivarc'h~\cite{Gui:80} in a variety of contexts, using a wide range of techniques. There is an extensive literature on this subject, with further limit theorems for real Lie groups given in \cite{BeQu:13, bougerol, GoGu:96, guivarc'h2,GLeP:04,kaimanovich,LePage}. 

The case of $p$-adic Lie groups is also rather well understood, with central limit theorems established by Lindlbauer and Voit~\cite{LV}, Cartwright and Woess~\cite{CW}, and Parkinson~\cite{P3}. Further limit theorems for $p$-adic Lie groups are given in~\cite{BS,PS,tolli}. Many of these papers employ a remarkable geometric object called the \textit{affine building} associated to the $p$-adic group, and utilize the rich representation theory available in the $p$-adic setting. 

The current paper lies at the confluent of the hyperbolic and Lie theoretic settings. Here we prove limit theorems for random walks on \textit{Fuchsian buildings} and the \textit{Kac-Moody groups} associated to them (see below for some descriptions). From the point of view of Lie theory, this is a natural next step in the progression from `spherical-type' Lie groups (the semisimple real Lie groups) and `affine-type' Lie groups (the $p$-adic case) to a theory for random walks on buildings and Kac-Moody groups of arbitrary type. From the hyperbolic point of view, the buildings that we consider contain many copies of the hyperbolic disc tessellated using a `Fuchsian Coxeter group', and thus while the buildings are certainly not planar, some of the renewal theory techniques from the planar surface group case~\cite{HMM:13} can be pushed through.

Before stating our main results, let us give a brief description of the objects
involved in this paper. Buildings are geometric/combinatorial objects that can
be defined axiomatically. Initial data required to define a building includes a
Coxeter system~$(W,S)$, and then a \textit{building $(\Delta,\delta)$ of type
  $(W,S)$} consists of a set $\Delta$ (whose elements are the \textit{chambers}
of the building) along with a ``generalised distance function''
$\delta:\Delta\times\Delta\to W$ satisfying various axioms  (see
Definition~\ref{defn:building}). Thus the ``distance'' between chambers
$x,y\in\Delta$ is an element $\delta(x,y)$ of the Coxeter group~$W$, and by taking word length in~$W$ this
gives rise to a metric $d(\cdot,\cdot)$ on the building. We fix a base chamber $o\in\Delta$. The `spherical' buildings are those with $|W|<\infty$, and the `affine' buildings are those where $W$ is a Euclidean reflection group.


The theory of spherical and affine buildings has been extensively studied. However there are many Coxeter systems which are neither finite nor affine. Examples of buildings of these more `exotic types' arise naturally in Kac-Moody theory. These groups can be seen as generalisations of the classical `groups of Lie type', since they admit presentations reminiscent of the Chevalley presentations of the classical groups based around an associated Coxeter system~$(W,S)$ (see \cite{titskac}). To each Kac-Moody group of type $(W,S)$ there is naturally associated a building of type $(W,S)$, and the Kac-Moody group acts highly transitively on this building. 

While the above construction produces a lot of very interesting buildings, it is certainly not true that all buildings arise in this way (see~\cite{ronanconstruction}, for example). Thus in this paper we consider the building as the primary object of interest. Results concerning groups may then be deduced as corollaries, although it is important to note that our results apply equally well to the situation where there is no underlying group. (Indeed the building may have trivial automorphism group!).

In this paper we consider the natural class of \textit{isotropic random walks} $(X_n)_{n\geq 0}$ on buildings, where the transition probabilities $p(x,y)$ of the random walk depend only on the generalised distance~$\delta(x,y)$. If the building comes from a Kac-Moody group, then isotropic random walks are induced by measures on the group which are bi-invariant with respect to the `Borel subgroup'~$B$. The results of much of the preliminary sections are valid for buildings of any type, however our main results concern the class of \textit{Fuchsian buildings}. These are buildings whose Coxeter groups are discrete subgroups of $PGL_2(\mathbb{R})$, and since they are neither spherical nor affine they are an interesting ``non-classical'' class of buildings. 

We use a mixture of algebraic, geometric and probabilistic techniques. We observe that the transition operator of an isotropic random walk can naturally be regarded as an element of a \textit{Hecke algebra}, and we use this result to show that our buildings are nonamenable. Next we develop the theory of cones, cone types, and automata for buildings, and we show that the Cannon automaton of a Fuchsian building is \textit{strongly connected}. Connectivity properties of automata have various applications and are interesting in their own right. To our knowledge our results on the strong connectivity of the Cannon automaton are the first besides the trivial cases of free groups and surface groups. There are also some interesting features of cones in buildings that are in contrast to the theory of cones in groups. For example cones of the same type in the building are not necessarily isomorphic as graphs (see Remark~\ref{rem:noniso}).

We use our the theory of cones in buildings to develop a \textit{renewal theory} for the isotropic random walks on Fuchsian buildings. The idea here is to find a decomposition of the trajectory of the walk into aligned pieces in such a way that these pieces are identically and independently distributed. To do this we define renewal times $(R_n)_{n\geq 1}$ for the walk as follows. Fix a (recurrent) cone type $\bT$ and let $R_1$ be the first time that the walk visits a cone of type~$\bT$ and never leaves this cone again. Inductively define $R_{n+1}$ to be the first time after $R_n$ that the 
walk visits a cone of type $\bT$ and never leaves it again (see~(\ref{eq:renewal}) for a more formal definition). Our main results are as follows.

\begin{thm}\label{thm:LLNbuilding}
Let $(\Delta, \delta)$ be a regular Fuchsian building and let $(X_n)_{n\geq 0}$
be an isotropic random walk on~$\Delta$ with bounded range. Then,
\begin{align}\label{eq:v} 
\frac{1}{n} d(o,X_{n}) \stackrel{a.s.}{\longrightarrow} v=  \frac{\EE[ d(X_{ R_{2}}, X_{ R_{1}})]}{\EE[  R_{2}- R_{1}]}>0~\mbox{ as }  n\to\infty.
\end{align}
\end{thm}

\begin{thm}\label{thm:CLTbuilding}
Let $(\Delta, \delta)$ be a regular Fuchsian building and let $(X_n)_{n\geq 0}$ be an isotropic random walk on~$\Delta$ with bounded range. Then,
$$ \frac{d(o,X_{n}) -nv }{\sqrt{n}} \stackrel{\cD}{\longrightarrow} \cN (0, \sigma^2),$$
with $v$ as in (\ref{eq:v}) and 
\begin{align*}
\sigma^{2}=\frac{\EE[(d(X_{R_{2}}, X_{R_{1}}) - (R_{2}- R_{1})v)^{2}]}{\EE[R_{2}-R_{1}]}.
\end{align*}
\end{thm}

When a group acts suitably transitively on a regular Fuchsian building the above theorems give limit theorems for random walks associated to these groups. For example, suppose that $G$ is a Kac-Moody group over a finite field with Coxeter system~$(W,S)$ (see \cite{titskac}). Let~$B$ be the positive root subgroup of~$G$. Then $\Delta=G/B$ is the set of chambers of a locally finite regular building of type~$(W,S)$, where $\delta(gB,hB)=w$ if and only if $g^{-1}hB\subseteq BwB$. Then we have the following corollary.

\begin{cor}\label{cor:Kac}
Let $G$ be a Kac-Moody group over a finite field with Fuchsian Coxeter system~$(W,S)$, and let $(\Delta,\delta)$ be the associated Fuchsian building. Let $\varphi$ be the density function of a $B$-bi-invariant probability measure on~$G$, and assume that $\varphi$ is supported on finitely many $BwB$ double cosets. Then the assignment
$$
p(go,ho)=\varphi(g^{-1}h)
$$
defines an isotropic random walk on~$(\Delta,\delta)$, and Theorems~\ref{thm:LLNbuilding} and~\ref{thm:CLTbuilding} provide a rate of escape theorem and a central limit theorem for this random walk.
\end{cor}

To conclude this introduction, let us outline the structure of this paper. Section~\ref{sect:2} gives definitions and examples of Coxeter groups and buildings. In Section~\ref{sect:3} we develop the theory of automata for buildings (and Coxeter groups). In Section~\ref{sect:4} we introduce isotropic random walks on buildings, and use algebraic techniques to prove general results on irreducibility and the spectral radius. We also introduce the retracted walk in this section, which is a main tool in our investigations. In Section~\ref{sect:renewal} we restrict our attention to Fuchsian buildings, and develop renewal theory for isotropic random walks on these buildings. We prove our main theorems in this section, following the general proof strategy of~\cite{HMM:13}. Finally, in Appendix~\ref{app:A} we explicitly construct the automaton for each Fuchsian Coxeter system, and deduce that these automata are \textit{strongly connected} (a property that was useful in the work of Section~\ref{sect:renewal}).

\section{Coxeter groups and buildings}\label{sect:2}

\subsection{Coxeter systems}

A \textit{Coxeter system} $(W,S)$ is a group $W$ generated by a finite set~$S$ with relations
$$
s^2=1\quad\textrm{and}\quad (st)^{m_{st}}=1\quad\textrm{for all $s,t\in S$ with $s\neq t$},
$$
where $m_{st}=m_{ts}\in\ZZ_{\geq 2}\cup\{\infty\}$ for all $s\neq t$ (if $m_{st}=\infty$ then it is understood that there is no relation between $s$ and $t$). The \textit{rank} of $(W,S)$ is $|S|$. The \textit{length} of $w\in W$ is
$$
\ell(w)=\min\{n\geq 0\mid w=s_1\cdots s_n\textrm{ with }s_1,\ldots,s_n\in S\},
$$
and an expression $w=s_1\cdots s_n$ with $n=\ell(w)$ is called a \textit{reduced expression} for~$w$. If $w\in W$ and $s\in S$ then $\ell(ws)\in\{\ell(w)-1,\ell(w)+1\}$. In particular, $\ell(ws)=\ell(w)$ is impossible. The \textit{distance} between elements $u\in W$ and $v\in W$ is 
$$
d(u,v)=\ell(u^{-1}v).
$$
The \textit{ball} of radius $R\geq 0$ with centre $u\in W$ is $\cB(u,R)=\{v\in W\mid d(u,v)\leq R\}$ and the \textit{sphere} of radius $R\geq 0$ with centre $u\in W$ is $\cS(u,R)=\{v\in W\mid d(u,v)=R\}$. 

If $I\subseteq S$ let $W_I$ be the subgroup of $W$ generated by~$I$. Then
$(W_I,I)$ is a Coxeter system. The subgroup $W_I$ is called the
\textit{standard parabolic subgroup of type~$I$}. A Coxeter system $(W,S)$ is \textit{irreducible} if there is no partition of the generating set $S$ into disjoint nonempty sets $S_1$ and $S_2$ such that $s_1s_2=s_2s_1$ for all $s_1\in S_1$ and all $s_2\in S_2$. We will always assume that $(W,S)$ is irreducible. 

\subsection{Fuchsian Coxeter groups}

We now define a special class of Coxeter groups that are discrete subgroups of~$PGL_2(\mathbb{R})$, called \textit{Fuchsian Coxeter groups}. 
Let $n\geq 3$ be an integer, and let $k_1,\ldots,k_n\geq 2$ be integers satisfying 
\begin{align}\label{eq:hyperbolic}
\sum_{i=1}^n\frac{1}{k_i}<n-2.
\end{align}
Assign the angles $\pi/k_i$ to the vertices of a combinatorial $n$-gon~$F$. There is a convex realisation of $F$ (which we also call~$F$) in the hyperbolic disc~$\mathbb{H}^2$, and the subgroup of $PGL_2(\mathbb{R})$ generated by the reflections in the sides of~$F$ is a Coxeter group $(W,S)$ (see \cite[Example~6.5.3]{davis}). If $s_1,\ldots,s_{n}$ are the reflections in the sides of~$F$ (arranged cyclically), then the order of $s_is_j$ is 
\begin{align}\label{eq:hyperbolic2}
\begin{aligned}
m_{ij}=\begin{cases}k_i&\textrm{if $j=i+1$}\\
\infty&\textrm{if $|i-j|>1$},
\end{cases}
\end{aligned}
\end{align}
where the indices are read cyclically with $n+1\equiv 1$. 

A Coxeter system~$(W,S)$ given by data~(\ref{eq:hyperbolic})
and~(\ref{eq:hyperbolic2}) is called a \textit{Fuchsian Coxeter
  system}. Observe that these
systems are always infinite. The group $W$ acts on $\mathbb{H}^2$ with fundamental domain~$F$. Note that this action does not preserve orientation, however the index~$2$ subgroup~$W'$ generated by the even length elements of~$W$ is orientation preserving. Thus $W'$ is a discrete subgroup of $PSL_2(\mathbb{R})$, and so is a `Fuchsian group' in the strictest sense of the expression.

The Fuchsian Coxeter system $(W,S)$ induces a tessellation of $\mathbb{H}^2$ by isometric polygons $wF$, $w\in W$. The polygons $wF$ are called \textit{chambers}, and we usually identify the set of chambers with~$W$ by $wF\leftrightarrow w$. We call this the \textit{hyperbolic realisation} of the Coxeter system~$(W,S)$ (it is closely related to the \textit{Davis complex} from~\cite{davis}, see the discussion in \cite[Example~12.43]{AB}).

\begin{example}\label{ex:21} (a) Let $a,b,c\geq 2$ be integers, and let $W_{abc}$ be the group generated by $S=\{s,t,u\}$ subject to the relations
$$
s^2=t^2=u^2=1\quad\textrm{and}\quad (st)^a=(tu)^b=(us)^c=1.
$$
These Coxeter groups are called \textit{triangle groups}, for they can be realised as groups generated by the reflections in the sides of a triangle on the sphere $\SS^2$ (when $\frac{1}{a}+\frac{1}{b}+\frac{1}{c}>1$), the Euclidean plane $\RR^2$ (when $\frac{1}{a}+\frac{1}{b}+\frac{1}{c}=1$), or the hyperbolic disc $\HH^2$ (when $\frac{1}{a}+\frac{1}{b}+\frac{1}{c}<1$). In the latter case the Coxeter group is Fuchsian. Up to permutation of the triple $(a,b,c)$, the irreducible spherical triangle groups are given by $(a,b,c)=(3,3,2),(4,3,2),(5,3,2)$, and the Euclidean triangle groups are given by $(a,b,c)=(3,3,3),(4,4,2),(6,3,2)$.

(b) Let $k_i=2$ for each~$1\leq i\leq n$ in~(\ref{eq:hyperbolic}). Thus each internal angle of the $n$-gon~$F$ is a right angle, and the corresponding Coxeter group is called a \textit{right angled polygon group} (by (\ref{eq:hyperbolic}) this group is Fuchsian if and only if~$n\geq 5$). 
\end{example}
\begin{figure}[!ht]
\centering
\subfigure[Affine triangle group $(3,3,3)$]{
\includegraphics[totalheight=4.85cm]{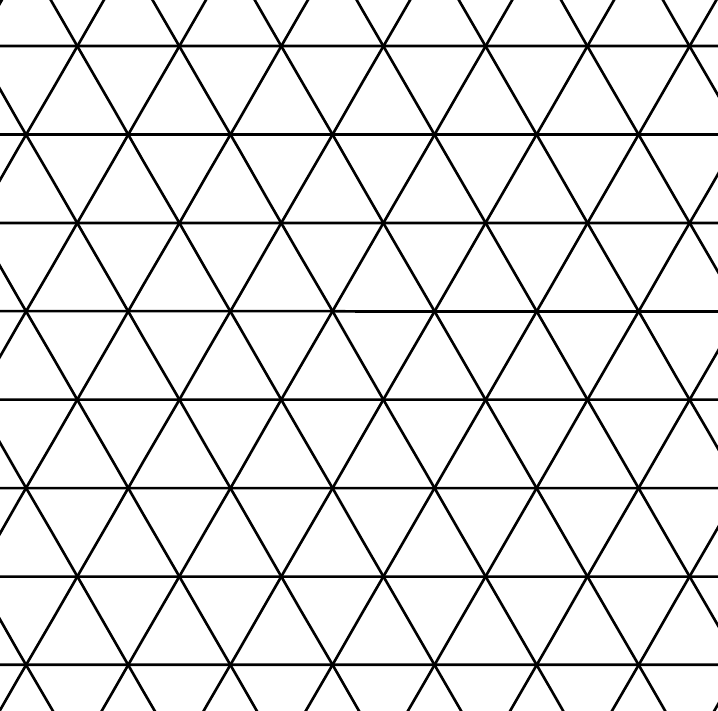}}\quad
\subfigure[\mbox{Fuchsian triangle group $(3,3,4)$}]{
\includegraphics[totalheight=4.85cm]{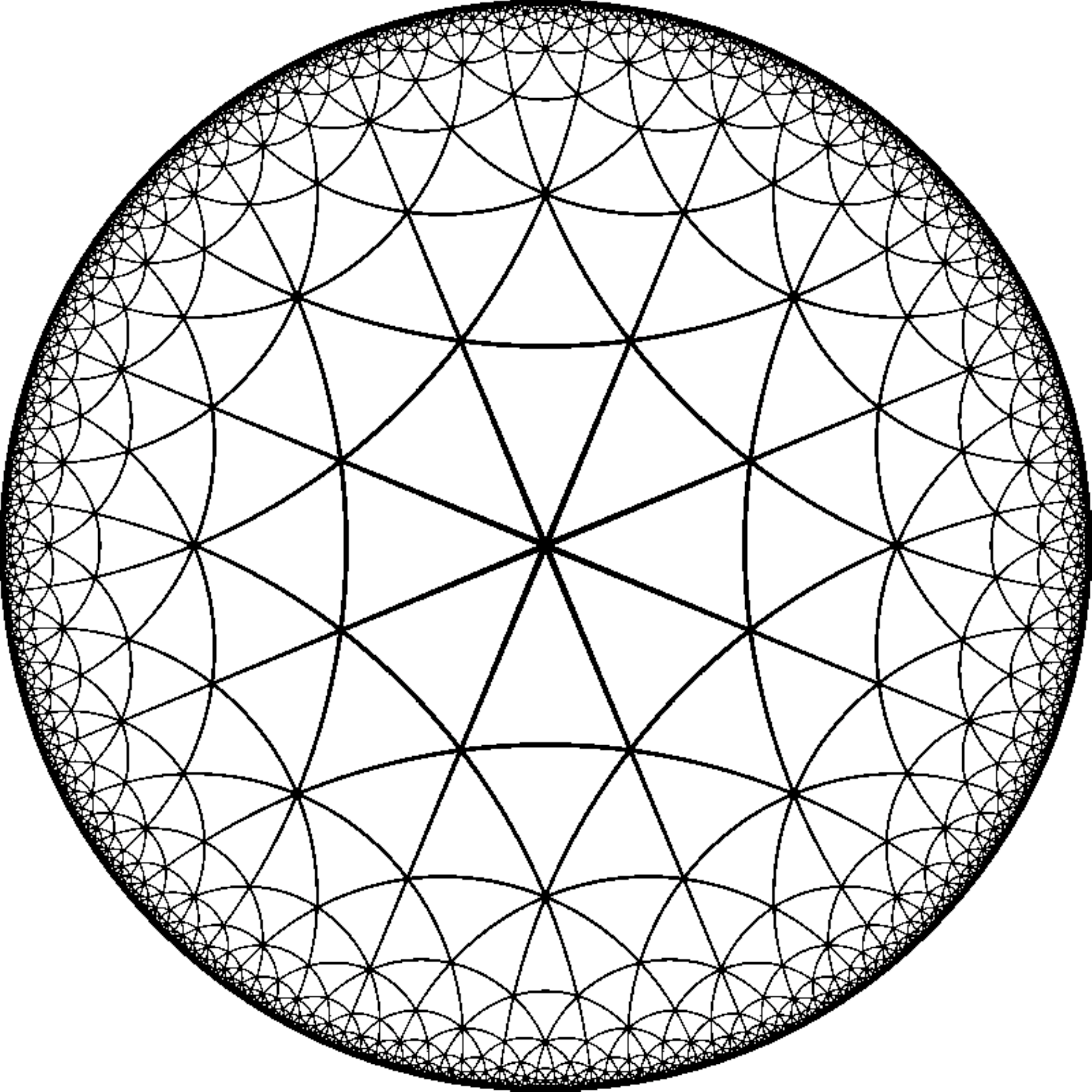}}\quad
\subfigure[Right angled polygon group]{
\includegraphics[totalheight=4.85cm]{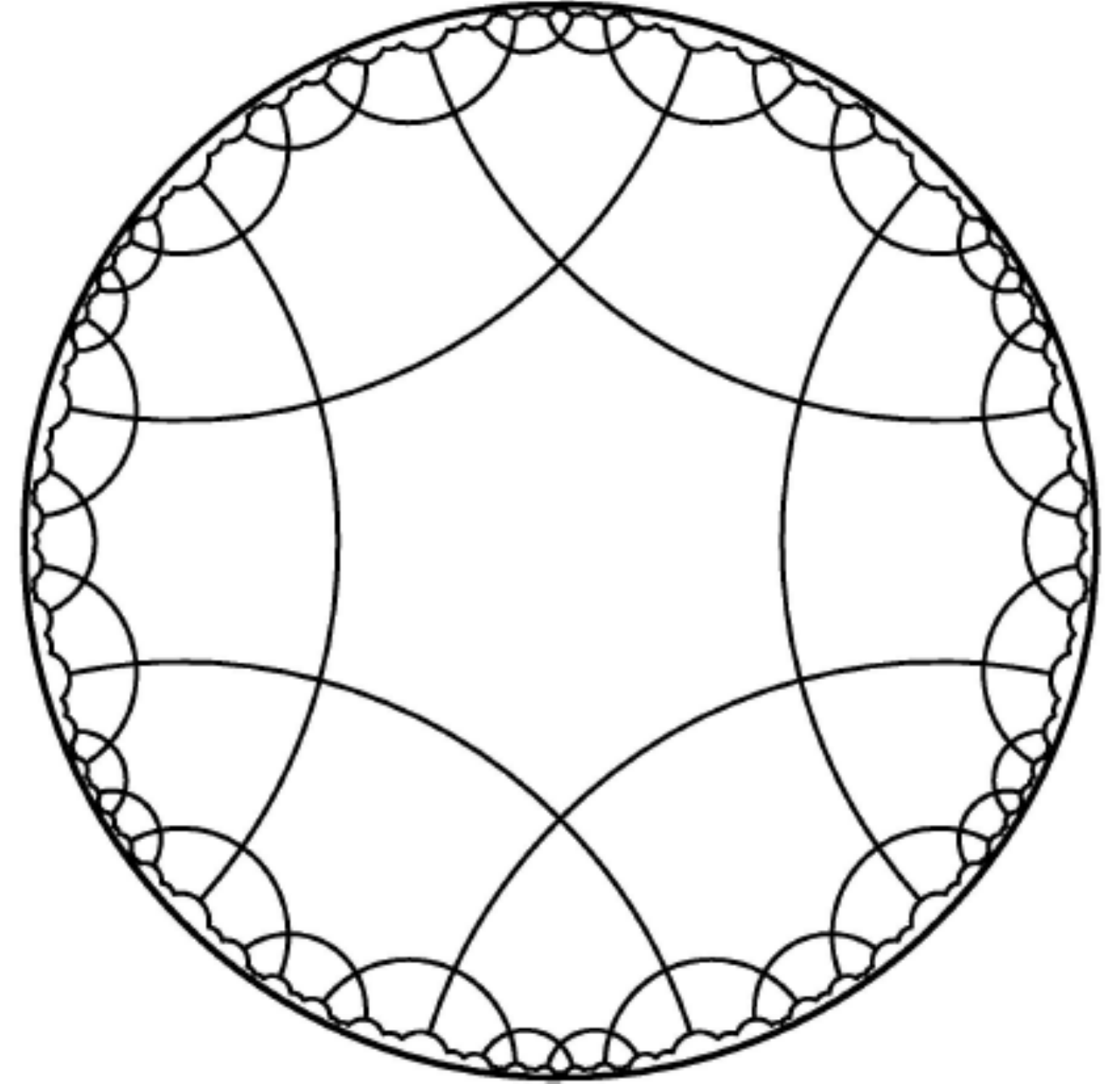}}
\caption{Triangle and polygon groups (pictures adapted from wikipedia)}\label{fig:triangle}
\end{figure}

\subsection{Definition of buildings}

We now give an axiomatic definition of buildings, following~\cite{AB}.

\begin{defn}\label{defn:building}
Let $(W,S)$ be a Coxeter system. A \textit{building of type $(W,S)$} is a pair $(\Delta,\delta)$ where~$\Delta$ is a nonempty set (whose elements are called \textit{chambers}) and $\delta:\Delta\times\Delta\to W$ is a function (called the \textit{Weyl distance function}) such that if $x,y\in\Delta$ then the following conditions hold:
\begin{enumerate}
\item[(B1)] $\delta(x,y)=1$ if and only if $x=y$.
\item[(B2)] If $\delta(x,y)=w$ and $z\in\Delta$ satisfies $\delta(y,z)=s$ with $s\in S$, then $\delta(x,z)\in\{w,ws\}$. If, in addition, $\ell(ws)=\ell(w)+1$, then $\delta(x,z)=ws$.
\item[(B3)] If $\delta(x,y)=w$ and $s\in S$, then there is a chamber $z\in\Delta$ with $\delta(y,z)=s$ and $\delta(x,z)=ws$. 
\end{enumerate}
\end{defn}

Let $(\Delta,\delta)$ be a building of type $(W,S)$ and let $s\in S$. Chambers
$x,y\in\Delta$ are \textit{$s$-adjacent} (written $x\sim_s y$) if
$\delta(x,y)=s$. One useful way to visualise a building is to imagine an
$|S|$-gon with edges labelled by the generators $s\in S$ (think of the edges as
being coloured by $|S|$ different colours). Call this $|S|$-gon the
\textit{base chamber} which we denote by $o$. Now take one copy of the base chamber for each element $x\in\Delta$, and glue these chambers together along edges so that $x\sim_s y$ if and only if the chambers are glued together along their $s$-edges.

A \textit{gallery of type $(s_1,\ldots,s_n)$} joining $x\in\Delta$ to $y\in\Delta$ is a sequence $x_0,x_1,\ldots,x_n$ of chambers with
$$
x=x_0\sim_{s_1}x_1\sim_{s_2}\cdots\sim_{s_n}x_n=y. 
$$
This gallery has \textit{length $n$}. 

The Weyl distance function $\delta$ has a useful description in terms of minimal length galleries in the building: If $s_1\cdots s_n$ is a reduced expression in~$W$ then $\delta(x,y)=s_1\cdots s_n$ if and only if there is a minimal length gallery in $\Delta$ from $x$ to $y$ of type $(s_1,\ldots,s_n)$. The \textit{(numerical) distance} between chambers $x,y\in\Delta$ is
$$
d(x,y)=(\textrm{length of a minimal length gallery joining $x$ to $y$})=\ell(\delta(x,y)),
$$
Note that we use the same notation $d(\cdot,\cdot)$ for distance in both the Coxeter system and the building.

A building $(\Delta,\delta)$ is called \textit{thick} if $|\{y\in\Delta\mid x\sim_s y\}|\geq 2$ for all chambers $x\in\Delta$, and \textit{thin} if $|\{y\in\Delta\mid x\sim_s y\}|=1$ for all chambers $x\in\Delta$. A building $(\Delta,\delta)$ is \textit{regular} if 
$$
q_s:=|\{y\in\Delta\mid x\sim_s y\}|\quad\text{is finite and does not depend on $x\in\Delta$}.
$$
For the remainder of this paper we will assume that $(\Delta,\delta)$ is regular. The numbers $(q_s)_{s\in S}$ are called the \textit{thickness parameters} of the building.

If $I\subseteq S$ and $x\in\Delta$ then the set $R_I(x)=\{y\in\Delta\mid \delta(x,y)\in W_I\}$ (called the \textit{$I$-residue of $x$}) is a building of type~$(W_I,I)$ with thickness parameters $(q_s)_{s\in I}$ (see \cite[Corollary~5.30]{AB}).

For each $x\in\Delta$ and each $w\in W$, let
$$
\Delta_w(x)=\{y\in\Delta\mid \delta(x,y)=w\}\quad\text{be the ``sphere of radius $w$'' centred at $x$.}
$$
 By \cite[Proposition~2.1]{P1} the cardinality $q_w=|\Delta_w(x)|$ does not depend on $x\in \Delta$, and is given by
$$
q_w=q_{s_1}\cdots q_{s_{\ell}}\quad\textrm{whenever $w=s_1\cdots s_{\ell}$ is a reduced expression.}
$$

We call $(\Delta,\delta)$ a \textit{Fuchsian building} if $(W,S)$ is a Fuchsian Coxeter system, and we call $(\Delta,\delta)$ a \textit{triangle building} if $W$ is an infinite triangle group.

Finally, a word about notation. Typically the letters $u,v,w$ will be used for elements of a Coxeter group~$W$, and the letters $x,y,z$ will be used for chambers of a building~$(\Delta,\delta)$. 

\subsection{Examples of buildings}

We now give some examples of buildings that are relevant to this paper. We also show that the class of locally finite thick Fuchsian buildings is sufficiently rich by proving existence of many such buildings.

\begin{example}[Thin buildings]
Let $(W,S)$ be a Coxeter system. Let $\Delta=W$, and define $\delta:\Delta\times \Delta\to W$ by $\delta(u,v)=u^{-1}v$. It is immediate that $(\Delta,\delta)$ is a building of type $(W,S)$. This rather simple example is called the~\textit{Coxeter complex} of~$(W,S)$. It is a thin building, because $\{v\in W\mid u\sim_s v\}=\{us\}$ for each $u\in W$. Conversely it is not difficult to see that every thin building is isomorphic to a Coxeter complex. 
\end{example}

\begin{example}[Generalised polygons]
If $(W,S)$ is a dihedral group of order $2m$ (that is, $S=\{s,t\}$ with $s^2=t^2=(st)^m=1$) then buildings of type $(W,S)$ are called \textit{generalised $m$-gons}. These `basic building blocks' play an important role in the theory (see the monograph~\cite{HVM} which is devoted to the study of generalised $m$-gons). The Feit-Higman Theorem~\cite{feithigman} implies that locally finite thick generalised $m$-gons only exist for $m\in\{2,3,4,6,8,\infty\}$. 
\end{example}

If $(\Delta,\delta)$ is a locally finite thick regular building of general type $(W,S)$, then the ``rank~$2$'' residues $R_{st}(x)=R_{\{s,t\}}(x)$ are generalised $m_{st}$-gons, and so necessarily 
\begin{align}\label{eq:feithigman}
m_{st}\in\{2,3,4,6,8,\infty\}\quad\text{for all $s,t\in S$ with $s\neq t$}.
\end{align}
A sufficient condition for the existence of a locally finite thick regular building of type~$(W,S)$ is that $m_{st}\in\{2,3,4,6,\infty\}$ for all $s,t\in S$ (see Example~\ref{ex:kac}). Allowing $m_{st}=8$ introduces some complications, see Proposition~\ref{prop:existence}). 

\begin{example}[Buildings from groups with $BN$-pairs]\label{ex:kac} 
Let $G$ be a group with a $BN$-pair $(B,N)$ and Coxeter system~$(W,S)$ (see \cite[\S~6.2.6]{AB} for the definition of $BN$-pairs). An instructive example is $G=GL_n(\mathbb{F})$ where $\mathbb{F}$ is a field, with $B$ the upper triangular invertible matrices, $N$ the monomial matrices (matrices with exactly one nonzero entry in each row and column) and $W=N/(N\cap B)$ the symmetric group on $n$ letters (represented as permutation matrices) with $S$ being the elementary transpositions. 

The group $G$ admits a \textit{Bruhat decomposition}
$G=\bigsqcup_{w\in W}BwB$.
Let $\Delta=G/B$, and define $\delta:\Delta\times\Delta\to W$ by
\begin{align*}
\delta(gB,hB)=w\quad\textrm{if and only if}\quad g^{-1}hB\subseteq BwB.
\end{align*}
Then $(\Delta,\delta)$ is a thick building of type $(W,S)$ (see \cite[Theorem~6.56]{AB}).

All \textit{groups of Lie type} (classical groups, Chevalley groups, Steinberg groups, Suzuki-Ree groups) admit a $BN$-pair. More generally, every ``Kac-Moody group'' admits a $BN$-pair. A \textit{Kac-Moody algebra} (cf. \cite{kac}) is a generalisation of the more familiar semisimple Lie algebras. These algebras share many properties with their finite dimensional counterparts, for example, Cartan subalgebras, root space decompositions, and Weyl groups. However in contrast to the semisimple Lie algebra case, the root systems and Weyl groups for infinite dimensional Kac-Moody algebras are infinite. There are Kac-Moody algebras associated to each \textit{crystallographic} Coxeter system (that is, $m_{st}\in\{2,3,4,6,\infty\}$ for all $s,t\in S$). To each such algebra, and for each choice of ground field~$\mathbb{F}$, one can define a \textit{Kac-Moody group} $G=G(\mathbb{F})$ by generators and relations in an analogous way to the construction of Chevalley groups in the finite dimensional setting (see \cite{steinberg} for the finite dimensional theory, and~\cite{titskac} for the Kac-Moody case). The group $G$ has a $BN$-pair, with Coxeter system $(W,S)$. The associated building $(G/B,\delta)$ has uniform thickness parameter $|\mathbb{F}|$, and so taking $\mathbb{F}=\mathbb{F}_q$ to be the finite field with $q$ elements yields a regular building of type~$(W,S)$ with thickness~$q$.
\end{example}

\begin{example}[Ronan's free construction]\label{ex:ronan}
Suppose that $(W,S)$ is a Coxeter system such that every irreducible rank~$3$ parabolic subgroup is infinite. Suppose that $(q_s)_{s\in S}$ is a sequence of integers, and that for each pair $s,t\in S$ with $s\neq t$ there exists a generalised $m_{st}$-gon $\Gamma_{st}$ with parameters $(q_s,q_t)$. Then Ronan's free construction~\cite{ronanconstruction} implies that there exists a locally finite thick regular building $(\Delta,\delta)$ of type $(W,S)$ with thickness parameters $(q_s)_{s\in S}$. 
\end{example}

It is obvious that every irreducible rank~$3$ parabolic subgroup of a Fuchsian Coxeter system~$(W,S)$ is infinite, and thus Ronan's free construction applies to Fuchsian buildings. Thus to exhibit the existence of a thick regular Fuchsian building~$(\Delta,\delta)$ of type $(W,S)$ with thickness parameters $(q_s)_{s\in S}$ it is sufficient to exhibit the existence of a family $\{\Gamma_{st}\mid s,t\in S, s\neq t\}$ of generalised $m_{st}$-gons $\Gamma_{st}$ with thickness parameters~$(q_s,q_t)$. In the following proposition we use this idea to classify those infinite triangle Coxeter systems admitting locally finite thick regular buildings.  This is elementary, although we have been unable to find a reference in the literature.

\begin{prop}\label{prop:existence} Let $(W,S)$ be an infinite triangle Coxeter system, with the generators $s,t,u$ arranged so that $m_{st}\geq m_{tu}\geq m_{us}$. A locally finite thick triangle building of type $(W,S)$ exists if and only if
$$
(m_{st},m_{tu},m_{us})\in\{(a,b,c)\mid a,b,c\in\{2,3,4,6,8\}\}\setminus \{(8,3,3),(8,6,3),(8,6,6),(8,8,8)\}
$$
Thus there are precisely~$24$ infinite triangle Coxeter systems (up to permuting the generators) admitting locally finite thick triangle buildings. Moreover, for each of these infinite triangle Coxeter systems $(W,S)$ there are infinitely many pairwise nonisomorphic buildings of type $(W,S)$.
\end{prop}
\begin{proof} Suppose that a locally finite thick regular building $(\Delta,\delta)$ of type $(W,S)$ exists. By (\ref{eq:feithigman}) we have $m_{st}\in\{2,3,4,6,8,\infty\}$, and the case $m_{st}=\infty$ is excluded for triangle groups by definition. Since infinite triangle groups have $m_{st}^{-1}+m_{tu}^{-1}+m_{us}^{-1}\leq 1$ this leaves precisely $28$ infinite triangle groups with $m_{st}\geq m_{tu}\geq m_{us}$ and $m_{st},m_{tu},m_{us}\in\{2,3,4,6,8\}$. We now show that the four cases $(m_{st},m_{tu},m_{us})=(8,3,3)$, $(8,6,3)$, $(8,6,6)$, or $(8,8,8)$ do not admit locally finite thick buildings. We recall from \cite[\S1.7]{HVM} that in a finite thick generalised $m$-gon with parameters $(q,q')$ we necessarily have that $q=q'$ if $m=3$, $\sqrt{qq'}\in\mathbb{Z}$ if $m=6$, and $\sqrt{2qq'}\in\mathbb{Z}$ if $m=8$. For example, consider the $(m_{st},m_{tu},m_{us})=(8,6,6)$ case. If a locally finite thick building with parameters $q_s,q_t,q_u$ exists, then $R_{st}(o)$ is a generalised $8$-gon with parameters $(q_s,q_t)$, and $R_{tu}(o)$ and $R_{us}(o)$ are generalised $6$-gons with parameters $(q_t,q_u)$ and $(q_u,q_s)$ (respectively). This implies that $\sqrt{2q_sq_t}\in\mathbb{Z}$, and $\sqrt{q_tq_u},\sqrt{q_uq_s}\in\mathbb{Z}$, a contradiction. The remaining cases are similar.

We now show that there exist locally finite thick regular buildings for each of the remaining~$24$ infinite triangle Coxeter systems, and moreover, that for each of these triangle Coxeter systems there are infinitely many buildings. For this we recall some known examples of generalised $m$-gons (see \cite{HVM} for details). If $m=2,3,4$ or $6$ then there is a generalised $m$-gon with parameters $(q,q)$ for each prime power~$q$. Thus if $m_{st},m_{tu},m_{us}\in\{2,3,4,6\}$ we can take generalised $m$-gons with parameters~$(q,q)$ as the basic building blocks, verifying the claim in this case. The cases where at least one of the $m$'s is $8$ require a little more care. We recall that there are generalised $4$-gons with parameters $(q,q^2)$ and $(q^2,q)$ for each prime power~$q$, and that for each $r=2^{2k+1}$ there are generalised $8$-gons with parameters $(r,r^2)$ and $(r^2,r)$ (in fact, these are the only known examples of finite thick generalised $8$-gons). For example, consider the $(8,6,4)$ triangle group. For each $r=2^{2k+1}$ there exists a generalised $8$-gon with parameters $(r^2,r)$, a generalised $6$-gon with parameters $(r,r)$, and a generalised $4$-gon with parameters $(r,r^2)$, and so there is a thick regular triangle building of type $(W,S)$ with parameters $(r^2,r,r)$. By varying~$k$ we obtain infinitely many buildings (pairwise non-isomorphic because they have different thicknesses). The remaining examples are similar. 
\end{proof}

Similar ideas show that there are infinitely many Fuchsian Coxeter systems $(W,S)$ with $|S|\geq 4$ for which locally finite thick regular buildings of type $(W,S)$ exist, and therefore the class of Fuchsian buildings is reassuringly rather large.

\subsection{Apartments and retractions}\label{sect:aptret}

Let $(\Delta,\delta)$ be a building of type $(W,S)$. The thin sub-buildings of
$(\Delta,\delta)$ of type $(W,S)$ are called the \textit{apartments} of $(\Delta,\delta)$. Thus each apartment is isomorphic to the Coxeter complex of~$(W,S)$. Two key facts concerning apartments are as follows:
\begin{enumerate}
\item[(A1)] If $x,y\in\Delta$ then there is an apartment $A$ containing both $x$ and $y$.
\item[(A2)] If $A$ and $A'$ are apartments containing a common chamber $x$ then there is a unique isomorphism $\theta:A'\to A$ fixing each chamber of the intersection $A\cap A'$.
\end{enumerate}
In fact conditions~(A1) and (A2) can be taken as an alternative, equivalent definition of buildings (see \cite[Definition~4.1]{AB} for the precise statement, and \cite[Theorem~5.91]{AB} for the equivalence of the two axiomatic systems).

Given chambers $x,y\in\Delta$, the \textit{convex hull} $[x,y]$ of $x$ and $y$ is the union of all chambers on minimal length galleries from $x$ to~$y$. That is,
$
[x,y]=\{z\in\Delta\mid d(x,y)=d(x,z)+d(z,y)\}.
$
Another useful fact about apartments is:
\begin{enumerate}
\item[(A3)] If $A$ is an apartment containing $x$ and $y$ then $[x,y]\subseteq A$.
\end{enumerate}
In fact, if $\Delta$ is thick then $[x,y]$ is the intersection of all apartments $A$ containing~$x$ and~$y$.

The hyperbolic realisation of each apartment of a Fuchsian building is a tesselation of the hyperbolic disc, as in Figure~\ref{fig:triangle}(b) and~(c). Roughly speaking, the properties (A1) and (A2) ensure that the hyperbolic metric on each apartment can be coherently `glued together' to make $(\Delta,\delta)$ a $\mathrm{CAT}(-1)$ space (see \cite[Theorem~18.3.9]{davis} for details). 

\textit{Retractions} play an important role in building theory, and indeed in this current work. Let $A$ be an apartment, and let $x$ be a chamber of $A$. The \textit{retraction $\rho_{A,x}$ of $\Delta$ onto $A$ with centre~$x$} is defined as follows: For each chamber $y\in\Delta$,
$$
\rho_{A,x}(y)=z,\quad\textrm{where $z$ is the unique chamber of $A$ with $\delta(x,z)=\delta(x,y)$}.
$$
Alternatively, let $A'$ be any apartment containing $x$ and $y$ (using (A1)) and let $\theta:A'\to A$ be the isomorphism from (A2) fixing $A\cap A'$. Then
$$
\rho_{A,x}(y)=\theta(y).
$$
Thus $\rho_{A,x}:\Delta\to A$ ``radially flattens'' the building onto $A$, with centre~$x\in A$. 

Fix, once and for all, an apartment $A_0$ and a chamber $o\in A_0$. Canonically identify $A_0$ with the Coxeter complex of $(W,S)$ such that $o$ is identified with $1$, the neutral element of~$W$. Thus we regard $W=A_0$ as a ``base apartment'' of~$\Delta$. To simplify notation, we write
$
\rho=\rho_{W,o}
$ for the retraction of $\Delta$ onto the apartment $W$ with centre~$o$. Thus
\begin{align}\label{eq:canonicalretraction}
\rho:\Delta\to W\quad\textrm{is given by}\quad \rho(x)=\delta(o,x).
\end{align}
We also note that in the apartment $A_0=W$ the Weyl distance function is given by
$$
\delta(u,v)=u^{-1}v\quad\text{for all $u,v$ in the base apartment~$W$}.
$$

\section{Automata for Coxeter groups and buildings}\label{sect:3}

The notions of \textit{cones}, \textit{cone types}, and \textit{automata} are well established for finitely generated groups, with \cite{cannon} being a standard reference. Let us briefly recall these notions in the specific context of Coxeter groups, and then extend the ideas into the (non-group) realm of buildings.  

Let $(W,S)$ be a Coxeter system. Let $w\in W$. The \textit{cone of $(W,S)$ with root $w$} is the set
$$
C_W(w)=\{v\in W\mid d(1,v)=d(1,w)+d(w,v)\}.
$$
Thus $C_W(w)$ is the set of all elements $v\in W$ such that there exists a geodesic from $1$ to $v$ passing through $w$. The \textit{cone type} of the cone $C_W(w)$ is
$$
T_W(w)=\{w^{-1}v\mid v\in C(w)\}=w^{-1}C_W(w).
$$

Let $\mathcal{T}(W,S)$ be the set of cone types of $(W,S)$. By \cite[Theorem~2.8]{howlett} there are only finitely many cone types in a Coxeter system $(W,S)$, and so $|\mathcal{T}(W,S)|<\infty$. 

\begin{defn}\label{defn:cannon} The \textit{Cannon automaton} of the Coxeter system $(W,S)$ is the directed graph $\mathcal{A}(W,S)$ with vertex set $\mathcal{T}(W,S)$ and with labelled edges defined as follows. There is a directed edge with label $s\in S$ from cone type $\bT$ to cone type $\bT'$ if and only if there exists $w\in W$ such that $\bT=T_W(w)$ and $\bT'=T_W(ws)$ and $d(1,ws)=d(1,w)+1$.
\end{defn}

A cone type $\bT'$ is \textit{accessible} from the cone type $\bT$ if there is a path from $\bT$ to $\bT'$ in the (directed) graph $\cA(W,S)$. In this case we 
write $\bT\to \bT'$. A cone type $\bT$ is called \textit{recurrent} if $\bT\to\bT$, and otherwise it is called \textit{transient}. The set of 
recurrent vertices induces a (directed) subgraph $\cA_R(W,S)$ of $\cA(W,S)$.
We call the automaton $\cA(W,S)$ \textit{strongly connected} if each recurrent cone type is accessible from any other recurrent cone type in the subgraph~$\cA_R(W,S)$.  

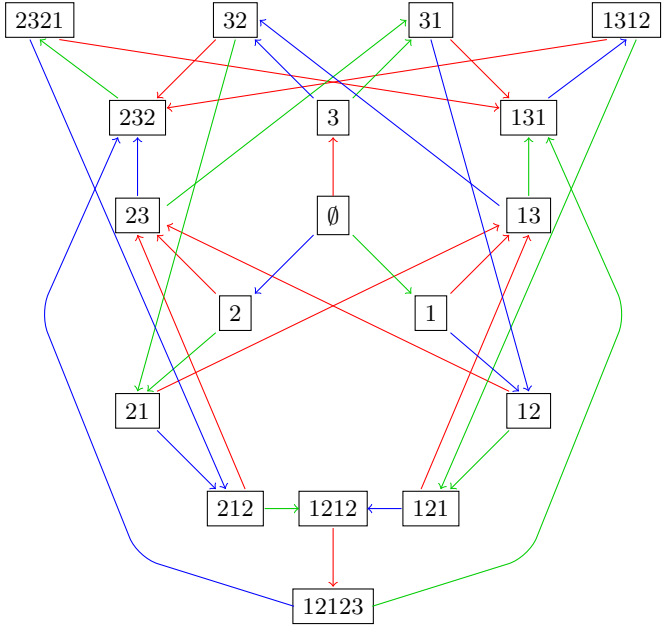
\begin{figure}[!h]
\begin{minipage}{0.6\textwidth}
\centering
\begin{tikzpicture} [scale=1.3]
\path (0,0) node [shape=rectangle,draw] {\footnotesize{$\emptyset$}}
         (1,-1) node [shape=rectangle,draw] {\footnotesize{$1$}}
         (2,-2) node [shape=rectangle,draw] {\footnotesize{$12$}}
         (1,-3) node [shape=rectangle,draw] {\footnotesize{$121$}}
         (0,-3) node [shape=rectangle,draw] {\footnotesize{$1212$}}
         (-1,-3) node [shape=rectangle,draw] {\footnotesize{$212$}}
         (-2,-2) node [shape=rectangle,draw] {\footnotesize{$21$}}
         (-1,-1) node [shape=rectangle,draw] {\footnotesize{$2$}}
         (2,0) node [shape=rectangle,draw] {\footnotesize{$13$}}
         (2,1) node [shape=rectangle,draw] {\footnotesize{$131$}}
         (1,2) node [shape=rectangle,draw] {\footnotesize{$31$}}
         (3,2) node [shape=rectangle,draw] {\footnotesize{$1312$}}
         (0,1) node [shape=rectangle,draw] {\footnotesize{$3$}}
         (-1,2) node [shape=rectangle,draw] {\footnotesize{$32$}}
         (-2,1) node [shape=rectangle,draw] {\footnotesize{$232$}}
         (-2,0) node [shape=rectangle,draw] {\footnotesize{$23$}}
         (-3,2) node [shape=rectangle,draw] {\footnotesize{$2321$}}
         (0,-4) node [shape=rectangle,draw] {\footnotesize{$12123$}};
\draw [->,red] (0,-3.2) -- (0,-3.8);
\draw [->,red] (-1.2,1.8) -- (-1.8,1.2);
\draw [->,red] (1.2,1.8) -- (1.8,1.2);
\draw [->,red] (-1.2,-0.8) -- (-1.8,-0.2);
\draw [->,red] (1.2,-0.8) -- (1.8,-0.2);
\draw [->,red] (-1.8,-1.8) -- (1.7,-0.1);
\draw [->,red] (1.8,-1.8) -- (-1.7,-0.1);
\draw [->,red] (-0.9,-2.8) -- (-2,-0.2);
\draw [->,red] (0.9,-2.8) -- (2,-0.2);
\draw [->,red] (0,0.2) -- (0,0.8);
\draw [->,red] (-2.8,1.8) -- (1.7,1.1);
\draw [->,red] (2.8,1.8) -- (-1.7,1.1);
\draw [->,Green] (0.2,-0.2) -- (0.8,-0.8);
\draw [->,Green] (-1.2,-1.2) -- (-1.9,-1.8);
\draw [->,Green] (-0.7,-3) -- (-0.35,-3);
\draw [->,Green] (1.8,-2.2) -- (1.2,-2.8);
\draw [->,Green] (0.2,1.2) -- (0.8,1.8);
\draw [->,Green] (2,0.2) -- (2,0.8);
\draw [->,Green] (-2.2,1.2) -- (-3,1.8);
\draw [->,Green] (-1,1.8) -- (-2,-1.8);
\draw [->,Green] (3.1,1.8) -- (1.1,-2.8);
\draw [->,Green, rounded corners=8pt] (0.4,-4) -- (2,-3.5) -- (3,-1) -- (2.2,0.8); 
\draw [->,Green] (-1.7,0.1) -- (0.75,2);
\draw [->,blue] (-0.2,-0.2) -- (-0.8,-0.8);
\draw [->,blue] (1.2,-1.2) -- (1.9,-1.8);
\draw [->,blue] (0.7,-3) -- (0.35,-3);
\draw [->,blue] (-1.8,-2.2) -- (-1.2,-2.8);
\draw [->,blue] (-0.2,1.2) -- (-0.8,1.8);
\draw [->,blue] (-2,0.2) -- (-2,0.8);
\draw [->,blue] (2.2,1.2) -- (3,1.8);
\draw [->,blue] (1,1.8) -- (2,-1.8);
\draw [->,blue] (-3.1,1.8) -- (-1.1,-2.8);
\draw [->,blue, rounded corners=8pt] (-0.4,-4) -- (-2,-3.5) -- (-3,-1) -- (-2.2,0.8); 
\draw [->,blue] (1.7,0.1) -- (-0.75,2);
\end{tikzpicture}
\caption{The Cannon automaton for~$W_{(3,3,4)}$}\label{fig:automata}
\end{minipage}
\begin{minipage}{0.4\textwidth}
Figure~\ref{fig:automata} shows the Cannon automaton for the (Fuchsian) triangle group $W_{(3,3,4)}$ (see Appendix~\ref{app:A} for details). The generators are labelled $1$, $2$, and~$3$, and the labels on the edges are indicated by colours (green, blue, red respectively). The cone types are given by the base element of a representative cone of that type. Thus the vertex $131$ is the cone type $T(131)$. All cone types, except for $\emptyset$, $1$, $2$, and $3$, are recurrent. This automaton is strongly connected. For example, the sequence $121\to 1212\to 12123\to 232\to 2321\to 212\to 23$ shows that $121\to 23$.  
\end{minipage}
\end{figure}

The existence of a strongly connected Cannon automaton is important for our renewal theory arguments in Section~\ref{sect:renewal}, thus in Appendix~\ref{app:A} we prove:

\begin{thm}\label{thm:stronglyconnected} The 
Cannon automaton of a Fuchsian Coxeter system is strongly connected. 
\end{thm}

\begin{remark}
It does not appear to be known in the literature which Coxeter systems have strongly connected automata. For example, our direct calculations in Appendix~\ref{app:A} show that affine triangle groups do not have strongly connected Cannon automata, and we suspect that no affine Coxeter group has a strongly connected Cannon automaton. 
\end{remark}

We now extend the above concepts to buildings. Let $(\Delta,\delta)$ be a building of type $(W,S)$ with fixed base chamber $o$. Let $x\in \Delta$ be a chamber. The \textit{cone of $(\Delta,\delta)$ with root $x$} is the set
$$
C_{\Delta}(x)=\{y\in \Delta\mid d(o,y)=d(o,x)+d(x,y)\}.
$$
Thus $C_{\Delta}(x)$ is the set of all chambers $y\in \Delta$ such that there exists a geodesic from $o$ to $y$ passing through $x$. The \textit{cone type} of the cone $C_{\Delta}(x)$ is
$$
T_{\Delta}(x)=\{\delta(x,y)\mid y\in C_{\Delta}(x)\}.
$$

If $A$ is an apartment of $\Delta$ containing $o$ and $x\in A$ we write
$$
C_A(x)=\{y\in A\mid d(o,y)=d(o,x)+d(x,y)\}.
$$
We collect together some useful facts about cones and cone types in buildings, and the connection with cones and cone types in Coxeter systems. Recall the definition of the canonical retraction $\rho:\Delta\to W$ from~(\ref{eq:canonicalretraction}). 

\begin{prop}\label{prop:conetypes}
Let $(\Delta,\delta)$ be a building of type $(W,S)$. 
\begin{enumerate} 
\item If $A$ is an apartment containing the chambers $o$ and $x$ then the isomorphism $\rho|_A:A\to W$ maps $C_A(x)$ onto $C_W(\rho(x))$.
\item $\rho(C_{\Delta}(x))=C_W(\rho(x))$ for all $x\in \Delta$.
\item $T_{\Delta}(x)=T_W(\rho(x))$ for all $x\in\Delta$.
\item $\rho^{-1}(C_W(w))=\bigsqcup_{x\in\Delta_w(o)}C_{\Delta}(x)$ for all $w\in W$.
\end{enumerate}
\end{prop}

\begin{proof}
If $A$ is an apartment containing $o$ and $x$ then the restriction
$\rho|_A:A\to W$ is an isomorphism. Thus $\rho|_A$ and $\rho|_A^{-1}$ map
minimal galleries to minimal galleries, and hence part~$1$ follows. 

Next we claim that $C_{\Delta}(x)=\bigcup_{A}C_A(x)$ where the union is over all apartments $A$ containing $o$ and~$x$. It is clear that $C_A(x)\subseteq C_{\Delta}(x)$ for each apartment $A$ containing $o$ and $x$, and thus $\bigcup_{A}C_A(x)\subseteq C_{\Delta}(x)$. On the other hand, suppose that $y\in C_{\Delta}(x)$. Let $A$ be an apartment containing $o$ and $y$. Then $A$ contains $x$ by~(A3), and so $y\in C_A(x)$, completing the proof of the claim. Part~$2$ follows using part~$1$, since
$
\rho(C_{\Delta}(x))=\bigcup_{A}\rho(C_A(x))=C_W(\rho(x)).
$

To prove part~$3$, note that by part~$2$,
\begin{align*}
T_W(\rho(x))&=\rho(x)^{-1}C_W(\rho(x))=\rho(x)^{-1}\rho(C_{\Delta}(x))=\rho(x)^{-1}\{\rho(y)\mid y\in C_{\Delta}(x)\}.
\end{align*}
If $y\in C_{\Delta}(x)$ then $\delta(o,y)=\delta(o,x)\delta(x,y)$. Thus $\rho(y)=\rho(x)\delta(x,y)$, and so
$
T_W(\rho(x))=T_{\Delta}(x)
$.

From part~$2$ it is immediate that $\rho^{-1}(C_W(w))=\bigcup_{x\in\Delta_w(o)}C_{\Delta}(x)$ for all $w\in W$. To see that the union is disjoint, suppose that $y\in C_{\Delta}(x)\cap C_{\Delta}(x')$ with $x,x'\in\Delta_w(o)$. Let $A$ be an apartment of $\Delta$ containing $o$ and $y$. Since $x$ and $x'$ are both on minimal galleries from $o$ to $y$, (A3) implies that $x,x'\in A$. Since $\rho|_A:A\to W$ is an isomorphism, and since $\rho(x)=w=\rho(x')$, we have $x=x'$. 
\end{proof}

We make a completely analogous definition to Definition~\ref{defn:cannon} for the Cannon automaton $\cA(\Delta,\delta)$ of a building~$(\Delta,\delta)$.

\begin{defn}\label{defn:cannon2} Let $(\Delta,\delta)$ be a building of type $(W,S)$. The \textit{Cannon automaton} of $(\Delta,\delta)$ is the directed graph $\mathcal{A}(\Delta,\delta)$ with vertex set $\mathcal{T}(\Delta,\delta)$ and with labelled edges defined as follows. There is a directed edge with label $s\in S$ from cone type $\bT$ to cone type $\bT'$ if and only if there exists $x\in \Delta$ and $y\in\Delta_s(x)$ such that $\bT=T_{\Delta}(x)$ and $\bT'=T_{\Delta}(y)$ and $d(o,y)=d(o,x)+1$.
\end{defn}

\begin{prop}
Let $(\Delta,\delta)$ be a building of type $(W,S)$. Then $\cA(\Delta,\delta)\cong\cA(W,S)$. 
\end{prop}

\begin{proof}
By Proposition~\ref{prop:conetypes} there is a bijection between the vertex sets of $\cA(\Delta,\delta)$ and $\cA(W,S)$, and it is elementary to check that this bijection preserves labelled oriented edges.
\end{proof}

For the remainder of this paper, when it is clear from context we will typically write $C(\cdot)$ and $T(\cdot)$ for cones and cone types in either Coxeter groups or buildings.

The \textit{boundary} of a cone $C$ of $(\Delta,\delta)$ is
$$
\partial C=\{y\in C\mid \text{ there exists $z\in \Delta\setminus C$ with $d(y,z)=1$}\}.
$$
If $L\geq 1$, the \textit{$L$-boundary} of a cone $C$ of $(\Delta,\delta)$ is defined to be
\begin{align}\label{eq:boundaryofcone}
\partial_L C=\{y\in C\mid \text{there exists $z\in \Delta\setminus C$ with $d(y,z)\leq L$}\}.
\end{align}
In particular, $\partial_1C=\partial C$. We call $\mathrm{Int}_{L} C=C\setminus \partial_{L} C$ the \textit{$L$-interior} of $C$.  We make analogous definitions for the boundary, $L$-boundary and $L$-interior of a cone $C$ of $(W,S)$. 

The \textit{$L$-boundary of a cone type} $\bT$ (of $\Delta$ or $W$) is defined  by
$$
\partial_L\bT=\{w\in\bT\mid\text{ there exists $v\in W\setminus\bT$ with $d(w,v)\leq L$}\},
$$
and the \textit{$L$-interior of the cone type} $\bT$ is $\mathrm{Int}_L \bT=\bT\setminus \partial_L\bT$.

\begin{lemma}\label{lem:onestep}
Let $x\in\Delta$ and $y\in C_{\Delta}(x)$. If there is a chamber $z\in\Delta$ with $d(y,z)=1$ and $z\notin C_{\Delta}(x)$ then there is a chamber $z'\in\Delta$ with $d(y,z')=1$ and $\rho(z')\notin C_W(\rho(x))$.  
\end{lemma}

\begin{proof}
Since $d(y,z)=1$ we have $\delta(y,z)=s$ for some $s\in S$. If $\ell(\delta(o,y)s)=\ell(\delta(o,y))+1$ then every minimal gallery from $o$ to $y$ can be extended to a minimal gallery from $o$ to $z$, and thus since $y\in C_{\Delta}(x)$ there is a minimal gallery from $o$ to $z$ passing through $x$, a contradiction. 

Thus $\ell(\delta(o,y)s)=\ell(\delta(o,y))-1$. Let $A$ be an apartment containing $o$ and $y$, and hence $A$ contains $x$ by~(A3). Let $z'\in A$ be the unique chamber of $A$ with $\delta(y,z')=s$. We claim that $\rho(z')\notin C_W(\rho(x))$. Suppose, for a contradiction, that $\rho(z')\in C_W(\rho(x))$. Then part 1 of Proposition~\ref{prop:conetypes} gives $z'\in C_A(x)$, and hence $z'\in C_{\Delta}(x)$. In particular $z'\neq z$ and so since both $z$ and $z'$ are $s$-adjacent to $y$ we have $\delta(z',z)=s$. 
Since $z'\in C_{\Delta}(x)$ there is a minimal gallery from $o$ to $z'$ passing through~$x$, and since $\delta(o,z')=\delta(o,y)s$ we have $\ell(\delta(o,z')s)=\ell(\delta(o,z'))+1$ and so we can extend this minimal gallery to give a minimal gallery from $o$ to $z$ passing through~$x$. Thus $z\in C_{\Delta}(x)$, a contradiction, and so $\rho(z')\notin C_W(\rho(x))$.
\end{proof}

\begin{prop}\label{prop:Lboundaries}
Let $(\Delta,\delta)$ be a building of type $(W,S)$. Then for each $x\in \Delta$ and each $L\geq 1$ we have 
\begin{equation*}
\rho(\partial_LC_{\Delta}(x))=\partial_LC_{W}(\rho(x)).
\end{equation*}
\end{prop}
\begin{proof}
Suppose that $v\in \partial_LC_W(\rho(x))$. Thus there is an element $v'\in W$ with $d(v,v')\leq L$ and $v'\notin C_W(\rho(x))$. Choose any apartment $A$ containing $o$ and~$x$. Let $y$ be the unique chamber of $A$ with $\delta(o,y)=v$, and let $y'$ be the unique chamber of~$A$ with $\delta(o,y')=v'$. Since $\rho(y)=v$ and $\rho(y')=v'$ and since $\rho|_A:A\to W$ is an isomorphism we have $d(y,y')=d(v,v')$. Moreover, from part~$1$ of Proposition~\ref{prop:conetypes} we have $y'\notin C_A(x)$ and it follows, using~(A3), that $y'\notin C_{\Delta}(x)$. Thus $y\in\partial_L C_{\Delta}(x)$ and so $v=\rho(y)\in\rho(\partial_LC_{\Delta}(x))$, giving $\partial_L C_W(\rho(x))\subseteq \rho(\partial_LC_{\Delta}(x))$.

Suppose that $y\in \partial_LC_{\Delta}(x)$, and so there is a chamber $z$ with $d(y,z)\leq L$ such that $z\notin C_{\Delta}(x)$. Choose this chamber~$z$ with $d(y,z)$ minimal, and let $y=y_0\sim y_1\sim\cdots\sim y_{k-1}\sim y_k=z$ be a minimal length gallery from~$y$ to $z$. By minimality of $d(y,z)$ we have that $y_{k-1}\in C_{\Delta}(x)$. Since $z\notin C_{\Delta}(x)$ Lemma~\ref{lem:onestep} implies that there is a chamber $z'$ adjacent to $y_{k-1}$ such that $\rho(z')\notin C_W(\rho(x))$. Since $d(\rho(y),\rho(z'))\leq d(y,z')=d(y,z)\leq L$ we have $\rho(y)\in\partial_LC_W(\rho(x))$, and hence $\rho(\partial_LC_{\Delta}(x))\subseteq \partial_LC_W(\rho(x))$. 
\end{proof}

\begin{remark}\label{rem:noniso} In the traditional setup of cones in groups, two cones with the same cone type are necessarily isomorphic since there is a group element taking one cone to the other. In the context of buildings the situation is quite different, for it follows from Ronan's free construction~\cite{ronanconstruction} of buildings with no rank~$3$ residues of spherical type that two cones in $\Delta$ of the same type are not necessarily isomorphic as graphs. In fact one can construct buildings in which there are \textit{infinitely many} pairwise non-isomorphic cones of a fixed type. However we note that Proposition~\ref{prop:conetypes} still guarantees that there will be only finitely many distinct cone types for the building.
\end{remark}




\section{Isotropic random walks on regular buildings}\label{sect:4}

In this section we investigate the structure of isotropic random walks in the general context of a regular building (not necessarily Fuchsian). 

\subsection{Definitions and transition operators}\label{subsec:prelim}

We will henceforth write $(\Delta,\delta)$ for a thick regular building of type $(W,S)$. A random walk $(X_n)_{n\geq 0}$ on the set $\Delta$ of chambers of the building $(\Delta,\delta)$ is \textit{isotropic} if the transition probabilities \mbox{$p(x,y)=\mathbb{P}[X_{n+1}=y\mid X_n=x]$} of the walk satisfy
$$
p(x,y)=p(x',y')\quad\textrm{whenever $\delta(x,y)=\delta(x',y')$}. 
$$
In other words, the probability of jumping from $x$ to $y$ in one step depends only on the Weyl distance $\delta(x,y)$. Thus an isotropic random walk is determined by the probabilities
\begin{align}\label{eq:transitionprob}
p_w=\mathbb{P}[X_1\in \Delta_w(x)\mid X_0=x],\quad\textrm{so that}\quad p(x,y)=p_w/q_{w}\quad\textrm{if $\delta(x,y)=w$},
\end{align}
and the transition operator of an isotropic random walk $(X_n)_{n\geq 0}$ on $\Delta$ with governing probabilities~(\ref{eq:transitionprob}) is given by
\begin{align}\label{eq:transitionop}
P=\sum_{w\in W}p_wP_w,
\end{align}
where for each $w\in W$, the operator $P_w$ acts on the space of all functions $f:\Delta\to\mathbb{C}$ by
$$
P_wf(x)=\frac{1}{q_w}\sum_{y\in\Delta_w(x)}f(y).
$$

For each $n\geq 0$ let
$$
p^{(n)}(x,y)=\mathbb{P}[X_n=y\mid X_0=x].
$$
Then $P^n=\sum_{w\in W}p_w^{(n)}P_w$, where $p^{(n)}(x,y)=p^{(n)}_w/q_w$ whenever $\delta(x,y)=w$. 

The random walk~$(X_n)_{n\geq 0}$ is \textit{irreducible} if for every pair $x,y\in\Delta$ there is an integer $n\geq 1$ such that $p^{(n)}(x,y)>0$. The spectral radius of an irreducible random walk~$(X_n)_{n\geq 0}$ with transition operator~$P$ is
$$
\varrho(P)=\limsup_{n\to\infty}p^{(n)}(x,y)
$$
(by irreducibility this value does not depend on the pair $x,y\in\Delta$). 

We will assume that the random walk has bounded range (although most of this section only requires a finite first moment assumption). Let $L_0=\max\{\ell(w)\mid p_w>0\}$, and so the largest possible jump of the random walk has length~$L_0$.

There is a beautiful algebraic structure underlying isotropic random walks. In particular the geometry of the building implies that (see \cite[Theorem~3.4]{P1})
\begin{align}\label{eq:algebra}
\begin{aligned}
P_wP_s=\begin{cases}P_{ws}&\textrm{if $\ell(ws)=\ell(w)+1$}\\
q_s^{-1}P_{ws}+(1-q_s^{-1})P_w&\textrm{if $\ell(ws)=\ell(w)-1$},
\end{cases}
\end{aligned}
\end{align}
from which it immediately follows that the vector space $\mathscr{A}$ over $\mathbb{C}$ with basis $\{P_w\mid w\in W\}$ is an algebra under composition (called the \textit{Hecke algebra} of the building, cf. \cite{P1}). The transition operator $P$ of a bounded range isotropic random walk is an element of the Hecke algebera~$\mathscr{A}$.

The following interpretation of the structure constants in the Hecke algebra, and the ``distance regularity'' statement~(\ref{eq:distreg}) that follows from this interpretation, will be crucial to our investigations.

\begin{prop}\label{prop:distanceregular}
Let $(\Delta,\delta)$ be a regular locally finite building of type $(W,S)$ and let $u,v\in W$. Then
$$
P_uP_v=\sum_{w\in W}\alpha_{u,v}^wP_w,\quad\textrm{where}\quad \alpha_{u,v}^w=\frac{q_w}{q_uq_v}|\Delta_u(x)\cap \Delta_{v^{-1}}(y)|
$$
for any pair of chambers $x,y\in\Delta$ with $\delta(x,y)=w$. In particular, the numbers 
\begin{align}\label{eq:distreg}
a_{u,v}^w=|\Delta_u(x)\cap \Delta_{v}(y)|\quad\textrm{with $\delta(x,y)=w$}
\end{align}
do not depend on the particular pair $x,y\in\Delta$ with $\delta(x,y)=w$.
\end{prop}

\begin{proof} Since $\mathscr{A}$ is an algebra, we have $P_uP_v=\sum_w \alpha_{u,v}^wP_w$ for some numbers $\alpha_{u,v}^w\in\mathbb{C}$. Let $y\in\Delta$, and let $\delta_y:\Delta\to\mathbb{C}$ be the Kronecker delta function. Then $P_w\delta_y(x)=q_w^{-1}$ if $y\in\Delta_w(x)$ and $0$ otherwise, and a direct calculation shows that $P_uP_v\delta_y(x)=q_u^{-1}q_v^{-1}|\Delta_u(x)\cap\Delta_{v^{-1}}(y)|$, completing the proof (see also~\cite[Proposition~3.9]{P1}).
\end{proof}

\begin{lemma}\label{lem:add1} If $\alpha_{u,v}^w\neq 0$ then $w=uv'$ for some $v'\in W$ with $\ell(v')\leq \ell(v)$. 
\end{lemma}

\begin{proof}
We prove the lemma by induction on $\ell(v)$, with the base case $\ell(v)=0$ being trivial. Suppose that the result is true for $\ell(v)=k$, and let $s\in S$ with $\ell(vs)=\ell(v)+1$. Then by~(\ref{eq:algebra}) and the induction hypothesis we have
$$
P_uP_{vs}=P_uP_vP_s=(P_uP_v)P_s=\sum_{z\in W\,:\,\ell(z)\leq \ell(v)}\alpha_{u,v}^{uz}P_{uz}P_s.
$$
By~(\ref{eq:algebra}) we have either $P_{uz}P_s=P_{uzs}$ (in the case that $\ell(uzs)=\ell(uz)+1$), or $P_{uz}P_s=q_s^{-1}P_{uzs}+(1-q_s^{-1})P_{uz}$ (in the case that $\ell(uzs)=\ell(uz)-1$). Since $\ell(z)\leq \ell(v)<\ell(vs)$ and $\ell(zs)\leq\ell(z)+1\leq \ell(v)+1=\ell(vs)$ we see that $P_uP_{vs}$ is a linear combination of the operators $P_{uz'}$ with $\ell(z')\leq \ell(vs)$, hence the result.
\end{proof}

\subsection{Irreducibility and aperiodicity}

Let $P=\sum_{w\in W}p_wP_w$ be the transition operator of an isotropic random walk~$(X_n)_{n\geq 0}$ on~$\Delta$. The \textit{support} of $P$ is $\mathrm{supp}(P)=\{w\in W\mid p_w>0\}$.

\begin{lemma}\label{lem:irreducible} Let $(X_n)_{n\geq 0}$ be an isotropic random walk on a thick regular building with transition operator~$P$ as in (\ref{eq:transitionop}), and write $P^n=\sum_wp_w^{(n)}P_w$.
\begin{enumerate}
\item If the support of~$P$ generates $W$ then $(X_n)_{n\geq 0}$ is irreducible. 
\item If $(X_n)_{n\geq 0}$ is irreducible then for each $k>0$ there is $M_k>0$ such that $p_w^{(M_k)}>0$ for all $w\in W$ with $\ell(w)\leq k$. 
\item If $(X_n)_{n\geq 0}$ is irreducible, then $(X_n)_{n\geq 0}$ is aperiodic.
\end{enumerate}
\end{lemma}

\begin{proof}
1. Let $x,y\in \Delta$, and let $A$ be an apartment containing~$x$ and~$y$. 
Since the support of $P$ generates~$W$ there are elements $w_1,\ldots,w_n\in
\mathrm{supp}(P)$ such that $\delta(x,y)=w_1w_2\cdots w_n$. Let $x_0=x$ and let $x_1,\ldots,x_n\in A$ be the unique chambers of the apartment~$A$ with $\delta(x,x_k)=w_1\cdots w_k$ for $k=1,\ldots,n$. In particular, $x_n=y$. Then, since $x_0,x_1,\ldots,x_k$ all lie in the apartment~$A$, we have $\delta(x_{k-1},x_k)=\delta(x,x_{k-1})^{-1}\delta(x,x_k)=w_k$. Thus $p(x_{k-1},x_k)=p_{w_k}/q_{w_k}>0$, and so
$$
p^{(n)}(x,y)\geq p(x,x_1)p(x_1,x_2)\cdots p(x_{n-1},y)>0,
$$
showing that $P$ is irreducible. 

2. Suppose that $(X_n)_{n\geq 0}$ is irreducible. Thus for each $s\in S$ there
is $N_s\geq 1$ such that $p_{s}^{(N_s)}>0$. The formula
$P_{s}^2=q_{s}^{-1}I+(1-q_s^{-1})P_{s}$ from (\ref{eq:algebra}) implies that
$p_1^{(2N_s)}>0$ and $p_{s}^{(2N_s)}>0$. Thus setting $N=2\sum_{s\in S} N_s$ we have $p_1^{(N)}>0$ and $p_s^{(N)}>0$ for all $s\in S$. Thus taking $M_k=kN$ gives $p_w^{(M_k)}>0$ for all $w\in W$ with $\ell(w)\leq k$.

 3. Suppose that $p_w>0$, and let $k=\ell(w)$. By the previous part we have $p_{w^{-1}}^{(M_k)}>0$. If $w=s_1\cdots s_{k}$ is reduced, then using (\ref{eq:algebra}) we have
\begin{align*}
P_wP_{w^{-1}}&=P_{s_1}\cdots P_{s_{k}}P_{s_{k}}\cdots P_{s_1}=q_w^{-1}I+\cdots,
\end{align*}
where ``$+\cdots$'' is a nonnegative linear combination of the $P_v$ with $v\in W$. Therefore
$$
P^{M_k+1}=PP^{M_k}=p_wp_{w^{-1}}^{(M_k)}P_wP_{w^{-1}}+\cdots=q_w^{-1}p_wp_{w^{-1}}^{(M_k)}I+\cdots,
$$
and so $p_1^{(M_k+1)}>0$. Since we also have $p_1^{(M_k)}>0$ the walk is aperiodic. 
\end{proof}

\begin{remark}
If $(X_n)_{n\geq 0}$ is irreducible then it is not necessarily true that
$\{w\in W\mid p_w>0\}$ generates~$W$. For example if $p_w>0$ if and only if
$\ell(w)=2$ then the random walk $(X_n)_{n\geq 0}$ is irreducible, yet $\{w\in
W \mid \ell(w)=2\}$ only generates the index~$2$ subgroup of all even length elements of~$W$. 
\end{remark}

\subsection{The retracted walk}

An indispensable technique in our analysis of isotropic random walks $(X_n)_{n\geq 0}$ on~$(\Delta,\delta)$ is to look at the image $\overline{X}_n=\rho(X_n)$ of the random walk under the canonical retraction $\rho:\Delta\to W$. In Proposition~\ref{prop:project} below we show that the stochastic process $(\overline{X}_n)_{n\geq 0}$ on $W$ is in fact a random walk on~$W$, which we call the \textit{retracted walk}. However we note in advance that the retracted walk is not $W$-invariant. That is, $\overline{p}(wu,wv)\neq \overline{p}(u,v)$ in general. However we we will prove a more delicate invariance property in Proposition~\ref{prop:invariance} later in this section.  

\begin{prop}\label{prop:project}
The isotropic random walk $(X_n)_{n\geq 0}$ is factorisable over $W$ with respect to the partition of $\Delta$ into sets $\Delta_w(o)$ with $w\in W$. Moreover, the transition probabilities $\overline{p}(u,v)$ of the factor walk $(\overline{X}_n)_{n\geq 0}$ (where $\overline{X}_n=\rho(X_n)$) on $W$ are given by
$$
\overline{p}(u,v)=\sum_{w\in W}a_{v,w}^uq_w^{-1}p_w=q_u^{-1}q_v\sum_{w\in W}\alpha_{v,w^{-1}}^up_w,
$$
where $a_{v,w}^u\geq 0$ and $\alpha_{v,w^{-1}}^u\geq 0$ are the numbers appearing in Proposition~\ref{prop:distanceregular}.
\end{prop}

\begin{proof}
Let $u,v\in W$, and let $x\in\Delta_u(o)$. Then by Proposition~\ref{prop:distanceregular},
\begin{align*}
\sum_{y\in\Delta_v(o)}p(x,y)&=\sum_{w\in W}\sum_{y\in\Delta_v(o)\cap \Delta_w(x)}p(x,y)=\sum_{w\in W}|\Delta_v(o)\cap \Delta_w(x)|q_w^{-1}p_w=\sum_{w\in W}a_{v,w}^uq_w^{-1}p_w.
\end{align*}
This proves the first equality, and the final equality follows from the definitions of the numbers $a_{v,w}^u$ and $\alpha_{v,w^{-1}}^u$. 
\end{proof}

The following proposition tells us that the return probabilities for the random walk $(X_n)_{n\geq 0}$ can be obtained from the return probabilities for the retracted walk $(\overline{X}_n)_{n\geq 0}$.

\begin{prop}\label{prop:retractedwalkspectralradius}
Let $P$ be an irreducible isotropic random walk on a regular building $(\Delta,\delta)$ of type $(W,S)$, and let $\overline{P}$ be the transition operator of the retracted walk on~$(W,S)$. Then
$$
p^{(n)}(o,o)=\overline{p}^{(n)}(1,1)\qquad\textrm{for all $n\geq 1$},
$$
and thus $\vrho(\overline{P})=\vrho(P)$.
\end{prop}

\begin{proof}
From Proposition~\ref{prop:project} (applied to $P^n$) we have
$$
\overline{p}^{(n)}(1,1)=\sum_{w\in W}a_{1,w}^1q_w^{-1}p_w^{(n)},
$$
and since $a_{1,w}^1=|\Delta_1(o)\cap\Delta_w(o)|=\delta_{w,1}$ we have $\overline{p}^{(n)}(1,1)=p_1^{(n)}=p^{(n)}(o,o)$. Since $P$ and $\overline{P}$ are irreducible it follows that $\varrho(P)=\limsup_{n\to\infty}p^{(n)}(o,o)=\limsup_{n\to\infty}\overline{p}^{(n)}(1,1)=\varrho(\overline{P})$. 
\end{proof}

The retracted walk is not $W$-invariant, however we have the following weaker invariance property which roughly says that the transition probabilities of the retracted walk in two cones of the same type are the same.

\begin{prop}\label{prop:invariance}
Let $\bT$ be a cone type of $(W,S)$ and $w_{1},w_{2}\in W$ with $T(w_1)=T(w_2)=\bT$.
Then
$$
\overline{p}(w_1u,w_1v)=\overline{p}(w_2u,w_2v)\quad\textrm{for all $u\in \bT$ and all $v\in \mathrm{Int}_{L_0} \bT=\bT\setminus\partial_{L_0}\bT$}.
$$
\end{prop}

\begin{proof} By the formula for $\overline{p}(\cdot,\cdot)$ in Proposition~\ref{prop:project} it is sufficient to show that
$$
q_{w_1u}^{-1}q_{w_1v}\alpha_{w_1v,w^{-1}}^{w_1u}=q_{w_2u}^{-1}q_{w_2v}\alpha_{w_2v,w^{-1}}^{w_2u}\quad\textrm{whenever $w\in W$ is such that $p_w>0$}.
$$
First note that since $u,v\in \bT$ we have $\ell(w_iu)=\ell(w_i)+\ell(u)$ and $\ell(w_iv)=\ell(w_i)+\ell(v)$ for each $i=1,2$, and therefore $q_{w_iu}=q_{w_i}q_u$ and $q_{w_iv}=q_{w_i}q_v$ for each $i=1,2$. Therefore it is sufficient to show that 
\begin{align}\label{eq:alphaeq}
\alpha_{w_1v,w}^{w_1u}=\alpha_{w_2v,w}^{w_2u}\quad\text{whenever $\ell(w)\leq L_0$}
\end{align}
(we have replaced $w$ by $w^{-1}$, and noted that $p_{w^{-1}}>0$ implies that $\ell(w)\leq L_0$). 

Since $\ell(w_1v)=\ell(w_1)+\ell(v)$ it follows from (\ref{eq:algebra}) and Proposition~\ref{prop:distanceregular} that
\begin{align*}
P_{w_1v}P_{w}&=P_{w_1}P_vP_{w}=P_{w_1}\sum_{w'\in W}\alpha_{v,w}^{w'}P_{w'}=\sum_{w'\in W}\alpha_{v,w}^{w'}P_{w_1}P_{w'}.
\end{align*}
By Lemma~\ref{lem:add1} we see that if $\alpha_{v,w}^{w'}\neq 0$ then $w'=v\tilde{w}$ for some $\tilde{w}$ with $\ell(\tilde{w})\leq \ell(w)$, and therefore 
$$
d(v,w')=\ell(v^{-1}w')=\ell(\tilde{w})\leq\ell(w)\leq L_0.
$$
Thus $w'\in \bT$ (since $v\in \bT\setminus\partial_{L_0} \bT$), and therefore $\ell(w_1w')=\ell(w_1)+\ell(w')$, giving $P_{w_1}P_{w'}=P_{w_1w'}$. Thus
\begin{align}\label{eq:psum1}
P_{w_1v}P_{w}=\sum_{w'\in W}\alpha_{v,w}^{w'}P_{w_1w'}.
\end{align}
On the other hand we have
\begin{align}\label{eq:psum2}
P_{w_1v}P_{w}=\sum_{w''\in W}\alpha_{w_1v,w}^{w''}P_{w''}=\sum_{w'\in W}\alpha_{w_1v,w}^{w_1w'}P_{w_1w'}.
\end{align}
Comparing~(\ref{eq:psum1}) and (\ref{eq:psum2}) and using the linear independence of the operators gives
$$
\alpha_{w_1v,w}^{w_1w'}=\alpha_{v,w}^{w'}\qquad\textrm{for all $w,w'\in W$ with $\ell(w)\leq L_0$}.
$$
The same formula holds with $w_1$ replaced by $w_2$, and (\ref{eq:alphaeq}) follows by taking $w'=u$.
\end{proof}

\subsection{The spectral radius}

In the following theorem we give a sufficient condition for the spectral radius of an isotropic random walk on a regular building to have spectral radius strictly less than~$1$. 

Let $(W,S)$ be a Coxeter system, and let $\cF=\{I\subseteq S\mid \text{$W_I$ is finite}\}$. For each $w\in W$, let $R(w)=\{s\in S\mid \ell(ws)=\ell(w)-1\}$ be the \textit{right descent set of~$w$}. By \cite[Corollary~2.18]{AB} we have that $R(w)\in\cF$ for all $w\in W$. 

\begin{thm}\label{thm:spectralradius} Let $(W,S)$ be a Coxeter system with $W$ infinite and let $(\Delta,\delta)$ be a regular building of type $(W,S)$. Let $P$ be the transition operator of an irreducible isotropic random walk on~$(\Delta,\delta)$. If
\begin{align}\label{eq:Iineq}
\sum_{s\in S\setminus I}q_s\geq |I|\quad\text{for all $I\in \cF$}
\end{align}
then the spectral radius $\vrho(P)$ is strictly less than~$1$. In particular, if $q_s\geq |S|-1$ for all $s\in S$ then~$\vrho(P)<1$.
\end{thm}

\begin{proof}
Suppose first that $P$ is the simple random walk on~$\Delta$. Furthermore, suppose first that strict inequality holds in~(\ref{eq:Iineq}) for all $I\in \cF$, and let $C=\min_{I\in \cF}(\sum_{s\in S\setminus I}q_s-|I|)/Q>0$ where $Q=\sum_{s\in S}q_s$ is the total number of chambers adjacent to any given chamber. Let $x\in\Delta$ and $w=\rho(x)$. Let $I=R(w)\in\cF$. Let $Y_n=d(o,X_n)$. Then
\begin{align*}
\mathbb{E}[Y_{n+1}-Y_n\mid X_n=x]&=\mathbb{P}[Y_{n+1}-Y_n=1\mid X_n=x]-\mathbb{P}[Y_{n+1}-Y_n=-1\mid X_n=x]\\
&=\frac{\sum_{s\in S\setminus I}q_s-|I|}{Q}.
\end{align*}
Thus $\mathbb{E}[Y_{n+1}-Y_n\mid X_n]\geq C$, and so the sequence $Z_n=Y_n-Cn$ is a submartingale with respect to $(X_n)_{n\geq 0}$. We have
$$
p^{(n)}(o,o)=\mathbb{P}[Y_n=0\mid X_0=o]\leq \mathbb{P}[Z_n\leq -Cn\mid X_0=o]\leq e^{-C^2n/2}
$$
where the last inequality is Azuma's Inequality. Thus $\vrho(P)\leq e^{-C^2/2}<1$. 

We now briefly sketch the proof in the more general case where we do not assume strict inequality in~(\ref{eq:Iineq}), with $P$ still the simple random walk on~$\Delta$. Note that the singleton $I=\{s'\}$ is in~$\cF$, and that $\sum_{s\neq s'}q_s>|I|$. It can be seen that there is a number $K>0$ such that for each chamber $x\in\Delta$ there is an element $x'$ with $d(x,x')\leq K$ such that $R(\rho(x'))=\{s'\}$. Using this fact, and looking at the $(K+1)$-step walk $P^{K+1}$, an argument analogous to the above, using a telescoping sum, shows that $\mathbb{E}[Y_{n+K+1}-Y_n\mid X_n]\geq C(1/Q)^{K+1}$, where $C=\sum_{s\neq s'}q_s-1>0$. The result now follows as above.

Now let $P=\sum_{w\in W} p_wP_w$ be an arbitrary isotropic random walk on $\Delta$. By Lemma~\ref{lem:irreducible} there is $N>0$ such that $p_s^{(N)}>0$ for all $s\in S$. Thus, writing $\tilde{P}$ for the simple random walk operator on~$\Delta$, we have
$$
P^N=b\tilde{P}+\sum_{w\in W}b_wP_w\qquad\textrm{where $b>0$ and $b_w\geq 0$ for all $w\in W$}.
$$
The condition $\sum_{w\in W} p_w^{(N)}=1$ gives $b+\sum b_w=1$. Since $\tilde{P}$ is symmetric we have~$\|\tilde{P}\|=\vrho(\tilde{P})<1$ by the above argument (where $\|P\|$ is the operator norm of $P:\ell^2(\Delta)\to\ell^2(\Delta)$). Thus, since $\|P_w\|\leq 1$ for all $w$, we see that\begin{align*}
\vrho(P)^N=\vrho(P^N)&\leq\|P^N\|\leq b\|\tilde{P}\|+\sum_{w\in W}b_w\|P_w\|< b+\sum_{w\in W}b_w=1.
\end{align*}
(The first equality holds since $P$ is irreducible and aperiodic, see \cite[Exercise~1.10]{woessbook}). 
\end{proof}

\begin{cor}\label{cor:spectralradiusbuilding} An isotropic random walk on any regular thick Fuchsian building has spectral radius strictly less than~$1$. 
\end{cor}

\begin{proof}
Let $(W,S)$ be a Fuchsian Coxeter system. Any three distinct elements of $S$ generate an infinite group, and so $|I|\leq 2$ for all $I\in \cF=\{I\subseteq S\mid \text{$W_I$ is finite}\}$. Thus if $|S|\geq 4$ we have 
$$
\sum_{s\in S\setminus I}q_s\geq 2|S\setminus I|=2(|S|-|I|)\geq 2(4-2)=4> |I|\quad\textrm{for all $I\in \cF$},
$$
and so the result follows from Theorem~\ref{thm:spectralradius}. If $|S|=3$ and $|I|=1$ then $\sum_{s\in S\setminus I}q_s\geq 4>|I|$, and if $|S|=3$ and $|I|=2$ then $\sum_{s\in S\setminus I}q_s\geq 2=|I|$, completing the proof.
\end{proof}

\begin{remark}
We do not think that Theorem~\ref{thm:spectralradius} is optimal. In fact, we believe that every thick regular building $(\Delta,\delta)$ of type $(W,S)$ with $W$ infinite has $\vrho(P)<1$ for all irreducible isotropic random walks. However the conclusion of Corollary~\ref{cor:spectralradiusbuilding} is sufficient for our purposes.
\end{remark}

\subsection{The path space}

Let $\cT_{x}\subset \Delta^{\mathbb{N}}$ denote the space of all paths in~$\Delta$ starting at~$x\in\Delta$ (with jumps of any length allowed). More formally, the path space is defined as the inverse limit
$$
\cT_x=\varprojlim\cT_x^n=\bigg\{\gamma\in\prod_{n\geq 0}\cT_x^n\,\big|\,\gamma_i=\pi_{ij}(\gamma_j)\text{ for all $i\leq j$}\bigg\}
$$
where $\cT_x^n=\{x\}\times \Delta^{n-1}\subset\Delta^n$ is the space of all paths in~$\Delta$ of length $n-1$ starting at~$x\in\Delta$, and $\pi_{ij}:\cT_x^j\to \cT_x^i$ are the natural projections. From this description we see (from Tychonoff's Theorem) that $\cT_x$ is a compact Hausdorff topological space.  

In the case of a Cayley graph of a group, there is naturally an automorphism of the graph taking any given vertex $x$ to any other vertex $y$, and thus there is a bijection $\psi_{xy}:\cT_x\to\cT_y$ mapping paths based at $x$ to ``isomorphic'' paths based at~$y$. In effect, this gives the intuition that random walks starting at~$x$ ``behave the same as'' random walks starting at~$y$. In our context there is typically not an automorphism of~$\Delta$ taking $x$ to $y$, and so we need to work a little harder to construct a suitable bijection $\psi_{xy}:\cT_x\to\cT_y$. The distance regularity of Proposition~\ref{prop:distanceregular} plays a crucial role here. 

\begin{prop}\label{prop:pathbijection}
For each $x,y\in \Delta$ there is a bijection $\psi_{xy}:\cT_x\to\cT_y$ such that:
\begin{enumerate}
\item If $\gamma=(x_0,x_1,x_2,\ldots)\in\cT_x$ and $\psi_{xy}(\gamma)=(y_0,y_1,y_2,\ldots)$, then $\delta(x_{i},x_{i+1})=\delta(y_{i},y_{i+1})$ and $\delta(x,x_i)=\delta(y,y_i)$ for all $i\geq 0$. 

\item For all $L\geq 0$ and for all $x,y$ of same cone type, if $\gamma=(x=x_0,x_1,x_2,\ldots)$ and $\psi_{xy}(\gamma)=(y=y_0,y_1,y_2,\ldots)$ and if $x_j\in\mathrm{Int}_{L} C(x)$ for some $j\geq 0$, then $y_j \in \mathrm{Int}_{L} C(y)$.

\item We have $\mathbb{P}_x[ (X_{n})_{n\geq 0}\in A]=\mathbb{P}_y[  (X_{n})_{n\geq 0}\in \psi_{xy}(A)]$ for all measurable $A\subseteq \cT_x$, where $\PP_{x}$ denotes the distribution of the isotropic random walk $(X_{n})_{n\geq 0}$  started at $X_0=x\in\Delta$.
\end{enumerate}
\end{prop}

\begin{proof}
We inductively build bijections $\psi_{xy}^n:\cT_x^n\to\cT_y^n$ satisfying part 1 of the proposition (for $0\leq i\leq n$). The case $n=0$ is trivial. Suppose that $\psi_{xy}^n:\cT_x^n\to\cT_y^n$ has been constructed. For any finite path $\gamma=(x_{1},\ldots, x_{n})$ and any point $x_{n+1}$ we define  $\gamma\circ x_{n+1}=(x_{1},\ldots, x_{n}, x_{n+1})$.

Write
\begin{align*}
\cT_x^{n+1}&=\{\gamma_n\circ x_{n+1}\mid \gamma_n\in\cT_x^n,\,x_{n+1}\in\Delta\}\\
&=\bigsqcup_{u,v\in W}\{\gamma_n\circ  x_{n+1}\mid\gamma_n\in\cT_x^n,\,x_{n+1}\in\Delta_u(x)\cap\Delta_v(x_n)\},
\end{align*}
where $\gamma_n=(x_0,\ldots,x_n)$. For each $\gamma_n\in\cT_x^n$, the set $\gamma_n\circ \{x_{n+1}\mid x_{n+1}\in\Delta_u(x)\cap\Delta_v(x_n)\}$ has cardinality $a_{uv}^w$, where $w=\delta(x,x_n)$ (see Proposition~\ref{prop:distanceregular}). We also have
$$
\cT_y^{n+1}=\bigsqcup_{u,v\in W}\{\psi_{xy}^n(\gamma_n)\cdot y_{n+1}\mid \gamma_n\in\cT_x^n,\,y_{n+1}\in \Delta_u(y)\cap \Delta_v(y_n)\}
$$
where $\psi_{xy}^n(\gamma_n)=(y_0,\ldots,y_n)$. For each $\gamma_n\in\cT_x^n$ the set $\psi_{xy}^n(\gamma_n)\circ\{ y_{n+1}\mid y_{n+1}\in \Delta_u(y)\cap \Delta_v(y_n)\}$ also has cardinality $a_{uv}^w$ since $\delta(y,y_n)=\delta(x,x_n)=w$ (by the induction hypothesis). Thus for each fixed $\gamma_n\in\cT_x^n$ and each $u,v\in W$ we can choose a bijection
$$
\theta_{xy}^{uv}[\gamma_n]:\{x_{n+1}\mid x_{n+1}\in\Delta_u(x)\cap\Delta_v(x_n)\}\to \{y_{n+1}\mid y_{n+1}\in \Delta_u(y)\cap \Delta_v(y_n)\}.
$$
Thus for each $\gamma_n\in\cT_x^n$ we obtain a bijection $\theta_{xy}[\gamma_n]:\Delta\to\Delta$ (depending on $\gamma_n$) by the rule
$$
\theta_{xy}[\gamma_n](x_{n+1})=\theta_{xy}^{uv}[\gamma_n](x_{n+1})\quad\text{if $x_{n+1}\in\Delta_u(x)\cap\Delta_v(x_n)$}.
$$
Then define $\psi_{xy}^{n+1}:\cT_x^{n+1}\to\cT_y^{n+1}$ by
\begin{align}\label{eq:project}
\psi_{xy}^{n+1}(\gamma_{n+1})=\psi_{xy}^n(\gamma_n)\circ \theta_{xy}[\gamma_n](x_{n+1}).
\end{align}
By construction this bijection satisfies the conditions in part~$1$ of the proposition for $0\leq i\leq n+1$.

We now construct a bijection $\psi_{xy}:\cT_x\to\cT_y$ satisfying part~1 of the proposition. For $\gamma=(\gamma_0,\gamma_1,\ldots)\in\cT_x$ let
$$
\psi_{xy}(\gamma)=(\psi_{xy}^0(\gamma_0),\psi_{xy}^1(\gamma_1),\ldots).
$$
By (\ref{eq:project}) we see that $\psi_{xy}(\gamma)\in\cT_y$ for all $\gamma\in\cT_x$, and hence $\psi_{xy}:\cT_x\to\cT_y$. If $\psi_{xy}(\gamma)=\psi_{xy}(\gamma')$ then $\gamma_i=\gamma_i'$ for all $i\geq 0$, and hence $\gamma=\gamma'$ and so $\gamma$ is injective. To check surjectivity, if $\gamma=(\gamma_0,\gamma_1,\ldots)\in\cT_y$ then let $\gamma'=((\psi_{xy}^0)^{-1}(\gamma_0),(\psi_{xy}^1)^{-1}(\gamma_1),\ldots)$. Then $\gamma'\in \cT_x$ (using (\ref{eq:project})), and hence $\psi_{xy}$ is  surjective. 

It follows that $\psi_{xy}:\cT_x\to\cT_y$ is a bijection satisfying the conditions in part~1 of the proposition. Then from Proposition~\ref{prop:conetypes} and Proposition~\ref{prop:Lboundaries} we have
$$
\partial_{L}C(x)=\{z\in C(x)\mid \delta(x,z)\in \partial_{L}T(\rho(x))\},
$$
and part~2 of the proposition follows from this description. Since $(X_n)_{n\geq 0}$ is an isotropic random walk part 1 of the proposition implies part~3. 
\end{proof}

On occasion we will consider $\psi_{xy}$ as a bijection $\psi_{xy}:\cT_x^n\to\cT_y^n$ for each fixed $n\geq 0$ (that is, we write $\psi_{xy}$ in place of $\psi_{xy}^n$).

\subsection{Isotropic random walks and groups}\label{sect:6}

The following proposition (cf. \cite[Lemma~8.1]{CW}) illustrates how isotropic random walks naturally arise from bi-invariant probability measures on groups acting on buildings.

\begin{prop}\label{prop:Kac}
Let $G$ be a locally compact group acting transitively on a regular building $(\Delta,\delta)$, and let $B$ be the stabiliser in~$G$ of a fixed base chamber~$o$. Normalise the Haar measure on $G$ so that $B$ has measure~$1$. Let $\varphi$ be the density function of a $B$-bi-invariant probability measure on~$G$. If the group~$B$ acts transitively on each set $\Delta_w(o)$ with $w\in W$, then the assignment
$$
p(go,ho)=\varphi(g^{-1}h)
$$
for $g,h\in G$ defines an isotropic random walk on~$(\Delta,\delta)$.
\end{prop}

\begin{proof}
To check that $p(\cdot,\cdot)$ is well defined, suppose that $g_1o=go$ and $h_1o=ho$. Then $g_1^{-1}g\in B$ and $h^{-1}h_1\in B$, and thus
$g_1^{-1}h_1\in Bg^{-1}hB$, and so $\varphi(g_1^{-1}h_1)=\varphi(g^{-1}h)$. 

For each $x\in\Delta$ use transitivity to fix an element $g_x\in G$ with $g_xo=x$. Then $G$ is the disjoint union of cosets $g_xB$, $x\in\Delta$, and thus
$$
\sum_{y\in\Delta}p(x,y)=\sum_{y\in\Delta}\varphi(g_x^{-1}g_y)=\sum_{y\in\Delta}\int_{g_x^{-1}g_yB}\varphi(g)\,dg=\int_G\varphi(g)\,dg=1.
$$

Clearly $p(gx,gy)=p(x,y)$ for all $g\in G$ and all $x,y\in\Delta$, and since $B$ is transitive on each set $\Delta_w(o)$ it follows that $p(\cdot,\cdot)$ is isotropic.
\end{proof}

Thus Theorems~\ref{thm:LLNbuilding} and~\ref{thm:CLTbuilding} give a rate of escape theorem and a central limit theorem (with formulas for the speed and variance) for random walks induced by $B$-bi-invariant measures on groups acting, as in Proposition~\ref{prop:Kac}, on Fuchsian buildings, where $B$ is the stabiliser of a chamber. The finite range assumption amounts to assuming that the density function of the $B$-bi-invariant measure is supported on finitely many $B$ double cosets. An important example is the case where $G=G(\mathbb{F}_q)$ is a Fuchsian Kac-Moody group over a finite field $\mathbb{F}_q$, acting on its natural building $G/B$ (as in Example~\ref{ex:kac}), and thus Corollary~\ref{cor:Kac} follows from Proposition~\ref{prop:Kac} and Theorems~\ref{thm:LLNbuilding} and~\ref{thm:CLTbuilding}.

\section{Isotropic random walks on regular Fuchsian buildings}\label{sect:renewal}

We now restrict our attention to irreducible isotropic random walks on a thick regular Fuchsian building. Thus in this section $(W,S)$ denotes a Fuchsian Coxeter system, $(\Delta,\delta)$ is a thick regular Fuchsian building of type~$(W,S)$, and $P=\sum_{w\in W}p_wP_w$ is the transition operator of an isotropic random walk $(X_n)_{n\geq 0}$ on~$\Delta$. For the remainder of this section we fix a recurrent cone type~$\bT$. 

We will assume that $(X_n)_{n\geq 0}$ has bounded range. Thus there is a minimal number $L_0\geq 0$ such that 
\begin{align}
\label{eq:L0}\text{$p_w\neq 0$ implies that $\ell(w)\leq L_0$}.
\end{align}

It is sufficient to prove Theorems~\ref{thm:LLNbuilding} and~\ref{thm:CLTbuilding} under the assumption that $p_s>0$ for all $s\in S$, and so there is an $\varepsilon>0$ such that
\begin{align}
\label{eq:support}\text{$p(x,y)>\varepsilon$ whenever $d(x,y)=1$}.
\end{align}
To see this, note that by Lemma~\ref{lem:irreducible}.2 there is an $M\geq 1$ such that the $M$-step walk $(X_{nM})_{n\geq 0}$ satisfies $p_s^{(M)}>0$ for all $s\in S$, and by the bounded range assumption proving Theorems~\ref{thm:LLNbuilding} and~\ref{thm:CLTbuilding} for the $M$-step walk implies the theorems for the $1$-step walk~$(X_n)_{n\geq 0}$. Thus, without loss of generality we will assume~(\ref{eq:support}) throughout this section.

\subsection{Renewal times and the proofs of Theorems~\ref{thm:LLNbuilding} and~\ref{thm:CLTbuilding}}\label{sect:yT}

In this section we setup a renewal structure for isotropic random walks on Fuchsian buildings. The main result is Theorem~\ref{thm:reg}, which is the key ingredient in the proofs of our rate of escape and central limit theorems. The proof of Theorem~\ref{thm:reg} will occupy Sections~\ref{sec:proofreg1} and~\ref{sec:proofreg2}. 

We note that Theorems~\ref{thm:LLNbuilding} and~\ref{thm:CLTbuilding} can be proven by only developing a renewal structure for the retracted random walk~$(\overline{X}_n)_{n\geq 0}$ on~$W$. However here we will develop a more satisfying picture by proving a renewal structure for the walk~$(X_n)_{n\geq 0}$ on the building. This only requires a small amount more work, and in our opinion is more natural.

We start by recalling the crucial fact that geodesics in a hyperbolic group
either stay within bounded distance of each other or diverge
exponentially. More precisely, there exists some exponential divergence
function $e: \NN_0 \rightarrow \RR$ such that the following holds: for all $u\in W$ and all geodesics
$\gamma_{1}$ from $u$ to any $v_{1}\in W$ and $\gamma_{2}$ from $u$ to any
$v_{2}\in W$ and all $r,R\in \NN_0$ with $R+r\leq \min\{d(u,v_{1}), d(u,v_{2})\}$
and $d(\gamma_{1}(R), \gamma_{2}(R))\geq e(0)$, all paths starting in $\gamma_{1}(R+r)$, visiting only vertices in $W\setminus
\cB(u,R+r)$ and ending in $\gamma_{2}(R+r)$ have length of
at least $e(r)$. Here $\gamma_i(n)$ is the point on $\gamma_i$ at distance
$n\in\NN_0$ to $u$. In particular, two geodesics that have been at least $e(0)$ apart can never
intersect again. 

\begin{lemma}\emph{(c.f. \cite[Lemma~2.4]{HMM:13})}\label{lem:boundary-rays}
Let be $u\in W\setminus \{e\}$. Then the boundary $\partial_1 C_W(u)$ is contained in the union of two geodesic rays starting at~$u$ in the Coxeter complex.
\end{lemma}

\begin{proof}
Since the Coxeter complex is homeomorphic to the hyperbolic disc, it can be endowed with an orientation. Let $r_1, r_2: \mathbb{N}_0\to W$ be two infinite geodesic rays in the Coxeter complex going through $u$ which coincide up to $u$. Let $c_1$ and $c_2$ be the geodesic rays extracted from $r_1,r_2$ starting at~$u$. Let $V$ be a component of  $W\setminus \{c_1 \cup c_2\}$ which does not contain~$e$; here we identify $c_1$ and $c_2$ as the sets of vertices which lie on the geodesics. Let us prove that $V$ is contained in $C_{W}(u)$: let $v\in V$, and let us consider a geodesic segment $c_v$ joining $e$ to $v$. Since the Coxeter complex is planar, JordanÕs Theorem implies that $c_v$ intersects $\partial V=\{ w_1\in V \mid \exists w_2\in W\setminus V: d(w_1,w_2)=1\}$ at some point $w$, hence $c_1$ or $c_2$. Let us assume that it intersects $c_1$. Since $c_1$ is geodesic, we may replace the portion of $c_v$ before $w$ by $c_1$: it follows that the concatenation of $c_1$ up to $w$ and $c_v$ from $w$ to $v$ is geodesic; this implies that $v\in C_W(u)$.

By Arzela-AscoliÕs theorem and the planarity of the Coxeter complex, we may find two rays $c_\ell$ and $c_r$ going through $u$ such that $C_W(u)$ is the union of those rays with all the components of their complement which do not contain~$e$.
\end{proof}

Recall that we have fixed a recurrent cone type~$\bT$, and we now fix some $x_{\bT}\in \Delta$ with $T(x_{\bT})=\bT$. Set
$L_1=\max \{L_0,e(0)\}$. In the following we will construct some $y_{\bT}\in \mathrm{Int}_{L_1} C_{\Delta}(x_{\bT})$ such that $C_{\Delta}(y_{\bT})\subset \mathrm{Int}_{L_1} C_{\Delta}(x_{\bT})$. Let us remark that $\mathrm{Int}_{L_1} C_{\Delta}(x_{\bT})$ always contains at least one infinite connected component.

\begin{lemma}\label{lem:equal-distance}
Let be $L\geq 1$ and $y\in \mathrm{Int}_{L} C_{\Delta}(x_{\bT})$. Then:
$$
d(y,\partial_{L} C_{\Delta}(x_{\bT}) )= d(\rho(y), \partial_{L} C_{W}(\rho(x_{\bT}))).
$$
\end{lemma}
\begin{proof}
Since retractions decrease the distance, and since $\rho(\partial_{L}C_{\Delta}(x_{\bT}))=\partial_{L}C_W(\rho(x_{\bT}))$ (see Proposition~\ref{prop:Lboundaries}) we have $d(y,\partial_{L} C_{\Delta}(x_{\bT}) )\geq  d(\rho(y), \partial_{L} C_{W}(\rho(x_{\bT})))$. It remains to show the other inequality to finish the proof. For this purpose, take a path of length $K=d(\rho(y), \partial_{L} C_{W}(\rho(x_{\bT})))$ from $\rho(y)$ to some $v\in \partial_{L} C_{W}(\rho(x_{\bT}))$, say the path $(w_0=\rho(y),w_1,\dots,w_K=v)$.
 We now want to construct a path of length $K$ from $y$ to $\partial_{L} C_{\Delta}(x_{\bT})$. Let $A$ be an apartment which contains $o$ and $y$ (and thus $x_{\bT}$ by (A3)). By Proposition~\ref{prop:conetypes} the retraction $\rho$ maps $C_A(x_{\bT})$ isometrically onto $C_W(\rho(x_{\bT}))$. Therefore, 
 $$
 \pi=((\rho|_A)^{-1}(w_0)=y,(\rho|_A)^{-1}(w_1),\dots,(\rho|_A)^{-1}(v))
 $$ 
 is a path of length $K$ from $y$ to $z=(\rho|_A)^{-1}(v) \in \partial_{L} C_A(x_{\bT})$. Now choose any $z'\in A$ with $d(z,z')=L$ and $z'\notin C_A(x_{\bT})$. Then $\rho(z')\notin \partial_{L} C_W(\rho(x_{\bT}))=\rho (\partial_{L} C_{\Delta}(\x_{\bT}))$, and hence $z'\notin \partial_{L} C_{\Delta}(\x_{\bT})$. That is, $\pi$ is a path of length $K$ in $A\subset \Delta$ which connects $y$ with $\partial_{L} C_{\Delta}(x_{\bT})$. This finishes the proof.
\end{proof}

The following lemma and its corollary will be used to construct~$y_{\bT}$. 

\begin{lemma}
Let $L\geq e(0)$ and let $u\in W\setminus \{e\}$ be such that $T(u)=\bT$. Then there is some $v\in \mathrm{Int}_{L} C_W(u)$ such that $T(v)=\bT$ and $C_W(v)\subset \mathrm{Int}_{1} C_W(u)$.
\end{lemma}

\begin{proof}
Let $u\in W\setminus \{e\}$ with $T(u)=\bT$. By Lemma \ref{lem:boundary-rays} there are two geodesic rays $\gamma_1,\gamma_2$ starting from $e$ whose union contains $\partial_1 C_W(u)$ and which coincide up to $u$. Since $T(u)$ is recurrent we can choose any end $\xi\in \partial_\infty C_W(u)$ which is different from the ends described by $\gamma_1$ and $\gamma_2$. Let $\pi$ be any geodesic ray which starts at $e$, follows $\gamma_1$ up to $u$  and describes $\xi$. It follows that, for every $i\in\{1,2\}$, the distance $d(\pi(t),\gamma_i(t))$  cannot be bounded for $t\geq 0$. Hence, there are $t_1,t_2\in\mathbb{N}$ such that $d(\pi(t_i),\gamma_i(t_i))\geq L+1\geq e(0)+1$ implying that $d(\pi(t),\gamma_i(t))\geq e(0)+1$ for all $t\geq \max\{t_1,t_2\}$ and all $i\in\{1,2\}$. That is, $\pi$ and $\gamma_i$, $i\in\{1,2\}$, diverge exponentially. In particular, there must be some $t_0\geq \max\{t_1,t_2\}$ such that $v'=\pi(t_0)\in \mathrm{Int}_{L} C_W(u)$ with $T(v')$ being recurrent.
Denote by $\gamma_1'$ and $\gamma_2'$ the geodesic rays starting at $e$ whose union contains $\partial_1 C_W(v')$ and pass through $v'$. Due to exponential divergence of $\gamma_i$ and $\gamma_j'$, where $i,j\in \{1,2\}$, we have that $\gamma_i(t)\neq \gamma_j'(t)$ for all $t\geq t_0$. This yields $C_W(v')\subset \mathrm{Int}_1(u)$

The cone $C_W(v')$ contains an element $v$ with $T(v)=\bT$, and since $C_W(v)\subset C_W(v')$ the result follows.
\end{proof}

We can iterate the last step by replacing the role of $u$ by $v$. This leads then to the following corollary:

\begin{cor}\label{cor:find-v}
Let be $u\in W\setminus \{e\}$ such that $T(u)=\bT$. Then there is some $v\in \mathrm{Int}_{L_1} C_W(u)$ such that $T(v)=\bT$ and $C_W(v)\subset \mathrm{Int}_{L_1}C_W(u)$.
\end{cor}

We now show how to construct $y_{\bT}$: take an apartment $A$ which contains $o$ and $x_{\bT}$ and recall that $\rho|_A$ denotes the restriction of $\rho$ to $A$ which becomes an isomorphism mapping $A$ onto $W$. We apply Corollary \ref{cor:find-v} on $u=\rho|_A(x_{\bT})$
 and find some $v\in W$ such that $C_W(v)\subset  \mathrm{Int}_{L_1} C_W(u)$. Due to Proposition \ref{prop:Lboundaries} and Lemma \ref{lem:equal-distance} we then must have $C_{\Delta}((\rho|_A)^{-1}(v))\subset  \mathrm{Int}_{L_1} C_{\Delta}(x_{\bT})$. Fix now for the rest of this section such a chamber $y_{\bT}=(\rho|_A)^{-1}(v)$ in dependence of $x_{\bT}$, $v$ and $A$.
Furthermore, fix a shortest path $\pi_{\bT}=[x_{\bT}, x_{\bT}^{(1)},\dots,x_{\bT}^{(k-1)},y_{\bT}]$ from $x_{\bT}$ to
$y_{\bT}$ contained in $C(x_{\bT})$. Note that for $x\in\Delta$ with $T(x)=\bT$
 the bijection $\psi_{x_{\bT}x}$ maps $\pi_{\bT}$ onto a path from $x$ to some $y\in
 \mathrm{Int}_{L_1} C(x)$ contained in $C_{\Delta}(x)$ (see Proposition \ref{prop:pathbijection}.2).

We now give the definition of \textit{renewal times}. For each $x\in\Delta$ with $T(x)=\bT$ let $\widehat{\mathcal{T}}_{x}$ be the set of all paths which start at $x$, initially follow
$$
\pi_x=\left( \psi_{x_{\bT}x}(x_{\bT}^{(1)}),\psi_{x_{\bT}x}(x_{\bT}^{(2)}), \dots,\psi_{x_{\bT}x}(y_{\bT})\right)
$$  
and stay in $\mathrm{Int}_{L_1} C(x)$ afterwards forever.
We define $ R_{0}=0$ and  let
$$
R_{1}=\inf\{{k\geq 0}\mid (X_{i})_{i\geq k} \in \widehat{\mathcal{T}}_{X_k},~T( X_{k})=\bT \}
$$ 
be the first time $k\in\mathbb{N}$ that the random walk hits the root of a cone of type $\bT$, visits consecutively $\psi_{x_{\bT}X_k}(x_{\bT}^{(1)}),\psi_{x_{\bT}X_k}(x_{\bT}^{(2)}),\dots,\psi_{x_{\bT}X_k}(y_{\bT})$  and
stays in $\mathrm{Int}_{L_1} C(X_k)$ afterwards forever. Inductively, 
\begin{align}\label{eq:renewal}
 R_{n}=\inf\{k> R_{n-1}\mid (X_{i})_{i\geq k} \in \widehat{\mathcal{T}}_{X_k},~T( X_{k})=\bT\}.
 \end{align}

Recall the notion of random variables with exponential moments. A real valued random variable $Y$ has \textit{exponential moments} if 
$\EE[\exp(\lambda Y)]<\infty$ for some $\lambda>0$, or equivalently, if  there
are positive constants $C>0$ and $c<1$ such that $\PP[Y=n]\leq C c^{n} $ for
all $n\in \NN_0$.

\begin{thm}\label{thm:reg}
Let $(X_n)_{n\geq 0}$ be an isotropic random walk on a thick regular Fuchsian building~$(\Delta, \delta)$ with bounded range. 
\begin{enumerate}
\item The renewal times~$R_{n}$ are almost surely finite, $d(o,X_{R_{n}})=\sum_{i=1}^{n} d( X_{R_{i-1}},  X_{R_{i}})$, and $(d( X_{R_{i-1}},  X_{R_{i}}))_{i\geq 2}$ are i.i.d.~random variables. 
\item The renewal time $ R_{1}$ and the increments $(R_{i+1}- R_{i})$ for $i\geq 1$ have exponential moments. The same holds true for $d(o,X_{ R_{1}})$ and $d( X_{ R_{i}}, X_{ R_{i+1}})$ for $i \geq 1$.
\end{enumerate}
\end{thm}

The proof of Theorem~\ref{thm:reg} will be given in Sections~\ref{sec:proofreg1} and~\ref{sec:proofreg2}. Assuming Theorem~\ref{thm:reg} for the moment, one can now argue
as in \cite{HMM:13} (verbatim modulo some notations) to prove our law of large
numbers and central limit theorem.

\begin{proof}[Proof of Theorems~\ref{thm:LLNbuilding} and~\ref{thm:CLTbuilding}]
We content ourselves with giving the main
idea and refer to \cite{HMM:13} for the technical details. The role of the
cones in the definition of the renewal times was that the trajectory of the
walk observed at renewal times is the ``aligned'' sum of i.i.d.~pieces, Theorem
\ref{thm:reg}; that is,
\begin{equation*}
d(o,X_{R_{n}})=\sum_{i=1}^{n} d(R_{i-1}, R_{i}).
\end{equation*}
Now, the law of large numbers and central limit theorem for real-valued random variables apply and the statements in Theorems~\ref{thm:LLNbuilding} and~\ref{thm:CLTbuilding} follow for the process $d(o,X_{R_{n}})$. It remains therefore to control the  distance or ``error'' between $X_{n}$ and the position of the last renewal before time $n$. More precisely, we define the last renewal time before time $n$:
\begin{equation*}
k(n)=\sup\{k\mid R_{k}\le n\}.
\end{equation*}
We have that 
\begin{equation*}
\frac{n}{k(n)}= \frac{n}{R_{k(n)}}\frac{R_{k(n)}}{k(n)}.
\end{equation*}
By the strong law of large numbers  the second factor tends a.s.~to $\EE[R_2-R_1]$.
For the first factor we observe  that $R_{k(n)}\le n \le R_{k(n)+1}$,
hence
\begin{equation*}
\limsup_{n\to\infty} \frac{R_{k(n)}}{n}\le 1.
\end{equation*}
On the other hand, since $n\ge k(n)$ and  $(R_{k(n)}-R_{k(n)+1})$ have finite first moments,
\begin{equation*}
\lim_{n\to\infty} \frac{R_{k(n)}-R_{k(n)+1}}{n}=0\quad \hbox{a.s.}
\end{equation*}
and hence
\begin{equation*}
\liminf_{n\to\infty}\frac{R_{k(n)}}{n}\ge \liminf_{n\to\infty}\left( \frac{R_{k(n)}-R_{k(n)+1}}{n}\right)+ \frac{R_{k(n)+1}}{n}\ge 1.
\end{equation*}
Eventually, we have that
\begin{equation*}
\frac{n}{k(n)} \xrightarrow[n\to \infty]{a.s.} \EE[R_{2}-R_{1}]<\infty.
\end{equation*}
Denote
\begin{equation*}
M_{k}=\sup\{d(Z_{n},Z_{R_{k}})\mid R_{k}\le n \leq R_{k+1}\}, ~k\geq 1.
\end{equation*}
The random variables $(M_{k})_{k\geq 1}$ form an i.i.d.~sequence of random variables with exponential moments. This is a consequence of the fact that  $d( X_{ R_{i+1}}, X_{ R_{i}})$ have exponential moments, see \cite[Corollary 4.2]{HMM:13}. As a consequence we have that
\begin{equation*}
\lim_{n\to\infty}\frac{d(Z_{n},e)-d(Z_{R_{k(n)}},e)}n  \le \lim_{k\to\infty}\frac{M_k}k=0 \quad\hbox{a.s.}.
\end{equation*}
Since the strong law of large numbers guarantees that
\begin{equation*}
\frac{d(Z_{R_{k(n)}},e)}{k(n)}   \xrightarrow[n\to \infty]{a.s.} \EE[ d(Z_{ R_{2}}, Z_{ R_{1}})]
\end{equation*}
we can conclude the proof of Theorem~\ref{thm:LLNbuilding}:
\begin{eqnarray*}
\frac{d(Z_{n},e)}n& =& \frac{d(Z_{n},e)-d(Z_{R_{k(n)}},e)}n + \frac{d(Z_{R_{k(n)}},e)}{k(n)} \frac{k(n)}{n}\cr
&  \xrightarrow[n\to \infty]{a.s.}  & 0  + \frac{\EE[ d(Z_{ R_{2}}, Z_{ R_{1}})]}{\EE[ R_{2}- R_{1}]}.
\end{eqnarray*} 
The proof of Theorem~\ref{thm:CLTbuilding} is more involved; we refer  \cite[Section 4.2]{HMM:13} for the remaining details.
\end{proof}

\subsection{Proof of Theorem~\ref{thm:reg}.1}\label{sec:proofreg1}

Recall that  $\overline{X}_{n}=\rho(X_n)$ denotes the retracted walk on $W$ and its  transition probabilities are given by  $\overline{p}(u,v)$ and its transition operator is~$\overline{P}$. This retracted walk is necessarily irreducible and aperiodic ($P$ is irreducible and thus aperiodic by Lemma~\ref{lem:irreducible}). Proposition~\ref{prop:retractedwalkspectralradius} and Corollary~\ref{cor:spectralradiusbuilding} give $\vrho(\overline{P})=\vrho( P)<1$

The retraction induces a probability measure on the space of trajectories $\cT$
in the underlying Coxeter group, we also denote this by $\PP$.   Recall that by~(\ref{eq:support}) we have a uniform bound on the next neighbour one-step probabilities: we have $\overline{p}(u,v)>\varepsilon$ for all $u,v\in W$ with $d(u,v)=1$.

For each cone $C(u)$ in $(W,S)$ let $\partial_{\infty} C(u)$ denote the closure
of $C(u)$ at infinity, that is, in the Gromov hyperbolic
compactification. If $u\in W$ has cone type $\bT$ let $\widehat{\mathcal{T}}_u$ be the set of all paths starting at $u$, initially following the path
\begin{align}\label{eq:piu}
\pi_u=\left(\rho(\psi_{x_{\bT}u}(x_{\bT}^{(1)})),
\rho(\psi_{x_{\bT}u}(x_{\bT}^{(2)})),\dots,\rho(\psi_{x_{\bT}u}(y_{\bT}))\right),
\end{align}
and staying in $\mathrm{Int}_{L_1} C(u)$ afterwards.

The invariance properties given in Propositions~\ref{prop:invariance} and \ref{prop:pathbijection} induce the following invariance property for the retracted walk.

\begin{lemma}\label{lem:coneinvariance}
For all $u,v\in W$ with $T(u)=T(v)=\bT$ and all measurable
sets~$A\subseteq\widehat{\mathcal{T}}_u$ we have
$$
 \PP_{u}[(\overline{X}_{n})_{n\geq 0}\in A  ]=\PP_{v}[(\overline{X}_{n})_{n\geq 0}\in vu^{-1} A ].
 $$
\end{lemma}

Since $(\overline X_{n})_{n\geq 0}$ is an irreducible Markov chain on a
hyperbolic graph with bounded range and  spectral
radius~$\vrho(\overline{P})<1$  the Markov chain $(\overline X_{n})_{\geq 0}$
converges almost surely to a random point $\overline{X}_{\infty}$ of the hyperbolic boundary~$\partial W$; since the detailed structure of Gromov hyperbolic boundary is not needed for our purposes, we refer e.g. to \cite[Theorem~22.19]{woessbook} for further details. The harmonic measure~$\nu$ is defined as the law of $\overline{X}_{\infty}$. More precisely, it is the probability measure on the hyperbolic boundary~$\partial W$ such that $\nu(A)=\PP[\overline{X}_{\infty}\in A]$ for each~$A\subset\partial W$.

\begin{lemma}\label{lem:nonatomic}
The harmonic measure~$\nu$ of $(\overline{X}_{n})_{n\geq 0}$ is not concentrated on a finite number of atoms.
\end{lemma}
\begin{proof}
Let us assume that $\nu$ is concentrated on the finite set  $\{\xi_{1}, \xi_{2},\ldots,
\xi_{k}\}\subset \partial W$.  Let $u\in W$ be such that $T(u)=\bT$
and that $\xi_1\in\mathrm{Int}(\partial_{\infty} C(u))$. Then by the definition of harmonic measure,
\begin{equation*}
\PP_{1}[\overline{X}_{\infty}=\xi_{1 }, \overline{X}_{n}\in C(u) \mbox{ for all but finitely many } n]= \nu(\xi_{1})>0.
\end{equation*}
Consequently there exists some $v\in C(u)$ such that
\begin{equation*}
\PP_{v}[\overline{X}_{\infty}=\xi_{1}, \overline{X}_{n}\in C(u) \mbox{ for all } n]>0.
\end{equation*} 
Since there exists a path of positive probability inside $C(u)$ from $u$ to $v$, we have
\begin{equation*}
\PP_{u}[\overline{X}_{\infty}=\xi_{1}, \overline{X}_{n}\in C(u) \mbox{ for all } n]>0.
\end{equation*}
As there are only a finite number of atoms and the automaton $\mathcal{A}(W,S)$ is strongly connected (Theorem~\ref{thm:stronglyconnected}), there exists some $w\in W$ with cone type $T(u)=\bT$ such that $\partial_{\infty}C(w)$ does not contain any of the atoms $\xi_1,\ldots,\xi_k$. However by Lemma~\ref{lem:coneinvariance} we have that there exists some~$\xi_{k+1}\in \partial_{\infty}C(w)$ such that 
\begin{equation*}
\PP_{w}[\overline{X}_{\infty}=\xi_{k+1},\overline{X}_{n}\in C(w) \mbox{ for all } n]>0,
\end{equation*}
and so $\xi_{k+1}$ is an atom, a contradiction. 
\end{proof}

The next lemma will be crucial for our proofs.

\begin{lemma}\label{lem:stayincone}
There exists a
constant $\overline{p}_{esc}>0$ such that for all 
$u\in W$ with $T(u)=\bT$ we have that  
$$
\PP_{u}[(\overline{X}_{n})_{n\geq 0}\in  \widehat{\mathcal{T}}_u]\ge\overline{p}_{esc}.
$$
\end{lemma}

\begin{proof}
First we claim that
\begin{equation}\label{lem:stayincone:eq:prelim}
\PP_{w}[\overline{X}_{n}\in C(w) \text{ for all } n]>0\quad\textrm{for all $w\in W$ with $T(w)=\bT$}.
\end{equation}
Due to strongly connectedness of the automaton $\mathcal{A}(W,S)$ (see
Theorem~\ref{thm:stronglyconnected}) and 
by the definition of recurrent cone types there
exists some $R\geq 0$ such that the sphere $\cS(1,R)$ contains only elements whose cone types are recurrent. Furthermore, by definition we have $W\setminus \cB(1,R-1)=\bigcup_{w'\in
  \cS(1,R)} C(w')$. By Lemma~\ref{lem:nonatomic} the support of~$\nu$ cannot be contained in the set of Gromov boundary points determined by the finitely many geodesics (from Lemma~\ref{lem:boundary-rays}) describing the boundaries of the cones $C(w')$ with $w'\in \cS(1,R)$. Thus there exists some
$v\in \cS(1,R)$ and some  open set $O\subset \partial_{\infty} C(v)$ such that
$\PP[\overline{X}_{\infty}\in O]>0$. 
On the event that $\overline{X}_{\infty}\in O$, at some moment the random walk
$(\overline{X}_n)$ enters $C(v)$ and never leaves it afterwards. If $w\in W$ with $T(w)=\bT$, then the cone $C(w)$ contains an element $v_1$ with $T(v_1)=T(v)$ (since $\bT$ is recurrent and $\mathcal{A}(W,S)$ is strongly connected). By~(\ref{eq:support}) there is positive probability of walking from $w$ to $v_1$ via a shortest path, and necessarily this path is contained in $C(w)$. Hence~(\ref{lem:stayincone:eq:prelim}) is established. 

Let $u\in W$ with $T(u)=\bT$, and let $w=u\delta(x_{\bT},y_{\bT})$. By the construction in Section~\ref{sect:yT} we have $T(w)=\bT$, and $C(w)\subset C(u)$. By~(\ref{eq:support}) there is positive probability that the retracted random walk with $X_0=u$ follows the path $\pi_u$ from (\ref{eq:piu}) initially, and so~(\ref{lem:stayincone:eq:prelim}) implies the Lemma. 
%
 \end{proof}

Recall from Proposition~\ref{prop:pathbijection}.3 that for all $x,y\in\Delta$ with  $T(x)=T(y)$, and all subsets $A\subset \widehat{\mathcal{T}}_{x}$, we have
\begin{equation}\label{eq:coneinvariantbuilding}
\PP_{x}[(X_{n})_{n\geq 0}\in A ] = \PP_{y}[(X_{n})_{n\geq 0}\in \psi_{xy}(A) ].
\end{equation}

\begin{lemma}\label{lem:stayinconebuilding}
There exists some constant $p_{esc}>0$ such that for all $x\in \Delta$
\begin{equation*}
\PP_{x}[(X_{n})_{n\geq 0}\in \widehat{\mathcal{T}}_{x}]\geq p_{esc}.
\end{equation*}
\end{lemma}
\begin{proof}
This is a consequence  of Proposition \ref{prop:Lboundaries} and Lemma \ref{lem:stayincone}.
\end{proof}

\begin{proof}[Proof of Theorem~\ref{thm:reg}.1] This is an adaption of~\cite[Theorem~3.1]{HMM:13}. We sketch the proof and refer to~\cite{HMM:13} for the details. 
The fact that the Cannon automaton~$\mathcal{A}(W,S)$ is strongly connected implies that there exists some $R\in\mathbb{N}$ such
that for all $x\in \Delta$ the ball $\cB(x,R)$ contains at least one chamber
of cone type~$\bT$. This can be extended to prove the following fact.
Denote by $\widehat{\mathcal{T}}^{(n)}_{x}$ the set of $y\in\Delta$ such that there exists a path $(x_{0}, x_{1},x_{2},\ldots) \in \widehat{\mathcal{T}}_{x}$ such that $x_{n}=y$. For $x\in\Delta$, denote the first exit time of $\widehat{\mathcal{T}}_{x}$
by $D_{x}=\inf\{n\geq 1 \mid X_{n}\not\in \widehat{\mathcal{T}}^{(n)}_{x}\}$.
Then there exists some constant $p_{h}$ and some $K$ such that for all choices of $x\in\Delta$ and all $y\in \mathrm{Int}_{L_1} C(x)$ we have
\begin{equation*}
\PP_{x}[ \{T(X_{n})\}_{n=1}^{K}\ni \bT]\ge p_{h} \quad\text{and}\quad \PP_{y}[ \{T(X_{n})\}_{n=1}^{K}\ni \bT\mid D_{x}=\infty]\ge p_{h}
\end{equation*}
Thus wherever the walk is it will reach a chamber $x\in\Delta$ of cone type $\bT$ after at
most $K$ steps with some positive probability of at least $p_{h}$.   By Lemma
\ref{lem:stayinconebuilding}, each time the walk is at a chamber of cone
type~$\bT$  it has a positive probability of at least $p_{esc}$ to follow the
walk $\psi_{x_{\bT}x}(\pi_{\bT})$ and to stay in the $L_1$-interior of this cone forever. If it does, this means that a renewal step was performed, and otherwise, the walk exits this last cone. Now, again the walk will hit a chamber of cone type~$\bT$ in at most $K$ steps with probability at least $p_{h}$ and we continue as above until we eventually performed one renewal step.  Hence by induction, the random times~$R_{n}$ are almost surely finite. 

It is clear that $d(o,X_{ R_{n}})=\sum_{i=1}^{n} d( X_{ R_{i-1}},  X_{ R_{i}})$, because the chambers $(X_{R_n})_{n\geq 0}$ lie in a sequence of nested cones, and so there is a geodesic from $o=X_0$ to $X_{R_n}$ passing through $X_{R_1},X_{R_2},\ldots,X_{R_{n-1}}$. 
The fact that what happens between two subsequent renewal times is independent
is a consequence of the following crucial property. For any $x,y\in\Delta$ with
cone type $\bT$ and any $A\subset \widehat{\mathcal{T}}_{x}$, (\ref{eq:coneinvariantbuilding}) implies that
$$
 \PP_x[ (X_{n})_{n\geq 0}\in A \mid D_{x}=\infty]=\PP_{y} [(X_{n})_{n\geq 0}\in \psi_{xy}(A)  \mid D_{y}=\infty].
 $$
Thus we may introduce a new probability measure: for $A\subset \psi_{xo}(\widehat{\mathcal{T}}_x)$ let
 \[ \QQ_{\bT}[ (X_{n})_{n\geq 0}\in A]= \PP_{x}[ (X_{n})_{n\geq0 }\in \psi_{ox}(A) \mid D_{x}=\infty],
 \] where $x$ is of cone type~$\bT$. Define the $\sigma$-algebras $$\cG_{n}=\sigma(R_{1},\ldots,  R_{n},  X_{0}, \ldots,  X_{ R_{n}}),~n\geq 1.$$
Although the $ R_{n}$'s are not stopping times, a check of the definition of
conditional probability yields the following ``Markov property'': 
for any measurable set $A\subset\psi_{xo}(\widehat{\cT}_{x})$ and any $x\in
\Delta$ of cone type $\bT$,
$$ 
\PP_{x}[ (X_{ R_{n}+k})_{k\geq 0}\in \psi_{oX_{R_{n}}}(A)\mid \cG_{n}] = \QQ_{\bT} [ (X_{k})_{k\geq 0}\in A]
$$
(see \cite[Lemma~3.3]{HMM:13} for details). Thus $d(( X_{ R_{i-1}},  X_{ R_{i}}))_{i\geq 2}$ are i.i.d.~random variables.
\end{proof}

\subsection{Proof of Theorem~\ref{thm:reg}.2}\label{sec:proofreg2}
As the Cayley graph of $(W,S)$ is planar, its Cayley $2$-complex is such that the  one skeleton is given by the Cayley graph and the $2$-cells are bounded by loops. These loops are described by the relations  
$$
s^2=1\quad\textrm{and}\quad (st)^{m_{st}}=1\quad\textrm{for all $s,t\in S$ with $s\neq t$},
$$
where $m_{st}=m_{ts}\in\ZZ_{\geq 2}\cup\{\infty\}$ for all $s\neq t$. Denote by $k$ the maximal length of all finite relations. Then, every loop in the Cayley $2$-complex has length of at most $k$.  At various places, we will use the fact that the $2$-complex is homeomorphic to the hyperbolic disc and can be endowed with an orientation.  

We make use of the following type of connectedness of spheres in Cayley graphs. We give an adaption of the results in  \cite{BaBe:99}  and \cite{Gournay} to our setting. Define the annulus
 \begin{equation*}
\cS^{(k)}(w,K)=\{w'\in W\mid K-k/2\leq d(w,w')\leq K+k/2\}, \quad k,K\in\NN_0, w\in W.
\end{equation*}
\begin{lemma}\label{lem:connected_sphere}
Let $K>k/2$. Then there is a simple cycle in $\cS^{(k)}(w,K)$ that forms a simple closed curve around $w$ in the Cayley $2$-complex.
\end{lemma}
\begin{proof}
By planarity we can order the elements in the sphere $\cS(w,K)$ in clockwise order and say that two elements are neighbors on the sphere if they are neighbors in the ordering. Pick two neighbors $u,v\in \cS(w,K)$.  Since the Cayley graph is one-ended and planar there exists a loop in the Cayley $2$-complex that contains $u,v$. Hence, $d(u,v)\leq k/2$ and thus there exists a path in  $\cS^{(k)}(w,K)$ connecting $u$ and $v$. The concatenation of all paths connecting the neighbors is a cycle $C$ that is contained in  $\cS^{(k)}(w,K)$. Since $K> k/2$ every infinite path starting from $w$ intersects the cycle $C$. It is straightforward to see that  the cycle $C$ contains a simple cycle as a subset  that forms a simple closed curve around $w$ in the Cayley $2$-complex.
\end{proof}

The following Proposition is a key ingredient in the proof of Theorem~\ref{thm:reg}.2. 

\begin{prop}\label{lem:expdecay}
There exist  $C<\infty$ and $\lambda<1$ such that for any $u \in W\setminus \{e\}$ and all $v\in C(u)$
\begin{equation*}
F(u,v)=\PP_{u}[\overline{X}_{n}=v\text{ for some } n]\leq C \lambda^{d(u,v)}.
\end{equation*}
\end{prop}

\begin{proof}
 Fix $K>k/2$ and let $w\in W$ be such that $d(1,w)>K+k/2$. Lemma \ref{lem:connected_sphere} and  the Jordan curve theorem yields now that  $\cS^{(k)}(1,d(1,w))\cap \cS^{(k)}(w,K)\neq \emptyset$. Let $\pi^{+}$ be a geodesic ray from $1$ passing through $w$ and $\pi^{-}$ be a geodesic ray starting  from $1$ and not passing through $\cS^{(k)}(w,K)$. Denote by $\overline\pi$ the bi-infinite path consisting of geodesics $\pi^{+}$ and $\pi^{-}$. There are at least two points $v_{1},v_{2}\in \cS^{(k)}(1,d(1,w))\cap \cS^{(k)}(w,K)$  on different sides of $\overline\pi$, i.e.~each path between $v_{1}$ and $v_{2}$ has to cross $\pi^{-}$ or $\pi^{+}$. We will make use of this fact later in the proof.

Let $\gamma$ be a geodesic from $1$ to $v$ passing through  $u$, and let $d=d(u,v)$.
Let $D\in\mathbb{N}$ be such that $D>2e(0)+2k$ and $e(D)>4e(0)+4k$. For $i\in\{1,\ldots,\lfloor d/D\rfloor\}$ denote by  $u_{i}=\gamma(d(1,u)+i\cdot D)$ and let $\cB(u_{i},2e(0)+2k)$  the ball of radius $2e(0)+2k$ around $u_{i}$.  We define
\begin{equation*}
B_{i}=\bigcup_{w\in \cB(u_{i},2e(0)+2k)  } C(w) 
\end{equation*}
The boundary of $B_i$ is given by $\partial B_i=\{w\in B_i \mid \exists w'\in
W\setminus B_i: w'\sim w\}$.

\noindent\underline{Claim:} The boundaries $\partial B_i$'s are disjoint.

\noindent\emph{Proof of the claim.}
The choice of  $D > 4e(0)+4k$  implies that the balls  $\cB(u_{i},2e(0)+2k)$
are disjoint. Let us assume that there exists some $w\in \partial
B_i\cap \partial B_{i+1}$ for some $i$ and show that this yields a
contradiction. Denote by $\gamma^{+}$ a geodesic continuation of $\gamma$ that
is contained in all $B_{i}$'s. Let $\gamma^{-}$ be a geodesic ray emanating from $1$ and not intersecting the $\partial B_{i}$'s, and let $\overline \gamma$ the bi-infinite path consisting of the elements of $\gamma^{+}$ and $\gamma^{-}$. 
Let $\gamma_{1}$ be the geodesic from $1$ to $w$ that maximizes $d(\gamma_{1}(t),\gamma^{+}(t))$ for all $t\leq d(1,w)$ among all geodesics from $1$ to $w$; this construction is well-defined since geodesics can only cross in vertices. 

Our next step is to show that $\gamma_{1}$ is sufficiently ``far away'' from $u_{i}$.
Due to planarity, there exists a geodesic ray $\gamma_{1}'$ starting from $1$, that passes through $\cS^{(k)}(1,d(1,u_{i}))\cap \cS^{(k)}(u_{i},2 e(0)+3k/2)$, and is between $\gamma^{+}$ and $\gamma_{1}$. 
To see this, take any $w_{1}'\in \cS^{(k)}(1,d(1,u_{i}))\cap \cS^{(k)}(u_{i},2 e(0)+3k/2)$ on the same side of $\overline \gamma$ as $\gamma_{1}$. By maximality of $\gamma_{1}$ the vertex $w_{1}'$ has to lie between $\gamma_{1}$ and $\gamma^{+}$.  First, choose a geodesic from $1$ to $w_{1}$ that lies in between $\gamma^{+}$ and $\gamma_{1}$. Now, augment this geodesic step by step (following some way in the automaton) till forever or until one  hits $\gamma_{1}$ in which case we follow $\gamma_{1}$ afterwards as long as we can, and then continue to follow some path in the automaton.
Since $d(1,u_{1})-k/2\leq d(1,w'_{1})\leq d(1,u_{1})+k/2$ and $ d(u_{1},w'_{1})\geq 2e(0)+k$ we have that  $ d( \gamma_{1}'(d(1,u_{i})),u_{i})\geq 2 e(0)+k/2 $.
The maximality of $\gamma_{1}$ implies now that
\begin{equation}\label{eq:lem:expdecay}
 d(\gamma_{1}(d(1,u_{i})),u_{i})\geq 2 e(0)+k/2.
\end{equation}

Since $w\in \partial B_{i+1}$ there exists a  geodesic $\gamma_{2}$  from $1$ to $w$ that passes through $\cB(u_{i},2e(0)+2k)$.  Denote by $v$ a point in the intersection of $\gamma_{2}$ and $\cB(u_{i+1},2e(0)+2k)$. We have
\begin{equation*}
d(1,u_{i+1})-2e(0)-2k \leq d(1, v)\leq d(1,u_{i+1})+2e(0)+2k \mbox{ and } d(v,u_{i+1})\leq 2e(0)+2k.
\end{equation*}
Therefore,
\begin{equation*}
d(\gamma_{2}(d(1,u_{i+1}), \gamma(d(1,u_{i+1})))\leq 4e(0)+4k.
\end{equation*}
Eventually, by the exponential divergence of geodesics and since $e(D)>4e(0)+4k$ we have that $d(\gamma(d(1,u_{i}), \gamma_{2}(d(1,u_{i})))<e(0)$. 
 Inequality (\ref{eq:lem:expdecay}) implies now that 
\begin{equation*}
d(\gamma_{1}(d(1,u_{i}), \gamma_{2}(d(1,u_{i})))>e(0),
\end{equation*}
and hence $\gamma_{1}$ and $\gamma_{2}$ diverge which contradicts the fact that they intersect in $w$. This proves the claim.

We define the  stopping times
\begin{equation*}
\tau_{i}=\inf\{n\geq 0: \overline{X}_{n}\in B_{i}\}.
\end{equation*}
In order to walk from $u$ to $v$ the walk has to enter each $B_{i}$, and so we find that
\begin{equation*}
F(u,v)\leq \PP_{u}[\tau_{1}<\infty, \tau_2<\infty,\ldots, \tau_{\lfloor d/D\rfloor}<\infty].
\end{equation*}
Our proof strategy is now to prove that 
\begin{equation}\label{eq:lem:decay}
\PP[\tau_{i+1}=\infty \mid \tau_{i}<\infty]\geq c\quad \textrm{for all } i\in\{1,\ldots,\lfloor d/D\rfloor\}
\end{equation}
for some constant $c>0$. An application of the strong Markov property  yields then  that 
\begin{equation*}
F(u,v)\leq \PP_{u}[\tau_{1}<\infty, \tau_2<\infty,\ldots, \tau_{\lfloor d/D\rfloor}<\infty]\leq (1-c)^{\lfloor d/D\rfloor},
\end{equation*}
which proves the claim. Therefore it remains to prove (\ref{eq:lem:decay}).
Assume $\tau_{i}<\infty$ and denote $w=\overline{X}_{\tau_{i}}$. 
Due to the connectivity of spheres there exists some $w_{1}$  such that $e(0)+k/2\leq d(w,w_{1})\leq e(0)+3k/2$, $d(1,w)-{k/2}\leq d(1, w_{1})\leq d(1, w)+k/2$ and $w_{1}\notin B_{i}$. Now, due to the fact that geodesics either stay at bounded distance at most $e(0)$ or diverge, we find that $C(w_{1})$ does not intersect $B_{i+1}$. Hence, a walk started in $w_{1}$ will stay with positive probability of at least $\overline{p}_{esc}$ in $C(w_{1})$ and therefore will never visit $B_{i+1}$. As the probability that a walk started in $w$ will visit $w_{1}$ is bounded below by $\varepsilon^{e(0)+3k/2}$  we obtain~(\ref{eq:lem:decay}) with $c=\varepsilon^{e(0)+3k/2}\overline{p}_{esc}$.
\end{proof}

For each $u\in W$ let
$$
D_{u}=\inf\{n\geq 1\mid \overline{X}_{n}\not\in \rho(\widehat{\mathcal{T}}^{(n)}_{x})\}.
$$ 
It is crucial to bound the moments of $ D_u$. In the following $\EE_v$ denotes the expectation given that $\overline{X}_0= v$, $v\in W$.

\begin{lemma}\label{lem:expmomentsDretract}
There exist constants $\lambda_{ D},K_{ D}$ such that for all  $u\in W$ with $T(u)=\bT$ we have 
\begin{equation*}
\EE_{v}[\exp(\lambda_{ D}  D_{u}) \1_{\{ D_{u}<\infty\}}]\leq K_{ D}\quad\text{for all $v\in C(u)$}.
\end{equation*}
\end{lemma}
\begin{proof}
Since $p(v,z)\geq \varepsilon$ whenever $d(v, z)=1$,
\cite[Lemma~8.1]{woessbook} guarantees the existence of a constant $A$ such that for all $v,z\in W$ we have
\begin{equation}\label{eq:Arho}
\bar p^{(n)}(v,z)\leq A^{d(v,z)} \vrho(\bar P)^{n}.
\end{equation}
We proceed with the tails of $\PP_{v}[ D_{u}=n]$ for $u$ such that $T(u)=\bT$ and $v\in C(u)$. Let $\delta>0$ to be chosen later, then 
\begin{equation}\label{eq:lem:momentbound:1}
\PP_{v}[ D_{u}=n+1]\leq \PP_{v}[d(v,\overline{X}_{n})\leq \delta n,   D_{u}=n+1]+ \PP_{v}[d(v,\overline{X}_{n})\geq \delta n,   D_{u}=n+1].\end{equation}
We will make use of the fact that $ D_{u}=n+1$ implies $\overline X_{n} \in \partial_{2L_1} C(u)$ for $n$ sufficiently large.
Due to the planarity of $(W, S)$ we have that 
$ c(n)=|\partial_{2L_1} C(u)\cap \cB(v,\delta n)|$ grows linearly in $n$. 
The first summand in inequality (\ref{eq:lem:momentbound:1}) can be bounded as follows:
\begin{eqnarray*}
\PP_{v}\left[d(v,\overline{X}_{n}) \leq \delta n,  D_{u}=n+1\right] 
  &\leq &  \PP_{v}[ d(v,\overline{X}_{n}) \leq \delta n, \overline{X}_{n}\in \partial_{2L_1} C(u)] \cr 
  &\leq & \sum_{z\in  \partial_{2L_1} C(u)\cap \cB(v,\delta n)} \bar p^{(n)}(v,z)\cr
  &\leq & \sum_{z\in  \partial_{2L_1} C(u)\cap \cB(v,\delta n)}  A^{d(v,z)} \vrho(\bar P)^{n}\cr
  & \leq& c(n) \max\{1,A^{\delta n}\} \vrho(\bar P)^{n}.
 \end{eqnarray*}
Choose $\delta>0$ sufficiently small so that the latter sum  decays exponentially in $n$.

For the second summand in~(\ref{eq:lem:momentbound:1}) we have
\begin{equation*}
\PP_{v}[d(v,\overline{X}_{n})\geq \delta n,   D_{u}=n+1] \leq \sum_{z\in
 \partial_{2L_1} C(u)\setminus {\cB(v,\delta n)}} F(v,z)
 \leq      \sum_{k=\lceil \delta n \rceil}^{\infty} c(k) C \lambda^{k},
\end{equation*}
which decays exponentially in $n$. The result follows.
\end{proof}

It turns out that in order to give a good estimate  on the length of the time intervals between renewal times, it suffices to control the tails of $D_{x}=\inf\{n\geq 1 \mid X_{n}\not\in \widehat{\mathcal{T}}^{(n)}_{x}\}$ on the event that $D_{x}$ is finite.     However, this can be achieved by  comparison with the retracted walk. 

\begin{lemma}\label{lem:expmomentsD}
There exists constants $\lambda_{D}',K_{D}'$ such that for all  $x$ of type $\bT$ we have 
\begin{equation*}
\EE_{y}[\exp(\lambda_{D}' D_{x}) \1_{\{D_{x}<\infty\}}]\leq K_{D}'\quad\text{for all $y\in C(x)$}.
\end{equation*}
\end{lemma}
\begin{proof}
 Let $u=\rho(x)$, then due to Proposition \ref{prop:Lboundaries} and Lemma \ref{lem:equal-distance} we have
\begin{eqnarray*}
\{ D_{x}=k\}&=& \{X_{k}\notin \widehat{\mathcal{T}}^{(k)}_{x}, \forall m<k: X_{m}\in  \widehat{\mathcal{T}}^{(m)}_{x}) \}\cr
&\subseteq&  \{\overline{X}_{k}\notin \rho(\widehat{\mathcal{T}}^{(k)}_{x}), \forall m<k:  \overline{X}_{m}\in \rho(\widehat{\mathcal{T}}^{(m)}_{x}) \}\cr
&=& \{D_{u}=k\}.
\end{eqnarray*}
Hence, for all $y\in C(x)$ and $v=\rho(y)$ we have that $\PP_{y}[
 D_{x}=k]\leq \PP_{v}[ D_{u}=k]$ and the claim follows with Lemma \ref{lem:expmomentsDretract}.
\end{proof}

\begin{proof}[Proof of Theorem~\ref{thm:reg}.2]
The essential ingredients are Lemmata~\ref{lem:stayinconebuilding} and~\ref{lem:expmomentsD} and the fact that Cannon automaton is strongly connected.  Since the proof is analogous to the proof of   \cite[Lemma 4.1]{HMM:13} we only give a sketch of the arguments here. In fact, the proof is a more quantitative analysis of the arguments given in the proof of Theorem~\ref{thm:reg}.1. Recall that  wherever the walk is, it will reach a chamber  of cone type $\bT$ after at
most $K$ steps with probability of at least $p_{h}$.   By Lemma
\ref{lem:stayinconebuilding}, each time the walk is at a chamber (in the $L_{1}$-interior) of cone type~$\bT$  it has a positive probability of  at least $p_{esc}$ to perform a renewal step. If it does not perform a renewal step, it takes the walk a random time $D$ to exit  $\widehat{\mathcal{T}}$. Now, again the walk will hit a chamber of cone type~$\bT$ in at most $K$ steps with probability of at least $p_{h}$ and we continue as above until we eventually performed one renewal step. The time until the walk does a renewal step can therefore be bounded by $\sum_{i=1}^{G} D_{i}$ where the $D_{i}$ are i.i.d.~copies of $D$ and $G$ is a geometric random variable (independent of the $D_{i}$'s) with success probability $p_{h}p_{esc}$.
Since the $G$ and the $D_{i}$'s have exponential moments one proves  that  $\sum_{i=1}^{G} D_{i}$ has exponential moments, too.
\end{proof}

\begin{appendix}
\section{Automata and the proof of Theorem~\ref{thm:stronglyconnected}}\label{app:A}

The automatic structure of Coxeter groups was first proven by Brink and Howlett~\cite{howlett} (see also \cite[Chapter 4]{bjorner}). In this section we explicitly construct the Cannon automaton for each Fuchsian Coxeter system, and deduce that these automata are strongly connected (hence proving Theorem~\ref{thm:stronglyconnected}). We also envisage that our explicit description of the automata will be useful for future work where precise information regarding the automata is required. 

It is convenient to divide the set of all Fuchsian Coxeter systems into $4$ classes. First consider triangle groups. Let $(W,S)$ be a triangle group generated by $s,t,u$. Let $a=m_{st}$, $b=m_{tu}$, and $c=m_{us}$, and rename the generators if necessary so that $a\geq b\geq c\geq 2$. Then $W$ is infinite if and only if (see Example~\ref{ex:21})
\begin{align*}
(a,b,c)\in\{(k_1,k_2,k_3),(k_4,k_5,2),(k_6,3,2)\mid k_1\geq k_2\geq k_3\geq 3,\, k_4\geq k_5\geq 4,\, k_6\geq 6\}.
\end{align*}
We partition the infinite triangle groups into $3$ classes:
\begin{itemize}
\item Class I consists of those groups with $a\geq b\geq c\geq 3$.
\item Class II consists of those groups with $a\geq b\geq 4$ and $c=2$.
\item Class III consists of those groups with $a\geq 6$, $b=3$, and $c=2$.
\end{itemize}
Note that the ``root'' of each class (that is, the group with $a+b+c$ minimal) is an affine triangle group: $(a,b,c)=(3,3,3),(4,4,2),(6,3,2)$. All other infinite triangle groups are Fuchsian. For no extra work we will include the affine triangle groups in our considerations. 

The remaining Fuchsian Coxeter systems are those with $|S|\geq 4$. We call these Fuchsian Coxeter systems of Class IV. 

\subsection{Preliminaries}

Before constructing the automata, we give some general background (see \cite{humphreys} for details). Let $(W,S)$ be a Fuchsian Coxeter complex (or an affine triangle group). The conjugates of the generators $S$ are called \textit{reflections}. Thus the set of all reflections is
$
R=\{wsw^{-1}\mid s\in S,w\in W\}.
$
Each reflection $r\in R$ determines a \textit{wall} (also called a \textit{hyperplane}) $H_r=\{\zeta\in\mathbb{H}^2\mid r\zeta=\zeta\}$ in~$\mathbb{H}^2$ (or in $\mathbb{R}^2$ for Euclidean triangle groups). 

Let $\mathcal{H}=\{H_r\mid r\in R\}$ be the set of all walls. Given a wall $H\in \mathcal{H}$ we write $s_H$ for the reflection in the wall~$H$. Thus $s_H=r$ if $H=H_r$. More generally, if $H$ is the wall separating $w$ from $ws$ then $s_H=wsw^{-1}$. 

Each wall $H\in\mathcal{H}$ determines two (closed) \textit{half-spaces} of the hyperbolic disc~$\HH^2$. The \textit{positive side} of $H$ is the half-space $H^+$ which contains the identity chamber~$1$, and the \textit{negative side} of $H$ is the half space $H^-$ which does not contain~$1$. Alternatively, we have
\begin{align}\label{eq:halfspaces}
H^+=\{w\in W\mid \ell(s_Hw)>\ell(w)\}\quad\textrm{and}\quad H^-=\{w\in W\mid \ell(s_Hw)<\ell(w)\}.
\end{align}

If $w=s_1\cdots s_{\ell}$ is a reduced expression then for each $1\leq k\leq \ell$ the element $w$ is on the negative side of each wall~$H_{r_k}$, where $r_k$ is the reflection $r_k=s_1\cdots s_{k-1}s_ks_{k-1}\cdots s_1$, and  $\{H_{r_k}\mid k=1,\ldots,\ell\}$ is precisely the set of all walls separating~$1$ from~$w$.

\begin{lemma}\label{lem:intersect}
Walls $H,H'$ of a Fuchsian Coxeter system (or affine triangle group) intersect if and only if the corresponding reflections $s_H$ and $s_{H'}$ generate a finite group.
\end{lemma}

\begin{proof}
Suppose that the walls $H$ and $H'$ intersect. If $H=H'$ then $s_H$ generates a group of order~$2$, and if $H\neq H'$ then $H$ and $H'$ intersect at a single point $x\in \mathbb{H}^2$ (or $x\in\mathbb{R}^2$). By construction of the realisation there are only finitely many walls through~$x$, and the group generated by the reflections in these walls is a conjugate of a finite standard parabolic subgroup. The group generated by $s_H$ and $s_{H'}$ is a subgroup of this finite group, and is thus finite.

On the other hand, if the reflections $s_H$ and $s_{H'}$ generate a finite group then by \cite[Proposition~2.87]{AB} the reflections $s_H$ and $s_{H'}$ both lie in a conjugate of a finite parabolic subgroup. Therefore the walls $H$ and $H'$ intersect. 
\end{proof}

The \textit{left descent set} of $w\in W$ is
$$
L(w)=\{s\in S\mid \ell(sw)=\ell(w)-1\}.
$$
Equivalently, $L(w)$ is the set of generators $s\in S$ for which there is a reduced expression for $w$ starting with the letter~$s$. By \cite[Corollary~2.18]{AB} the subgroup of $W$ generated by $L(w)$ is finite for each fixed $w\in W$. Moreover, if $v$ is an element of the group generated by $L(w)$ then by \cite[Proposition~2.17]{AB} there exists an expression
\begin{align}\label{eq:startexpression}
w=vw'\quad\textrm{with $\ell(w)=\ell(v)+\ell(w')$}.
\end{align}

\subsection{Class I}

\begin{lemma}\label{lem:reduction} Let $(W,S)$ be a Coxeter system and let $s,t,u\in S$ be distinct generators. If $m_{st},m_{tu},m_{us}\geq 3$ then the subgroup of $W$ generated by $u$ and $tst$ is infinite.
\end{lemma}

\begin{proof}
Consider the word $w_n=(tstu)^n=tstutstutstu\cdots tstu$. If $m_{st},m_{tu}>3$ then $w_n$ has no available Coxeter moves, and is thus reduced, and so the subgroup generated by $u$ and $tst$ is infinite.
In the cases that one or both of $m_{st}$ and $m_{tu}$ are $3$ then there are
Coxeter moves available, although it is not hard to see that no reduction in the length of $w_n$ is possible, and so the subgroup is finite in these cases too.
\end{proof}

\begin{lemma}\label{lem:automaton1}
Let $(W,S)$ be a triangle group with $S=\{s,t,u\}$ and $3\leq
m_{st},m_{tu},m_{su}<\infty$. Let $x$ be the longest element of $W_{st}$. Let
$v\in W_{st}$ with $v\notin \{x,1\}$, and let $v=s_{\ell}\cdots s_1$ be the unique reduced expression for~$v$. Then
\begin{enumerate}
\item $T(vu)=T(s_1u)$,
\item $T(xus)=T(sus)$ and $T(xut)=T(tut)$.
\end{enumerate}
\end{lemma}

\begin{proof}
1. By symmetry we may suppose that $s_1=s$. We are required to show that for each fixed $w\in W$,
$$
\textrm{$vuw$ is reduced}\quad\textrm{if and only if}\quad \textrm{$suw$ is reduced}.
$$
It is clear that if $vuw$ is reduced then $suw$ is also reduced (since truncations of reduced expressions are reduced). Suppose, for a contradiction, that $suw$ is reduced and that $vuw$ is not reduced. Let $2\leq k\leq \ell$ be minimal subject to $s_{k-1}\cdots s_1uw$ is reduced and $s_k\cdots s_1uw$ is not reduced. Since $\{s_{k-1},s_k\}=\{s,t\}$ we see that $s,t\in L(s_{k-1}\cdots s_1uw)$. Thus by (\ref{eq:startexpression}) there is a reduced expression for $s_{k-1}\cdots s_1uw$ starting with any chosen reduced word in~$W_{st}$. Since $\ell(s_{k-1}\cdots s_1)\leq m_{st}-2$ the word $s_{k-1}\cdots s_1ts\in W_{st}$ is reduced, and thus $s_{k-1}\cdots s_1uw=s_{k-1}\cdots s_1tsv'$ for some $v'\in W$ (with the expressions on both sides being reduced) and so $uw=tsv'$ (again with both expressions reduced). Therefore the element $uw=tsv'$ lies on the negative side of the wall $H$ separating $1$ from $u$, and on the negative side of the wall $H'$ separating $t$ from $ts$, and so $H^-\cup(H')^-\neq\emptyset$. The reflection in the hyperplane $H$ is $s_H=u$, and the reflection in the hyperplane $H'$ is $s_{H'}=tut$. Either $H^-\subseteq (H')^-$, or $(H')^-\subseteq H^-$, or $H$ and $H'$ intersect. If $H^-\subseteq (H')^-$ then $u\in (H')^-$, and so $\ell(s_{H'}u)<\ell(u)=1$. But $s_{H'}u=tstu$ has length~$4$. Similarly, if $(H')^-\subseteq H^-$ then $ts\in H^-$ and so $\ell(s_Hts)<\ell(ts)=2$, but $s_Hts=uts$ has length~$3$. Therefore $H$ and $H'$ intersect, and so by Lemma~\ref{lem:intersect} the group generated by $s_H=u$ and $s_{H'}=tst$ is finite, contradicting Lemma~\ref{lem:reduction}. 

2. We show that $T(xus)=T(sus)$. We are required to show that for each fixed $w\in W$,
$$
\textrm{$xusw$ is reduced}\quad\textrm{if and only if}\quad \textrm{$susw$ is reduced}.
$$
Choose the reduced expression $x=s_{\ell}\cdots s_1$ with $s_1=s$. Thus the expression $susw$ can be regarded as a truncation of the expression $xusw$, and so it follows that if the latter is reduced then the former is also reduced. Suppose, for a contradiction, that $susw$ is reduced and that $xusw$ is not reduced. Let $2\leq k\leq \ell$ be minimal subject to $s_{k-1}\cdots s_1usw$ is reduced and $s_k\cdots s_1usw$ is not reduced. Then $\{s,t\}\in L(s_{k-1}\cdots s_1usw)$ and so there is a reduced expression $s_{k-1}\cdots s_1usw=s_{k-1}\cdots s_1tv'$ for some $v'\in W$, and so $usw=tv'$ with both sides reduced. Hence, as above, the group generated by $usu$ and $t$ is finite, contradicting Lemma~\ref{lem:reduction}.
\end{proof}

Lemma~\ref{lem:automaton1} completely determined the Cannon automaton for all triangle groups in Class~I. An illustration is given in Figure~\ref{fig:I}.

\begin{figure}[!h]
\begin{minipage}{0.65\textwidth}
\centering
\begin{tikzpicture} [scale=6]
\draw [->,red,-triangle 45] (0,0) -- (0,0.25); 
\draw [->,Green,-triangle 45] (0,0) -- ({0.25*cos(-30)},{0.25*sin(-30)});
\draw [->,blue,-triangle 45] (0,0) -- ({0.25*cos(210)},{0.25*sin(210)});
\draw [->,domain=135:170,Green,-triangle 45] plot ({0.3*cos(\x)},{0.3*sin(\x)-0.335});
\draw [->,domain=170:205,blue,-triangle 45] plot ({0.3*cos(\x)},{0.3*sin(\x)-0.335});
\draw [->,domain=205:240,Green,-triangle 45] plot ({0.3*cos(\x)},{0.3*sin(\x)-0.335});
\draw [->,gray,dashed,domain=240:270,-triangle 45] plot ({0.3*cos(\x)},{0.3*sin(\x)-0.335});
\draw [->,domain=45:10,blue,-triangle 45] plot ({0.3*cos(\x)},{0.3*sin(\x)-0.335});
\draw [->,domain=10:-25,Green,-triangle 45] plot ({0.3*cos(\x)},{0.3*sin(\x)-0.335});
\draw [->,domain=-25:-60,blue,-triangle 45] plot ({0.3*cos(\x)},{0.3*sin(\x)-0.335});
\draw [->,gray,dashed,domain=-60:-90,-triangle 45] plot ({0.3*cos(\x)},{0.3*sin(\x)-0.335});
\draw [->,red,-triangle 45] ({0.3*cos(-90)},{0.3*sin(-90)-0.335}) -- ({0.3*cos(-90)},{0.3*sin(-90)-0.335-0.16});
\draw [->,Green,-triangle 45] ({0.3*cos(-90)},{0.3*sin(-90)-0.335-0.16}) -- ({0.3*cos(-90)-0.16},{0.3*sin(-90)-0.335-0.16});
\draw [->,blue,-triangle 45] ({0.3*cos(-90)},{0.3*sin(-90)-0.335-0.16}) -- ({0.3*cos(-90)+0.16},{0.3*sin(-90)-0.335-0.16});
\draw [->,red,-triangle 45] ({0.3*cos(170)},{0.3*sin(170)-0.335}) -- ({0.2*cos(170)},{0.2*sin(170)-0.335});
\draw [->,red,-triangle 45] ({0.3*cos(205)},{0.3*sin(205)-0.335}) -- ({0.2*cos(205)},{0.2*sin(205)-0.335});
\draw [->,red,-triangle 45] ({0.3*cos(240)},{0.3*sin(240)-0.335}) -- ({0.2*cos(240)},{0.2*sin(240)-0.335});
\draw [->,red,-triangle 45] ({0.3*cos(10)},{0.3*sin(10)-0.335}) -- ({0.2*cos(10)},{0.2*sin(10)-0.335});
\draw [->,red,-triangle 45] ({0.3*cos(-25)},{0.3*sin(-25)-0.335}) -- ({0.2*cos(-25)},{0.2*sin(-25)-0.335});
\draw [->,red,-triangle 45] ({0.3*cos(-60)},{0.3*sin(-60)-0.335}) -- ({0.2*cos(-60)},{0.2*sin(-60)-0.335});
\path (0,-0.05) node {\footnotesize{$\emptyset$}}
 ({0.25*cos(-30)-0.02},{0.25*sin(-30)-0.04}) node {\footnotesize{$1$}}
  ({-0.25*cos(-30)+0.02},{0.25*sin(-30)-0.04}) node {\footnotesize{$2$}}
  (-0.04,0.23) node {\footnotesize{$3$}}
   ({0.3*cos(10)+0.05},{0.3*sin(10)-0.335}) node {\footnotesize{$12$}}
    ({0.3*cos(-25)+0.06},{0.3*sin(-25)-0.335}) node {\footnotesize{$121$}}
   ({0.3*cos(-60)+0.03},{0.3*sin(-60)-0.335-0.06}) node {\footnotesize{$1212$}}
    ({-0.3*cos(10)-0.05},{0.3*sin(10)-0.335}) node {\footnotesize{$21$}}
    ({-0.3*cos(-25)-0.06},{0.3*sin(-25)-0.335}) node {\footnotesize{$212$}}
    ({-0.3*cos(-60)-0.03},{0.3*sin(-60)-0.335-0.06}) node {\footnotesize{$2121$}}
    ({-0.3*cos(-90)},{0.3*sin(-90)-0.335+0.05}) node {\footnotesize{$x$}}
      ({-0.3*cos(-90)},{0.3*sin(-90)-0.335+0.05-0.25}) node {\footnotesize{$x3$}}
        ({-0.3*cos(-90)+0.21},{0.3*sin(-90)-0.335-0.16}) node [gray]{\footnotesize{$232$}}
        ({-0.3*cos(-90)-0.21},{0.3*sin(-90)-0.335-0.16}) node [gray]{\footnotesize{$131$}}
       ({0.2*cos(10)-0.04},{0.2*sin(10)-0.335}) node [gray]{\footnotesize{$23$}}
    ({0.2*cos(-25)-0.04},{0.2*sin(-25)-0.335}) node [gray]{\footnotesize{$13$}}
    ({0.2*cos(-60)-0.04+0.01},{0.2*sin(-60)-0.335+0.025}) node [gray]{\footnotesize{$23$}}
    ({-0.2*cos(10)+0.04},{0.2*sin(10)-0.335}) node [gray]{\footnotesize{$13$}}
    ({-0.2*cos(-25)+0.04},{0.2*sin(-25)-0.335}) node [gray]{\footnotesize{$23$}}
    ({-0.2*cos(-60)+0.04-0.01},{0.2*sin(-60)-0.335+0.025}) node [gray] {\footnotesize{$13$}};
\begin{scope}[rotate=120]
\draw [->,domain=135:170,red,-triangle 45] plot ({0.3*cos(\x)},{0.3*sin(\x)-0.335});
\draw [->,domain=170:205,Green,-triangle 45] plot ({0.3*cos(\x)},{0.3*sin(\x)-0.335});
\draw [->,domain=205:240,red,-triangle 45] plot ({0.3*cos(\x)},{0.3*sin(\x)-0.335});
\draw [->,gray,dashed,domain=240:270,-triangle 45] plot ({0.3*cos(\x)},{0.3*sin(\x)-0.335});
\draw [->,domain=45:10,Green,-triangle 45] plot ({0.3*cos(\x)},{0.3*sin(\x)-0.335});
\draw [->,domain=10:-25,red,-triangle 45] plot ({0.3*cos(\x)},{0.3*sin(\x)-0.335});
\draw [->,domain=-25:-60,Green,-triangle 45] plot ({0.3*cos(\x)},{0.3*sin(\x)-0.335});
\draw [->,gray,dashed,domain=-60:-90,-triangle 45] plot ({0.3*cos(\x)},{0.3*sin(\x)-0.335});
\draw [->,blue,-triangle 45] ({0.3*cos(-90)},{0.3*sin(-90)-0.335}) -- ({0.3*cos(-90)},{0.3*sin(-90)-0.335-0.16});
\draw [->,red,-triangle 45] ({0.3*cos(-90)},{0.3*sin(-90)-0.335-0.16}) -- ({0.3*cos(-90)-0.16},{0.3*sin(-90)-0.335-0.16});
\draw [->,Green,-triangle 45] ({0.3*cos(-90)},{0.3*sin(-90)-0.335-0.16}) -- ({0.3*cos(-90)+0.16},{0.3*sin(-90)-0.335-0.16});
\draw [->,blue,-triangle 45] ({0.3*cos(170)},{0.3*sin(170)-0.335}) -- ({0.2*cos(170)},{0.2*sin(170)-0.335});
\draw [->,blue,-triangle 45] ({0.3*cos(205)},{0.3*sin(205)-0.335}) -- ({0.2*cos(205)},{0.2*sin(205)-0.335});
\draw [->,blue,-triangle 45] ({0.3*cos(240)},{0.3*sin(240)-0.335}) -- ({0.2*cos(240)},{0.2*sin(240)-0.335});
\draw [->,blue,-triangle 45] ({0.3*cos(10)},{0.3*sin(10)-0.335}) -- ({0.2*cos(10)},{0.2*sin(10)-0.335});
\draw [->,blue,-triangle 45] ({0.3*cos(-25)},{0.3*sin(-25)-0.335}) -- ({0.2*cos(-25)},{0.2*sin(-25)-0.335});
\draw [->,blue,-triangle 45] ({0.3*cos(-60)},{0.3*sin(-60)-0.335}) -- ({0.2*cos(-60)},{0.2*sin(-60)-0.335});
\path 
   ({0.3*cos(10)+0.05},{0.3*sin(10)-0.335}) node {\footnotesize{$31$}}
    ({0.3*cos(-25)+0.06},{0.3*sin(-25)-0.335}) node {\footnotesize{$313$}}
    ({0.3*cos(-60)+0.03},{0.3*sin(-60)-0.335-0.06}) node {\footnotesize{$3131$}}
    ({-0.3*cos(10)-0.05},{0.3*sin(10)-0.335}) node {\footnotesize{$13$}}
    ({-0.3*cos(-25)-0.06},{0.3*sin(-25)-0.335}) node {\footnotesize{$131$}}
    ({-0.3*cos(-60)-0.03},{0.3*sin(-60)-0.335-0.06}) node {\footnotesize{$1313$}}
    ({-0.3*cos(-90)},{0.3*sin(-90)-0.335+0.05}) node {\footnotesize{$z$}}
      ({-0.3*cos(-90)},{0.3*sin(-90)-0.335+0.05-0.25}) node {\footnotesize{$z2$}}
        ({-0.3*cos(-90)+0.21},{0.3*sin(-90)-0.335-0.16}) node [gray] {\footnotesize{$121$}}
        ({-0.3*cos(-90)-0.21},{0.3*sin(-90)-0.335-0.16}) node [gray] {\footnotesize{$323$}}
       ({0.2*cos(10)-0.04},{0.2*sin(10)-0.335}) node [gray] {\footnotesize{$12$}}
    ({0.2*cos(-25)-0.04},{0.2*sin(-25)-0.335}) node [gray] {\footnotesize{$32$}}
    ({0.2*cos(-60)-0.04+0.01},{0.2*sin(-60)-0.335+0.025})  node [gray] {\footnotesize{$12$}}
    ({-0.2*cos(10)+0.04},{0.2*sin(10)-0.335}) node [gray]{\footnotesize{$32$}}
    ({-0.2*cos(-25)+0.04},{0.2*sin(-25)-0.335}) node [gray] {\footnotesize{$12$}}
    ({-0.2*cos(-60)+0.04-0.01},{0.2*sin(-60)-0.335+0.025})  node [gray] {\footnotesize{$32$}};
\end{scope}
\begin{scope}[rotate=-120]
\draw [->,domain=135:170,blue,-triangle 45] plot ({0.3*cos(\x)},{0.3*sin(\x)-0.335});
\draw [->,domain=170:205,red,-triangle 45] plot ({0.3*cos(\x)},{0.3*sin(\x)-0.335});
\draw [->,domain=205:240,blue,-triangle 45] plot ({0.3*cos(\x)},{0.3*sin(\x)-0.335});
\draw [->,gray,dashed,domain=240:270,-triangle 45] plot ({0.3*cos(\x)},{0.3*sin(\x)-0.335});
\draw [->,domain=45:10,red,-triangle 45] plot ({0.3*cos(\x)},{0.3*sin(\x)-0.335});
\draw [->,domain=10:-25,blue,-triangle 45] plot ({0.3*cos(\x)},{0.3*sin(\x)-0.335});
\draw [->,domain=-25:-60,red,-triangle 45] plot ({0.3*cos(\x)},{0.3*sin(\x)-0.335});
\draw [->,gray,dashed,domain=-60:-90,-triangle 45] plot ({0.3*cos(\x)},{0.3*sin(\x)-0.335});
\draw [->,Green,-triangle 45] ({0.3*cos(-90)},{0.3*sin(-90)-0.335}) -- ({0.3*cos(-90)},{0.3*sin(-90)-0.335-0.16});
\draw [->,blue,-triangle 45] ({0.3*cos(-90)},{0.3*sin(-90)-0.335-0.16}) -- ({0.3*cos(-90)-0.16},{0.3*sin(-90)-0.335-0.16});
\draw [->,red,-triangle 45] ({0.3*cos(-90)},{0.3*sin(-90)-0.335-0.16}) -- ({0.3*cos(-90)+0.16},{0.3*sin(-90)-0.335-0.16});
\draw [->,Green,-triangle 45] ({0.3*cos(170)},{0.3*sin(170)-0.335}) -- ({0.2*cos(170)},{0.2*sin(170)-0.335});
\draw [->,Green,-triangle 45] ({0.3*cos(205)},{0.3*sin(205)-0.335}) -- ({0.2*cos(205)},{0.2*sin(205)-0.335});
\draw [->,Green,-triangle 45] ({0.3*cos(240)},{0.3*sin(240)-0.335}) -- ({0.2*cos(240)},{0.2*sin(240)-0.335});
\draw [->,Green,-triangle 45] ({0.3*cos(10)},{0.3*sin(10)-0.335}) -- ({0.2*cos(10)},{0.2*sin(10)-0.335});
\draw [->,Green,-triangle 45] ({0.3*cos(-25)},{0.3*sin(-25)-0.335}) -- ({0.2*cos(-25)},{0.2*sin(-25)-0.335});
\draw [->,Green,-triangle 45] ({0.3*cos(-60)},{0.3*sin(-60)-0.335}) -- ({0.2*cos(-60)},{0.2*sin(-60)-0.335});
\path 
   ({0.3*cos(10)+0.05},{0.3*sin(10)-0.335}) node {\footnotesize{$23$}}
    ({0.3*cos(-25)+0.06},{0.3*sin(-25)-0.335}) node {\footnotesize{$232$}}
    ({0.3*cos(-60)+0.03},{0.3*sin(-60)-0.335-0.06}) node {\footnotesize{$2323$}}
    ({-0.3*cos(10)-0.05},{0.3*sin(10)-0.335}) node {\footnotesize{$32$}}
    ({-0.3*cos(-25)-0.06},{0.3*sin(-25)-0.335}) node {\footnotesize{$323$}}
   ({-0.3*cos(-60)-0.03},{0.3*sin(-60)-0.335-0.06}) node {\footnotesize{$3232$}}
    ({-0.3*cos(-90)},{0.3*sin(-90)-0.335+0.05}) node {\footnotesize{$y$}}
      ({-0.3*cos(-90)},{0.3*sin(-90)-0.335+0.05-0.25}) node {\footnotesize{$y1$}}
        ({-0.3*cos(-90)+0.21},{0.3*sin(-90)-0.335-0.16}) node [gray] {\footnotesize{$313$}}
        ({-0.3*cos(-90)-0.21},{0.3*sin(-90)-0.335-0.16}) node [gray] {\footnotesize{$212$}}
       ({0.2*cos(10)-0.04},{0.2*sin(10)-0.335}) node  [gray] {\footnotesize{$31$}}
    ({0.2*cos(-25)-0.04},{0.2*sin(-25)-0.335}) node [gray] {\footnotesize{$21$}}
    ({0.2*cos(-60)-0.04+0.01},{0.2*sin(-60)-0.335+0.025})  node [gray] {\footnotesize{$31$}}
    ({-0.2*cos(10)+0.04},{0.2*sin(10)-0.335}) node [gray] {\footnotesize{$21$}}
    ({-0.2*cos(-25)+0.04},{0.2*sin(-25)-0.335}) node [gray] {\footnotesize{$31$}}
    ({-0.2*cos(-60)+0.04-0.01},{0.2*sin(-60)-0.335+0.025})  node [gray] {\footnotesize{$21$}};
\end{scope}
\end{tikzpicture}
\caption{Cannon automaton for Class~I triangle groups}\label{fig:I}
\end{minipage}\begin{minipage}{0.3\textwidth}
Figure~\ref{fig:I} shows the Cannon automaton for Class I triangle groups. The generators are labelled $1$, $2$, and~$3$, and the labels on the edges are indicated by colours, with $1=\text{green}$, $2=\text{blue}$, and $3=\text{red}$. The cone types are given by the base element of a representative cone of that type. The cone types in grey are duplicates, and the reader should instead imagine the arrow pointing to the corresponding cone type in black. Finally, $x=121\cdots$, $y=232\cdots$ and $z=131\cdots$ are the longest elements of the parabolic subgroups $W_{12}$, $W_{23}$, and $W_{13}$ respectively. \end{minipage}
\end{figure}
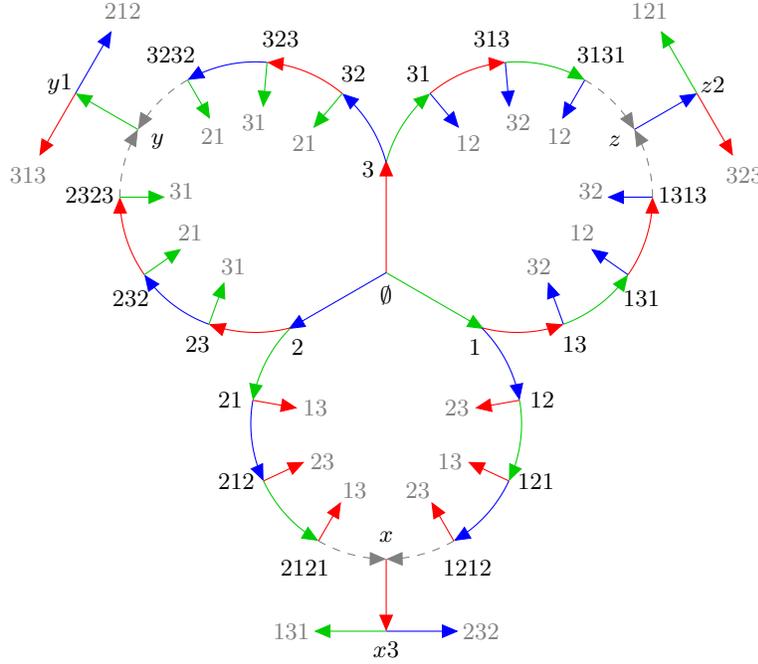

\subsection{Class II}

The proof of the following lemma, which determines the Cannon automata for the triangle groups in Class~II, is similar to the proof of Lemma~\ref{lem:automaton1} and the details are omitted. 

\begin{lemma}\label{lem:lem:automaton2} 
Let $(W,S)$ be a triangle group with $S=\{s,t,u\}$. Suppose that
$m_{st},m_{tu}\geq 4$ and that $m_{su}=2$. Let $x$ be the longest element of
$W_{st}$. Let $v\in W_{st}$ with $v\notin \{x,xs\}$, and write $v=s_{\ell}\cdots s_1$ for the unique reduced expression for~$v$. Then
\begin{align*}
T(vu)&=T(s_1u)&T((xs)ut)&=T(tut)&T(xutu)&=T(tutu)&T(xutsu)&=T(tutsu)\\
T(xutst)&=T(stst)&T(suts)&=T(sts)&T(sutu)&=T(utu).
\end{align*}
\end{lemma}
The resulting automaton has $a+b+c+4$ vertices, given by $T(w)$ with $w\in\bW$, where
$$
\bW=W_{st}\cup W_{tu}\cup W_{us}\cup \{sut\}\cup x\cdot\{su,u,ut,uts\}\cup y\cdot\{us,s,st,stu\},
$$
where $x=sts\cdots$ is the longest element of $W_{st}$ and $y=tut\cdots$ is the longest element of $W_{tu}$. See Figure~\ref{fig:II} for an illustration.

\subsection{Class III}

The proof of the following lemma, which determines the Cannon automaton for the triangle groups in Class~III, is similar to the proof of Lemma~\ref{lem:automaton1} and the details are omitted. 

\begin{lemma}\label{lem:automaton3} Let $(W,S)$ be a triangle group with
  $S=\{s,t,u\}$. Suppose that $m_{st}\geq 6$, $m_{tu}=3$, and $m_{su}=2$. Let
  $x$ be the longest element element of $W_{st}$. Let $v\in W_{st}$ with
  $v\notin \{1,x,xs,xt,xts\}$, and let $v=s_{\ell}\cdots s_1$ be the unique reduced expression for~$v$. Then
\begin{align*}
T(vu)&=T(s_1u)&T(xutstst)&=T(ststst)&T((xs)ut)&=T(tut)\\
T(xutstst)&=T(ststst)&T(xutststut)&=T(stststut)&T(xutststuts)&=T(stststuts)\\
T(xutststutstu)&=T(stststutstu)&T((xt)utstutststs)&=T(tststs)&T((xts)ut)&=T(tut)\\
T((xt)utsts)&=T(ststs)&T((xt)utstsu)&=T(ststsu)&T(tutstst)&=T(tstst)\\
T(utstuts)&=T(tuts)&T(utstut)&=T(tut)&T(utst)&=T(tst)\\
T(stut)&=T(tut)&T(ustst)&=T(stst)&T(stuts)&=T(tuts)
\end{align*}
\end{lemma}

An illustration of the resulting automaton is given in Figure~\ref{fig:III}.

\subsection{Class IV}

The proof of the following lemma, which determines the Cannon automata for the triangle groups in Class~IV, is similar to the proof of Lemma~\ref{lem:automaton1} and the details are omitted. 

\begin{lemma}\label{lem:automaton4}
Let $(W,S)$ be a Fuchsian Coxeter system, and let $s,t,u\in S$ be pairwise distinct generators. 
\begin{enumerate}
\item If $2\leq m_{st},m_{tu}<\infty$ and $m_{us}=\infty$ then for all $v\in W_{st}$ we have
$$
T(vu)=\begin{cases}
T(u)&\textrm{if $\ell(vt)=\ell(v)+1$}\\
T(tu)&\textrm{if $\ell(vt)=\ell(v)-1$}.
\end{cases}
$$
\item If $2\leq m_{st}<\infty$ and $m_{tu}=m_{us}=\infty$ then for all $v\in W_{st}$ we have
$
T(vu)=T(u)
$
\end{enumerate}
\end{lemma}

The resulting automaton has vertices given by $\{T(w)\mid w\in\bW\}$, where $\bW$ is the union of all finite parabolic subgroups:
$$
\bW=W_{1,2}\cup W_{2,3}\cup \cdots \cup W_{n-1,n}\cup W_{n,1},
$$
where the generators are labelled $1,2,\ldots,n$.

\subsection{Proof of Theorem~\ref{thm:stronglyconnected}}
 
\begin{proof}[Proof of Theorem~\ref{thm:stronglyconnected}]
Consider Class~I first. It is clear that the cone types $\emptyset$, $1$, $2$ and $3$ are not recurrent (in the notation of Figure~\ref{fig:I}, we will simply write $w$ for the cone type $T(w)$). 

Suppose, without loss of generality, that $m_{12}>3$. We first note that there is a cycle
$$
c=(12\to23\to31\to12\to 121\to13\to32\to 21\to 212\to 23\to31\to 12)
$$
containing all cone types $w$ with $\ell(w)=2$ (it is important here that $121$ and $212$ are not the longest words of $W_{12}$). 

Now let $w$ be a cone type other than $\emptyset$, $1$, $2$, or $3$. From any such $w$ there is a path $\gamma_1$ in the automaton to some cone type $w_{ij}k$ where $(i,j,k)$ is some permutation of the generating set $(1,2,3)$ and $w_{ij}=iji\cdots$ is the longest element of~$W_{ij}$. Then
\begin{align*}
w_{ij}k\to jkj\to\begin{cases}
ji&\text{if $m_{jk}>3$}\\
jkji\to jij\to jk&\text{if $m_{jk}=3$ and $m_{ij}>3$}\\
jkji\to jij\to ijik\to kik\to kj&\text{if $m_{jk}=m_{ij}=3$ and $m_{ik}>3$},
\end{cases}
\end{align*}
and so in all cases there is a path $\gamma_2$ from $w_{ij}k$ to a cone type $i'j'$ of length~$2$, and so there is a path $\gamma_1\gamma_2$ from $w$ to $i'j'$ in the automaton. Furthermore, there is a path $\gamma_3$ from some cone type $i''j''$ to $w$ (because every reduced path in $W$ must pass through a word of length~$2$). Thus, using~$c$, there is a loop $\gamma$ from $12$ to $12$ passing through $w$. This shows that $w$ is recurrent, and readily implies that the automaton is strongly connected.

Now consider Class~II. The set $\mathbf{R}$ of recurrent cone types is as follows:
\begin{align*}
\mathbf{R}=\begin{cases}
\{T(w)\mid w\in\bW\text{ with }\ell(w)\geq 2\}\backslash\{st,ut\}&\text{if $m_{st},m_{tu}>4$}\\
\{T(w)\mid w\in\bW\text{ with }\ell(w)\geq 2\}\backslash\{st,ut,tu\}&\text{if $m_{st}=4$ and $m_{tu}>4$}\\
\{T(w)\mid w\in\bW\text{ with }\ell(w)\geq 2\}\backslash\{st,ut,ts\}&\text{if $m_{st}>4$ and $m_{tu}=4$}\\
\{T(w)\mid w\in\bW\text{ with }\ell(w)\geq 2\}\backslash\{st,ut,ts,tu\}&\text{if $m_{st}=m_{tu}=4$}.
\end{cases}
\end{align*}
Direct inspection (see Figure~\ref{fig:II}) shows that the automaton is strongly connected as long as either $m_{st}>4$ or $m_{tu}>4$. We omit the full details, however for example consider the concrete case $m_{st}=4$ and $m_{tu}=5$ (thus in Figure~\ref{fig:II} we have $x=1212=2121$ and $y=23232=32323$). We have the following paths (in the notation of Figure~\ref{fig:II})
\begin{align*}
\gamma_1&=(13\to132\to 121\to 1212=x\to x3\to x32)\\
\gamma_2&=(x32\to x321\to 1212=x\to x3\to x32\to 2323\to 13)\\ 
\gamma_3&=(13\to132\to 323\to 3232=y3\to y31\to 212=x1\to x13\to 232\to 2323\to 13)\\
\gamma_4&=(x32\to 2323\to 23232=y\to y1\to y12\to y123\to 2323\to 13),
\end{align*}
and the concatenation $\gamma_1\gamma_2\gamma_3\gamma_1\gamma_4$ gives a loop containing all recurrent cone types. Thus the automaton is strongly connected. 

We omit the details of Class~III, and refer the reader to Lemma~\ref{lem:automaton3} and Figure~\ref{fig:III}. 

Finally, consider the groups in Class~IV. Let $1,2,\ldots,n$ be the generators of $W$, arranged cyclically around the fundamental chamber. If $n\geq 5$, then for each pair $(i,i+1)$ there is a generator $j$ with $m_{i,j}=\infty$ and $m_{{i+1},j}=\infty$, and thus $i\to j \to {i+1}$ in the automaton. Moreover, for any $w\in W_{i,i+1}$ we have $w\to j$. It follows that every node other than $\emptyset$ is recurrent, and moreover that the automaton is strongly connected. If $n=4$ then we may assume that $m_{12}\geq 3$ (for if $m_{ij}=2$ for all $i,j$ then $W$ is affine type $\tilde{A}_1\times\tilde{A}_1$, where $\tilde{A}_1$ is the infinite dihedral group). Then $21\to 3$ and $12\to 4$. Thus $1\to 12\to 4\to 2\to21\to3\to1$. Again it follows that every node other than $\emptyset$ is recurrent, and that the automaton is strongly connected.
\end{proof}

\begin{remark} The Cannon automata for the affine triangle groups $(3,3,3)$, $(4,4,2)$ and $(6,3,2)$ are not strongly connected. For example, for the $(4,4,2)$ group, note that the cone types $13$ and $1212$ are recurrent, yet there is no path from $1212$ to $13$ in the automaton. 
\end{remark}

\newpage

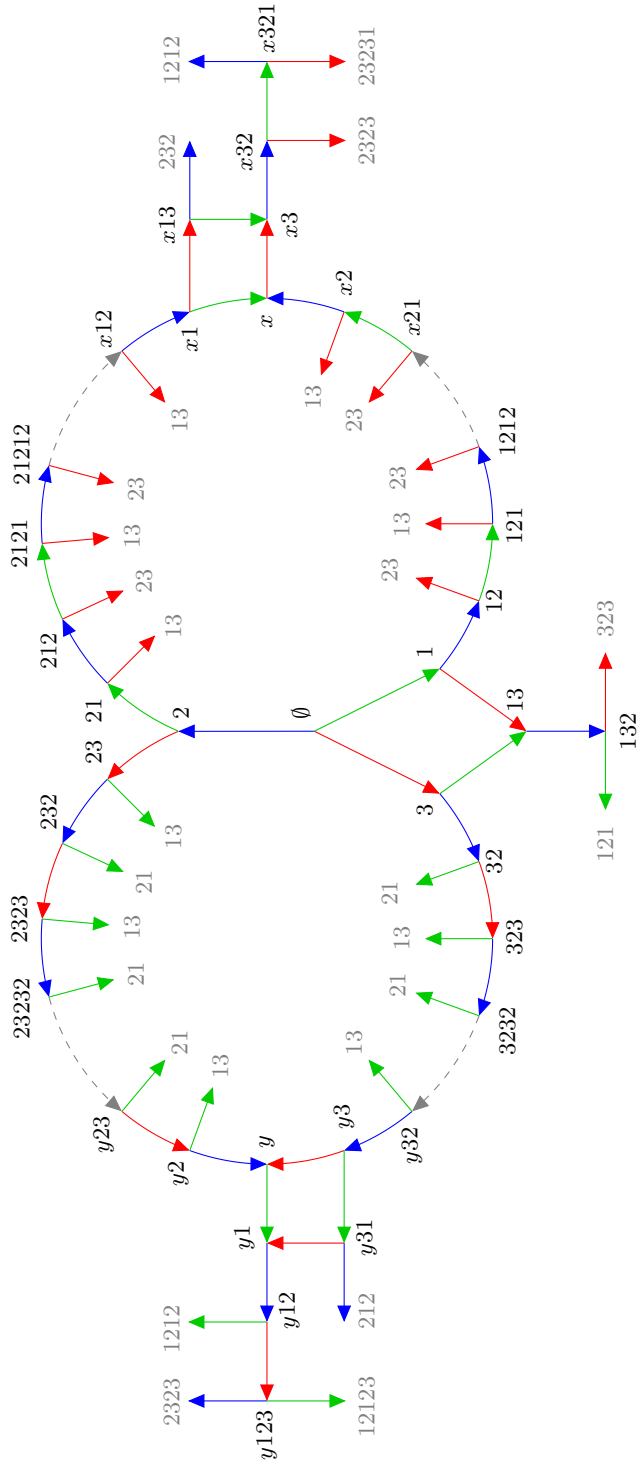
\begin{figure}[!h]
\centering
\begin{tikzpicture} [scale=2.7]
\node [rotate=90] {\begin{tikzpicture} [scale=3]
\draw [blue,->,-triangle 45] (0,-0.21) -- ({cos(23.15)-0.92},{sin(23.15)});
\draw [red,->,-triangle 45] (0,-0.21) -- ({cos(310)-0.92},{sin(310)});
\draw [Green,->,-triangle 45] (0,-0.21) -- ({-cos(310)+0.92},{sin(310)});
\draw [Green,->,-triangle 45] ({cos(310)-0.92},{sin(310)}) -- (0,-1.15);
\draw [red,->,-triangle 45] ({-cos(310)+0.92},{sin(310)}) -- (0,-1.15);
\draw [blue,->,-triangle 45] (0,-1.15) -- (0,-1.5);
\draw [Green,->,-triangle 45] (0,-1.5) -- (-0.35,-1.5);
\draw [red,->,-triangle 45] (0,-1.5) -- (0.35,-1.5);
\draw [red,->,domain=23.15:45,-triangle 45] plot ({cos(\x)-0.92},{sin(\x)});
\draw [blue,->,domain={45}:{45+20},-triangle 45] plot ({cos(\x)-0.92},{sin(\x)});
\draw [red,->,domain={45+1*20}:{45+20+1*20},-triangle 45] plot ({cos(\x)-0.92},{sin(\x)});
\draw [blue,->,domain={45+2*20}:{45+20+2*20},-triangle 45] plot ({cos(\x)-0.92},{sin(\x)});
\draw [blue,->,domain=160:180,-triangle 45] plot ({cos(\x)-0.92},{sin(\x)});
\draw [red,->,domain=140:160,-triangle 45] plot ({cos(\x)-0.92},{sin(\x)});
\draw [->,gray,dashed,domain=105:140,-triangle 45] plot ({cos(\x)-0.92},{sin(\x)});
\draw [Green,->,domain=156.85:135,-triangle 45] plot ({cos(\x)+0.92},{sin(\x)});
\draw [blue,->,domain={135}:{135-20},-triangle 45] plot ({cos(\x)+0.92},{sin(\x)});
\draw [Green,->,domain={135-1*20}:{135-20-1*20},-triangle 45] plot ({cos(\x)+0.92},{sin(\x)});
\draw [blue,->,domain={135-2*20}:{135-20-2*20},-triangle 45] plot ({cos(\x)+0.92},{sin(\x)});
\draw [Green,->,domain=20:0,-triangle 45] plot ({cos(\x)+0.92},{sin(\x)});
\draw [blue,->,domain=40:20,-triangle 45] plot ({cos(\x)+0.92},{sin(\x)});
\draw [->,gray,dashed,domain=75:40,-triangle 45] plot ({cos(\x)+0.92},{sin(\x)});
\draw [blue,->,domain={310}:{290},-triangle 45] plot ({cos(\x)-0.92},{sin(\x)});
\draw [red,->,domain={310-1*20}:{290-1*20},-triangle 45] plot ({cos(\x)-0.92},{sin(\x)});
\draw [blue,->,domain={310-2*20}:{290-2*20},-triangle 45] plot ({cos(\x)-0.92},{sin(\x)});
\draw [red,->,domain={200}:{180},-triangle 45] plot ({cos(\x)-0.92},{sin(\x)});
\draw [blue,->,domain={220}:{200},-triangle 45] plot ({cos(\x)-0.92},{sin(\x)});
\draw [->,gray,dashed,domain=250:220,-triangle 45] plot ({cos(\x)-0.92},{sin(\x)});
\draw [blue,->,domain={230}:{250},-triangle 45] plot ({cos(\x)+0.92},{sin(\x)});
\draw [Green,->,domain={230+1*20}:{250+1*20},-triangle 45] plot ({cos(\x)+0.92},{sin(\x)});
\draw [blue,->,domain={230+2*20}:{250+2*20},-triangle 45] plot ({cos(\x)+0.92},{sin(\x)});
\draw [blue,->,domain={-20}:{0},-triangle 45] plot ({cos(\x)+0.92},{sin(\x)});
\draw [Green,->,domain={-40}:{-20},-triangle 45] plot ({cos(\x)+0.92},{sin(\x)});
\draw [->,gray,dashed,domain={-70}:{-40},-triangle 45] plot ({cos(\x)+0.92},{sin(\x)});
\draw [Green,->,-triangle 45] (-1.92,0) -- ({-1.92-0.35},0);
\draw [blue,->,-triangle 45] ({-1.92-0.35},0) --  ({-1.92-2*0.35},0);
\draw [red,->,-triangle 45] ({-1.92-2*0.35},0) --  ({-1.92-3*0.35},0);
\draw [blue,->,-triangle 45] ({-1.92-3*0.35},0) -- ({-1.92-3*0.35},0.35);
\draw [Green,->,-triangle 45] ({-1.92-2*0.35},0) -- ({-1.92-2*0.35},0.35);
\draw [Green,->,-triangle 45] ({-1.92-3*0.35},0) -- ({-1.92-3*0.35},-0.35);
\draw [Green,->,-triangle 45] ({cos(200)-0.92},{sin(200)}) -- ({-1-0.92-0.35},{sin(200)});
\draw [red,->,-triangle 45] ({-1-0.92-0.35},{sin(200)}) -- ({-1.92-0.35},0);
\draw [blue,->,-triangle 45] ({-1-0.92-0.35},{sin(200)}) -- ({-1-0.92-2*0.35},{sin(200)});
\draw [red,->,-triangle 45] (1.92,0) -- ({1.92+0.35},0);
\draw [blue,->,-triangle 45] ({+1.92+0.35},0) --  ({+1.92+2*0.35},0);
\draw [Green,->,-triangle 45] ({+1.92+2*0.35},0) --  ({+1.92+3*0.35},0);
\draw [blue,->,-triangle 45] ({+1.92+3*0.35},0) -- ({+1.92+3*0.35},0.35);
\draw [red,->,-triangle 45] ({+1.92+3*0.35},0) -- ({+1.92+3*0.35},-0.35);
\draw [red,->,-triangle 45] ({+1.92+2*0.35},0) -- ({+1.92+2*0.35},-0.35);
\draw [red,->,-triangle 45] ({-cos(160)+0.92},{sin(160)}) -- ({1+0.92+0.35},{sin(160)});
\draw [Green,->,-triangle 45] ({+1+0.92+0.35},{sin(160)}) -- ({+1.92+0.35},0);
\draw [blue,->,-triangle 45] ({+1+0.92+0.35},{sin(160)}) -- ({+1+0.92+2*0.35},{sin(160)});
\draw [Green,->,-triangle 45] ({cos(45)-0.92},{sin(45)}) -- ({0.7*cos(45)-0.92},{0.7*sin(45)});
\draw [Green,->,-triangle 45] ({cos(65)-0.92},{sin(65)}) -- ({0.7*cos(65)-0.92},{0.7*sin(65)});
\draw [Green,->,-triangle 45] ({cos(85)-0.92},{sin(85)}) -- ({0.7*cos(85)-0.92},{0.7*sin(85)});
\draw [Green,->,-triangle 45] ({cos(105)-0.92},{sin(105)}) -- ({0.7*cos(105)-0.92},{0.7*sin(105)});
\draw [Green,->,-triangle 45] ({cos(140)-0.92},{sin(140)}) -- ({0.7*cos(140)-0.92},{0.7*sin(140)});
\draw [Green,->,-triangle 45] ({cos(160)-0.92},{sin(160)}) -- ({0.7*cos(160)-0.92},{0.7*sin(160)});
\draw [Green,->,-triangle 45] ({cos(220)-0.92},{sin(220)}) -- ({0.7*cos(220)-0.92},{0.7*sin(220)});
\draw [Green,->,-triangle 45] ({cos(250)-0.92},{sin(250)}) -- ({0.7*cos(250)-0.92},{0.7*sin(250)});
\draw [Green,->,-triangle 45] ({cos(270)-0.92},{sin(270)}) -- ({0.7*cos(270)-0.92},{0.7*sin(270)});
\draw [Green,->,-triangle 45] ({cos(290)-0.92},{sin(290)}) -- ({0.7*cos(290)-0.92},{0.7*sin(290)});
\draw [red,->,-triangle 45] ({-cos(45)+0.92},{sin(45)}) -- ({-0.7*cos(45)+0.92},{0.7*sin(45)});
\draw [red,->,-triangle 45] ({-cos(65)+0.92},{sin(65)}) -- ({-0.7*cos(65)+0.92},{0.7*sin(65)});
\draw [red,->,-triangle 45] ({-cos(85)+0.92},{sin(85)}) -- ({-0.7*cos(85)+0.92},{0.7*sin(85)});
\draw [red,->,-triangle 45] ({-cos(105)+0.92},{sin(105)}) -- ({-0.7*cos(105)+0.92},{0.7*sin(105)});
\draw [red,->,-triangle 45] ({-cos(140)+0.92},{sin(140)}) -- ({-0.7*cos(140)+0.92},{0.7*sin(140)});
\draw [red,->,-triangle 45] ({-cos(200)+0.92},{sin(200)}) -- ({-0.7*cos(200)+0.92},{0.7*sin(200)});
\draw [red,->,-triangle 45] ({-cos(220)+0.92},{sin(220)}) -- ({-0.7*cos(220)+0.92},{0.7*sin(220)});
\draw [red,->,-triangle 45] ({-cos(250)+0.92},{sin(250)}) -- ({-0.7*cos(250)+0.92},{0.7*sin(250)});
\draw [red,->,-triangle 45] ({-cos(270)+0.92},{sin(270)}) -- ({-0.7*cos(270)+0.92},{0.7*sin(270)});
\draw [red,->,-triangle 45] ({-cos(290)+0.92},{sin(290)}) -- ({-0.7*cos(290)+0.92},{0.7*sin(290)});
\path (0.08,-0.15) node {\footnotesize{$\emptyset$}}
(0.35,-0.7) node {\footnotesize{$1$}}
(-0.35,-0.7) node {\footnotesize{$3$}}
(0.08,0.38) node {\footnotesize{$2$}}
(0.14,-1.1) node {\footnotesize{$13$}}
(0,-1.6) node {\footnotesize{$132$}}
(-0.5,-1.5) node [gray] {\footnotesize{$121$}}
(+0.5,-1.5) node [gray] {\footnotesize{$323$}}
;
\path ({cos(45)-0.92+0.12},{sin(45)+0.06}) node {\footnotesize{$23$}}
({cos(65)-0.92+0.14},{sin(65)+0.06}) node {\footnotesize{$232$}}
({cos(85)-0.92},{sin(85)+0.1}) node {\footnotesize{$2323$}}
({cos(105)-0.92-0.05},{sin(105)+0.12}) node {\footnotesize{$23232$}}
({cos(140)-0.92-0.1},{sin(140)+0.08}) node {\footnotesize{$y23$}}
({cos(160)-0.92-0.1},{sin(160)+0.05}) node {\footnotesize{$y2$}}
({cos(180)-0.92+0.1},{sin(180)}) node {\footnotesize{$y$}}
({cos(200)-0.92+0.15},{sin(200)}) node {\footnotesize{$y3$}}
({cos(220)-0.92-0.15},{sin(220)}) node {\footnotesize{$y32$}}
({cos(255)-0.92-0.15},{sin(255)-0.1}) node {\footnotesize{$3232$}}
({cos(275)-0.92-0.08},{sin(275)-0.1}) node {\footnotesize{$323$}}
({cos(295)-0.92-0.08},{sin(295)-0.1}) node {\footnotesize{$32$}}
({cos(180)-0.92-0.325},{sin(180)+0.1}) node {\footnotesize{$y1$}}
({cos(180)-0.92-2*0.325},{sin(180)-0.1}) node {\footnotesize{$y12$}}
({cos(180)-0.92-4*0.325+0.1},{sin(180)}) node {\footnotesize{$y123$}}
({cos(180)-0.92-0.325},{sin(200)-0.1}) node {\footnotesize{$y31$}}
;
\path [gray] 
({0.5*cos(45)-0.92+0.1},{0.5*sin(45)+0.06}) node {\footnotesize{$13$}}
({0.6*cos(65)-0.92},{0.6*sin(65)}) node {\footnotesize{$21$}}
({0.6*cos(85)-0.92},{0.6*sin(85)}) node {\footnotesize{$13$}}
({0.6*cos(105)-0.92},{0.6*sin(105)}) node {\footnotesize{$21$}}
({0.6*cos(140)-0.92},{0.6*sin(140)}) node {\footnotesize{$21$}}
({0.6*cos(160)-0.92},{0.6*sin(160)}) node {\footnotesize{$13$}}
({0.6*cos(220)-0.92},{0.6*sin(220)}) node {\footnotesize{$13$}}
({0.6*cos(255)-0.92-0.05},{0.6*sin(255)}) node {\footnotesize{$21$}}
({0.6*cos(275)-0.92-0.05},{0.6*sin(275)}) node {\footnotesize{$13$}}
({0.6*cos(295)-0.92-0.05},{0.6*sin(295)}) node {\footnotesize{$21$}}
({cos(180)-0.92-3*0.325-0.1},{sin(200)-0.1}) node {\footnotesize{$12123$}}
({cos(180)-0.92-3*0.325-0.08},{sin(160)+0.08}) node {\footnotesize{$2323$}}
({cos(180)-0.92-2*0.325-0.08},{sin(160)+0.08}) node {\footnotesize{$1212$}}
({cos(180)-0.92-2*0.325},{sin(200)-0.1}) node {\footnotesize{$212$}}
;
\path ({-cos(45)+0.92-0.12},{sin(45)+0.06}) node {\footnotesize{$21$}}
({-cos(65)+0.92-0.14},{sin(65)+0.06}) node {\footnotesize{$212$}}
({-cos(85)+0.92},{sin(85)+0.1}) node {\footnotesize{$2121$}}
({-cos(105)+0.92+0.05},{sin(105)+0.12}) node {\footnotesize{$21212$}}
({-cos(140)+0.92+0.1},{sin(140)+0.08}) node {\footnotesize{$x12$}}
({-cos(160)+0.92-0.12},{sin(160)}) node {\footnotesize{$x1$}}
({-cos(180)+0.92-0.1},{sin(180)}) node {\footnotesize{$x$}}
({-cos(200)+0.92+0.13},{sin(200)}) node {\footnotesize{$x2$}}
({-cos(220)+0.92+0.15},{sin(220)}) node {\footnotesize{$x21$}}
({-cos(255)+0.92+0.15},{sin(255)-0.1}) node {\footnotesize{$1212$}}
({-cos(275)+0.92+0.08},{sin(275)-0.1}) node {\footnotesize{$121$}}
({-cos(295)+0.92+0.08},{sin(295)-0.1}) node {\footnotesize{$12$}}
({-cos(180)+0.92+0.325},{sin(180)-0.1}) node {\footnotesize{$x3$}}
({-cos(180)+0.92+2*0.325},{sin(180)+0.1}) node {\footnotesize{$x32$}}
({-cos(180)+0.92+4*0.325-0.1},{sin(180)}) node {\footnotesize{$x321$}}
({-cos(180)+0.92+0.325},{sin(160)+0.1}) node {\footnotesize{$x13$}}
;
\path [gray] ({-0.5*cos(45)+0.92-0.1},{0.5*sin(45)+0.06}) node {\footnotesize{$13$}}
({-0.6*cos(65)+0.92},{0.6*sin(65)}) node {\footnotesize{$23$}}
({-0.6*cos(85)+0.92},{0.6*sin(85)}) node {\footnotesize{$13$}}
({-0.6*cos(105)+0.92},{0.6*sin(105)}) node {\footnotesize{$23$}}
({-0.6*cos(140)+0.92},{0.6*sin(140)}) node {\footnotesize{$13$}}
({-0.6*cos(200)+0.92},{0.6*sin(200)}) node {\footnotesize{$13$}}
({-0.6*cos(220)+0.92},{0.6*sin(220)}) node {\footnotesize{$23$}}
({-0.6*cos(255)+0.92+0.05},{0.6*sin(255)}) node {\footnotesize{$23$}}
({-0.6*cos(275)+0.92+0.05},{0.6*sin(275)}) node {\footnotesize{$13$}}
({-0.6*cos(295)+0.92+0.05},{0.6*sin(295)}) node {\footnotesize{$23$}}
({-cos(180)+0.92+3*0.325+0.1},{sin(200)-0.1}) node {\footnotesize{$23231$}}
({-cos(180)+0.92+3*0.325+0.08},{sin(160)+0.08}) node {\footnotesize{$1212$}}
({-cos(180)+0.92+2*0.325},{sin(160)+0.1}) node {\footnotesize{$232$}}
({-cos(180)+0.92+2*0.325+0.05},{sin(200)-0.1}) node {\footnotesize{$2323$}}
;\end{tikzpicture}};
\end{tikzpicture}
\caption{Cannon automata for Class~II triangle groups, with $a$ even and $b$ odd}\label{fig:II}
\end{figure}

\newpage

\begin{figure}[!h]
\centering
\begin{tikzpicture} [scale=4]
\node [rotate=90] {\begin{tikzpicture} [scale=3.9]
\draw [blue,->,domain=180:165,-triangle 45] plot ({cos(\x)},{sin(\x)});
\draw [Green,->,domain=165:150,-triangle 45] plot ({cos(\x)},{sin(\x)});
\draw [blue,->,domain=150:135,-triangle 45] plot ({cos(\x)},{sin(\x)});
\draw [Green,->,domain=135:120,-triangle 45] plot ({cos(\x)},{sin(\x)});
\draw [blue,->,domain=120:105,-triangle 45] plot ({cos(\x)},{sin(\x)});
\draw [Green,->,domain=180:195,-triangle 45] plot ({cos(\x)},{sin(\x)});
\draw [blue,->,domain=195:210,-triangle 45] plot ({cos(\x)},{sin(\x)});
\draw [Green,->,domain=210:225,-triangle 45] plot ({cos(\x)},{sin(\x)});
\draw [blue,->,domain=225:240,-triangle 45] plot ({cos(\x)},{sin(\x)});
\draw [Green,->,domain=240:255,-triangle 45] plot ({cos(\x)},{sin(\x)});
\draw [Green,->,domain=15:0,-triangle 45] plot ({cos(\x)},{sin(\x)});
\draw [blue,->,domain=30:15,-triangle 45] plot ({cos(\x)},{sin(\x)});
\draw [Green,->,domain=45:30,-triangle 45] plot ({cos(\x)},{sin(\x)});
\draw [blue,->,domain=60:45,-triangle 45] plot ({cos(\x)},{sin(\x)});
\draw [Green,->,domain=75:60,-triangle 45] plot ({cos(\x)},{sin(\x)});
\draw [blue,->,domain=345:360,-triangle 45] plot ({cos(\x)},{sin(\x)});
\draw [Green,->,domain=330:345,-triangle 45] plot ({cos(\x)},{sin(\x)});
\draw [blue,->,domain=315:330,-triangle 45] plot ({cos(\x)},{sin(\x)});
\draw [Green,->,domain=300:315,-triangle 45] plot ({cos(\x)},{sin(\x)});
\draw [blue,->,domain=285:300,-triangle 45] plot ({cos(\x)},{sin(\x)});
\draw [->,gray,dashed,domain=105:75,-triangle 45] plot ({cos(\x)},{sin(\x)});
\draw [->,gray,dashed,domain=255:285,-triangle 45] plot ({cos(\x)},{sin(\x)});
\draw [red,->,-triangle 45] (-1,0) -- (-1.2,0);
\draw [red,->,-triangle 45] ({cos(195)},{sin(195)}) -- (-1.2,{sin(195)});
\draw [red,->,-triangle 45] ({cos(165)},{sin(165)}) -- (-1.2,{sin(165)});
\draw [blue,->,-triangle 45] (-1.2,0) -- (-1.45,0);
\draw [blue,->,-triangle 45] (-1.2,{sin(165)}) -- (-1.45,{sin(165)});
\draw [red,->,-triangle 45] (-1.45,0) -- (-1.45,{sin(165)});
\draw [Green,->,-triangle 45] (-1.2,0) -- (-1.2,{sin(195)});
\draw [Green,->,-triangle 45] (-1.45,{sin(165)}) -- (-1.7,{sin(165)});
\draw [blue,->,-triangle 45] (-1.7,{sin(165)}) -- (-1.95,{sin(165)});
\draw [red,->,-triangle 45] (-1.95,{sin(165)}) -- (-1.95,{sin(165)+0.25});
\draw [Green,->,-triangle 45] (-1.95,{sin(165)}) -- (-2.2,{sin(165)});
\draw [red,->,-triangle 45] (-2.2,{sin(165)}) -- (-2.2,{sin(165)+0.25});
\draw [blue,->,-triangle 45] (-2.2,{sin(165)}) -- (-2.2,{sin(165)-0.25});
\draw [Green,->,-triangle 45] (-1.45,{sin(180)}) -- (-1.7,{sin(180)});
\draw [blue,->,-triangle 45] (-1.7,{sin(180)}) -- (-1.95,{sin(180)});
\draw [red,->,-triangle 45] (-1.7,{sin(180)}) -- (-1.7,{sin(165)});
\draw [blue,->,-triangle 45] (-1.2,{sin(195)}) -- (-1.45,{sin(195)});
\draw [Green,->,-triangle 45] (-1.45,{sin(195)}) -- (-1.7,{sin(195)});
\draw [blue,->,-triangle 45] (-1.7,{sin(195)}) -- (-1.95,{sin(195)});
\draw [red,->,-triangle 45] (-1.45,{sin(195)}) -- (-1.45,{sin(195)-0.25});
\draw [red,->,-triangle 45] (-1.7,{sin(195)}) -- (-1.7,{sin(195)-0.25});
\draw [red,->,-triangle 45] (1,{sin(0)}) -- (1.2,{sin(0)});
\draw [red,->,-triangle 45] ({cos(15)},{sin(15)}) -- (1.2,{sin(15)});
\draw [red,->,-triangle 45] ({cos(-15)},{sin(-15)}) -- (1.2,{sin(-15)});
\draw [red,->,-triangle 45] ({cos(-30)},{sin(-30)}) -- (1.2,{sin(-30)});
\draw [blue,->,-triangle 45] (1.2,{sin(0)}) -- (1.45,{sin(0)});
\draw [blue,->,-triangle 45] (1.2,{sin(-15)}) -- (1.45,{sin(-15)});
\draw [Green,->,-triangle 45] (1.45,{sin(0)}) -- (1.7,{sin(0)});
\draw [Green,->,-triangle 45] (1.45,{sin(-15)}) -- (1.7,{sin(-15)});
\draw [blue,->,-triangle 45] (1.7,{sin(0)}) -- (1.95,{sin(0)});
\draw [blue,->,-triangle 45] (1.7,{sin(-15)}) -- (1.95,{sin(-15)});
\draw [red,->,-triangle 45] (1.95,{sin(0)}) -- (2.2,{sin(0)});
\draw [red,->,-triangle 45] (1.95,{sin(-15)}) -- (2.2,{sin(-15)});
\draw [Green,->,-triangle 45] (1.95,{sin(0)}) -- (1.95,{sin(15)});
\draw [Green,->,-triangle 45] (2.2,{sin(0)}) -- (2.2,{sin(15)});
\draw [Green,->,-triangle 45] (1.2,{sin(15)}) -- (1.2,{sin(0)});
\draw [Green,->,-triangle 45] (1.2,{sin(-30)}) -- (1.2,{sin(-15)});
\draw [red,->,-triangle 45] (1.45,{sin(-15)}) -- (1.45,{sin(0)});
\draw [red,->,-triangle 45] (1.7,{sin(-15)}) -- (1.7,{sin(0)});
\draw [blue,->,-triangle 45] (2.2,{sin(-15)}) -- (2.2,{sin(0)});
\draw [red,->,-triangle 45] (1.95,{sin(15)}) -- (2.2,{sin(15)});
\draw [blue,->,-triangle 45] (2.2,{sin(15)}) -- (2.2,{sin(15)+0.25});
\draw [Green,->,-triangle 45] (2.2,{sin(15)+0.25}) -- (2.2,{sin(15)+2*0.25});
\draw [blue,->,-triangle 45] (2.2,{sin(15)+2*0.25}) -- (2.2,{sin(15)+3*0.25});
\draw [red,->,-triangle 45] (2.2,{sin(15)+3*0.25}) -- (2.2,{sin(15)+4*0.25});
\draw [blue,->,-triangle 45] (2.2,{sin(15)+4*0.25}) -- (1.95,{sin(15)+4*0.25});
\draw [Green,->,-triangle 45] (2.2,{sin(15)+4*0.25}) -- (2.45,{sin(15)+4*0.25});
\draw [Green,->,-triangle 45] (2.2,{sin(15)+3*0.25}) -- (2.45,{sin(15)+3*0.25});
\draw [red,->,-triangle 45] (2.2,{sin(15)+2*0.25}) -- (2.45,{sin(15)+2*0.25});
\draw [red,->,-triangle 45] (2.2,{sin(15)+1*0.25}) -- (2.45,{sin(15)+1*0.25});
\draw [blue,->,-triangle 45] (1.95,{sin(15)}) -- (1.7,{sin(15)});
\draw [Green,->,-triangle 45] (2.2,{sin(-15)}) -- (2.2,{sin(-15)-0.25});
\draw [Green,->,-triangle 45] (1.95,{sin(-15)}) -- (1.95,{sin(-15)-0.25});
\draw [blue,->,-triangle 45] (1.2,{sin(15)}) -- (1.45,{sin(15)});
\draw [blue,->,-triangle 45] (1.2,{sin(-30)}) -- (1.45,{sin(-30)});
\draw [Green,->,-triangle 45] (-1.2,{sin(180-15)}) -- (-1.2,{sin(180-15)+0.25});
\draw [->,red,-triangle 45] ({cos(30)},{sin(30)}) -- ({0.8*cos(30)},{0.8*sin(30)});
\draw [->,red,-triangle 45] ({cos(45)},{sin(45)}) -- ({0.8*cos(45)},{0.8*sin(45)});
\draw [->,red,-triangle 45] ({cos(60)},{sin(60)}) -- ({0.8*cos(60)},{0.8*sin(60)});
\draw [->,red,-triangle 45] ({cos(75)},{sin(75)}) -- ({0.8*cos(75)},{0.8*sin(75)});
\draw [->,red,-triangle 45] ({cos(180-2*15)},{sin(180-2*15)}) -- ({0.8*cos(180-2*15)},{0.8*sin(180-2*15)});
\draw [->,red,-triangle 45] ({cos(180-3*15)},{sin(180-3*15)}) -- ({0.8*cos(180-3*15)},{0.8*sin(180-3*15)});
\draw [->,red,-triangle 45] ({cos(180-4*15)},{sin(180-4*15)}) -- ({0.8*cos(180-4*15)},{0.8*sin(180-4*15)});
\draw [->,red,-triangle 45] ({cos(180-5*15)},{sin(180-5*15)}) -- ({0.8*cos(180-5*15)},{0.8*sin(180-5*15)});
\draw [->,red,-triangle 45] ({cos(180+2*15)},{sin(180+2*15)}) -- ({0.8*cos(180+2*15)},{0.8*sin(180+2*15)});
\draw [->,red,-triangle 45] ({cos(180+3*15)},{sin(180+3*15)}) -- ({0.8*cos(180+3*15)},{0.8*sin(180+3*15)});
\draw [->,red,-triangle 45] ({cos(180+4*15)},{sin(180+4*15)}) -- ({0.8*cos(180+4*15)},{0.8*sin(180+4*15)});
\draw [->,red,-triangle 45] ({cos(180+5*15)},{sin(180+5*15)}) -- ({0.8*cos(180+5*15)},{0.8*sin(180+5*15)});
\draw [->,red,-triangle 45] ({cos(-45)},{sin(-45)}) -- ({0.8*cos(-45)},{0.8*sin(-45)});
\draw [->,red,-triangle 45] ({cos(-60)},{sin(-60)}) -- ({0.8*cos(-60)},{0.8*sin(-60)});
\draw [->,red,-triangle 45] ({cos(-75)},{sin(-75)}) -- ({0.8*cos(-75)},{0.8*sin(-75)});
\path (-0.95,0) node {\footnotesize{$\emptyset$}}
({cos(180+15)+0.07},{sin(180+15)}) node {\footnotesize{$1$}}
({cos(180-15)+0.07},{sin(180-15)}) node {\footnotesize{$2$}}
(-1.2,{sin(180-15)-0.07}) node {\footnotesize{$23$}}
(-1.2,{sin(180-15)+0.3}) node [gray] {\footnotesize{$(13)$}}
(-1.2,{sin(180)+0.07}) node {\footnotesize{$3$}}
(-1.2,{sin(180)-0.25-0.07}) node {\footnotesize{$13$}}
(-1.45,{sin(180)-0.25+0.07}) node {\footnotesize{$132$}}
(-1.7,{sin(180)-0.25+0.07}) node {\footnotesize{$1321$}}
(-1.7,{sin(180)-0.55}) node [gray] {\footnotesize{$(2321)$}}
(-1.45,{sin(180)-0.55}) node [gray] {\footnotesize{$(232)$}}
(-2.1,{sin(180)-0.257}) node [gray] {\footnotesize{$(1212)$}}
(-1.45,{sin(180)-0.07}) node {\footnotesize{$32$}}
(-1.7,{sin(180)-0.07}) node {\footnotesize{$321$}}
(-1.95,{sin(180)+0.075}) node [gray] {\footnotesize{$(212)$}}
(-1.45,{sin(180-15)+0.07}) node {\footnotesize{$232$}}
(-1.7,{sin(180-15)+0.07}) node {\footnotesize{$2321$}}
(-1.95,{sin(180-15)-0.07}) node {\footnotesize{$23212$}}
(-2.35,{sin(180-15)}) node {\footnotesize{$232121$}}
(-2.2,{sin(180-15)+0.3}) node [gray] {\footnotesize{$(2321)$}}
(-2.22,{sin(180-15)-0.3}) node [gray] {\footnotesize{$(21212)$}}
(-1.95,{sin(180-15)+0.3}) node [gray] {\footnotesize{$(232)$}}
({cos(180-30)-0.09},{sin(180-30)+0.02}) node {\footnotesize{$21$}}
({cos(180-45)-0.11},{sin(180-45)+0.02}) node {\footnotesize{$212$}}
({cos(180-60)-0.11},{sin(180-60)+0.06}) node {\footnotesize{$2121$}}
({cos(180-75)-0.09},{sin(180-75)+0.08}) node {\footnotesize{$21212$}}
({cos(180+30)-0.09},{sin(180+30)-0.02}) node {\footnotesize{$12$}}
({cos(180+45)-0.14},{sin(180+45)-0.02}) node {\footnotesize{$121$}}
({cos(180+60)-0.11},{sin(180+60)-0.06}) node {\footnotesize{$1212$}}
({cos(180+75)-0.09},{sin(180+75)-0.08}) node {\footnotesize{$12121$}}
({0.8*cos(180-30)+0.07},{0.8*sin(180-30)}) node [gray] {\footnotesize{$(13)$}}
({-0.8*cos(180-30)-0.07},{0.8*sin(180-30)}) node [gray] {\footnotesize{$(23)$}}
({0.8*cos(180-45)+0.05},{0.8*sin(180-45)-0.03}) node [gray] {\footnotesize{$(23)$}}
({-0.8*cos(180-45)-0.05},{0.8*sin(180-45)-0.03}) node [gray] {\footnotesize{$(13)$}}
({0.8*cos(180-60)+0.02},{0.8*sin(180-60)-0.06}) node [gray] {\footnotesize{$(13)$}}
({-0.8*cos(180-60)-0.02},{0.8*sin(180-60)-0.06}) node [gray] {\footnotesize{$(23)$}}
({0.8*cos(180-75)+0.02},{0.8*sin(180-75)-0.05}) node [gray] {\footnotesize{$(23)$}}
({-0.8*cos(180-75)-0.02},{0.8*sin(180-75)-0.05}) node [gray] {\footnotesize{$(13)$}}
({0.8*cos(180+30)+0.07},{0.8*sin(180+30)}) node [gray] {\footnotesize{$(23)$}}
({0.8*cos(180+45)+0.05},{0.8*sin(180+45)+0.03}) node [gray] {\footnotesize{$(13)$}}
({0.8*cos(180+60)+0.02},{0.8*sin(180+60)+0.06}) node [gray] {\footnotesize{$(23)$}}
({0.8*cos(180+75)+0.02},{0.8*sin(180+75)+0.05}) node [gray] {\footnotesize{$(13)$}}
({-0.8*cos(180+75)-0.02},{0.8*sin(180+75)+0.05}) node [gray] {\footnotesize{$(23)$}}
({-0.8*cos(180+60)-0.02},{0.8*sin(180+60)+0.06}) node [gray] {\footnotesize{$(13)$}}
({-0.8*cos(180+45)-0.05},{0.8*sin(180+45)+0.03}) node [gray] {\footnotesize{$(23)$}}
({cos(0)-0.08},{sin(0)}) node {\footnotesize{$x$}}
({1.2},{-0.05}) node {\footnotesize{$x3$}}
({1.2},{-0.175}) node {\footnotesize{$x23$}}
({1.2},{-0.57}) node {\footnotesize{$x213$}}
({1.57},{-0.5}) node [gray] {\footnotesize{$(232)$}}
({1.45},{0.075}) node {\footnotesize{$x32$}}
({1.45},{-0.33}) node {\footnotesize{$x232$}}
({1.7},{-0.33}) node {\footnotesize{$x2321$}}
({1.95},{-0.19}) node {\footnotesize{$x23212$}}
({2.4},{sin(-15)}) node {\footnotesize{$x232123$}}
({2.4},{sin(-15)-0.25}) node [gray] {\footnotesize{$(121213)$}}
({2.4},{0}) node {\footnotesize{$x32123$}}
({2.4},{sin(15)}) node {\footnotesize{$x321231$}}
({2.43},{sin(15)+0.32}) node [gray] {\footnotesize{$(12121232)$}}
({2.45},{sin(15)+0.25+0.32}) node [gray] {\footnotesize{$(121212321)$}}
({2.4},{sin(15)+2*0.25+0.32}) node [gray] {\footnotesize{$(121212)$}}
({1.89},{sin(15)+2*0.25+0.32+0.1}) node [gray] {\footnotesize{$(12121232123)$}}
({2.42},{sin(15)+2*0.25+0.32+0.1}) node [gray] {\footnotesize{$(1212123)$}}
({1.98},{sin(15)+0.25}) node {\footnotesize{$x3212312$}}
({1.965},{sin(15)+2*0.25}) node {\footnotesize{$x32123121$}}
({1.94},{sin(15)+3*0.25}) node {\footnotesize{$x321231212$}}
({2.2},{sin(15)+4*0.25+0.1}) node {\footnotesize{$x321231212$}}
({1.95},{-0.57}) node [gray] {\footnotesize{$(12121)$}}
({1.7},{0.075}) node {\footnotesize{$x321$}}
({1.73},{sin(15)-0.085}) node [gray] {\footnotesize{$(121212)$}}
({1.95},{-0.065}) node {\footnotesize{$x3212$}}
({1.95},{0.32}) node {\footnotesize{$x32121$}}
({cos(15)-0.1},{sin(15)}) node {\footnotesize{$x1$}}
({1.2},{sin(15)+0.075}) node {\footnotesize{$x13$}}
({1.45},{sin(15)-0.085}) node [gray] {\footnotesize{$(232)$}}
({cos(30)+0.12},{sin(30)}) node {\footnotesize{$x12$}}
({cos(45)+0.14},{sin(45)+0.02}) node {\footnotesize{$x121$}}
({cos(60)+0.14},{sin(60)+0.06}) node {\footnotesize{$x1212$}}
({cos(75)+0.13},{sin(75)+0.07}) node {\footnotesize{$x12121$}}
({cos(-15)-0.1},{sin(-15)}) node {\footnotesize{$x2$}}
({cos(-30)-0.125},{sin(-30)}) node {\footnotesize{$x21$}}
({cos(-45)+0.125},{sin(-45)}) node {\footnotesize{$x212$}}
({cos(-60)+0.1},{sin(-60)-0.05}) node {\footnotesize{$x2121$}}
({cos(-75)+0.13},{sin(-75)-0.07}) node {\footnotesize{$x21212$}}
;\end{tikzpicture}};
\end{tikzpicture}
\caption{Cannon automata for Class~III triangle groups with $a$ even}\label{fig:III}
\end{figure}

\end{appendix}

\medskip

\noindent\begin{minipage}{0.5\textwidth}
Lorenz Gilch\newline
Institut f\"{u}r Mathematische Strukturtheorie\newline
Technische Universit\"{a}t Graz\newline
Steyrergasse 30\newline
8010 Graz, Austria\newline
\texttt{gilch@TUGraz.at}
\end{minipage}
\begin{minipage}{0.5\textwidth}
\noindent Sebastian M\"{u}ller\newline
Aix Marseille Universit\'e\newline
CNRS  Centrale Marseille\newline 
Institut de Math\'ematiques de Marseille (I2M)\newline
UMR 7373\newline 13453 Marseille France\newline
\texttt{mueller@cmi.univ-mrs.fr}
\end{minipage}
\bigskip

\noindent James Parkinson\newline
School of Mathematics and Statistics\newline
University of Sydney\newline
Carslaw Building, F07\newline
NSW, 2006, Australia\newline
\texttt{jamesp@maths.usyd.edu.au}


\begin{thebibliography}{100}


\bibitem{AB} P. Abramenko, K. Brown \textit{Buildings: Theory and Applications}, Graduate Texts in Mathematics, 248, Springer, 2008.


\bibitem{BaBe:99} E. Babson, I. Benjamini \textit{Cut sets and normed cohomology with applications to percolation},  Proc. Amer. Math. Soc. 127:589--597, 1999.

\bibitem{bellman} R.~Bellman, \textit{Limit theorems for non-commutative operations I}, Duke Math. J., 21:491--500, 1954.

\bibitem{BeQu:13}  Y. Benoist, J.-F. Quint \textit{Central Limit Theorem for linear groups}, preprint, 2013.

\bibitem{BeQu:14} Y. Benoist, J.-F. Quint \textit{Central Limit Theorem for hyperbolic groups}, preprint, 2014. 

\bibitem{bjorklund} M.~Bj\"{o}rklund, \textit{Central Limit Theorems for Gromov hyperbolic groups}, Journal of Theoretical Probability, 23:871--887, 2010.

\bibitem{bjorner} A. Bj\"{o}rner, F. Brenti, \textit{Combinatorics of Coxeter Groups}, Graduate Texts in Mathematics, 231, Springer, 2005.

\bibitem{bougerol} P. Bougerol, \textit{Th\'eor\`eme central limite local sur certains groupes de Lie}, Ann. Sci. \'{E}cole Norm. Sup. (4)  14, no. 4, 403--432, 1981.

\bibitem{howlett} B. Brink, B. Howlett, \textit{A finiteness property and an automatic structure for Coxeter groups}, Mathematische Annalen, 296:179--190, 1993.


\bibitem{CW} D.I. Cartwright, W. Woess, \textit{Isotropic random walks in a building of type $\tilde{A}_d$}, Mathematische Zeitschrift, 247, 101--135, 2004.

\bibitem{davis} M.W. Davis, \textit{The geometry and topology of Coxeter groups}, London Mathematical Society Monograph Series, Princeton University Press, 32, 2008.


\bibitem{cannon} D.B.A. Epstein, J.W. Cannon, D.F. Holt, V.F. Levy, M.S. Paterson, W.P. Thurston, \textit{Word processing in groups}, Jones and Bertlett Publishers, Boston, MA, 1992. 


\bibitem{feithigman} W. Feit, G. Higman, \textit{The nonexistence of certain generalized polygons}, J. Algebra, 1:114--131, 1964. 


\bibitem{FK} H.~Furstenberg, H.~Kesten, \textit{Products of random matrices}, Ann. Math. Statist., 31:457--469, 1960.

\bibitem{gilch:PhD} L. Gilch, \textit{Rate of Escape of Random Walks},
  PhD. thesis, University of Technology Graz, Austria, 2007.

 \bibitem{GoGu:96} I. Goldsheid, Y. Guivarc'h, \textit{Zariski closure and the dimension of the {G}aussian law of the product of random matrices. {I}}, Probab. Theory Related Fields, vol. \textbf{105}, 109--142, 1996.


\bibitem{Gournay} A. Gournay, \textit{A remark on the connectedness of spheres
     in Cayley graphs}, Comptes Rendus Mathematique, vol. \textbf{53}, Issues
   7--8, 573--576, 2014.


  \bibitem{Gui:80} Y. Guivarc'h, \textit{Loi des grands nombres et rayon spectral d'une marche al\'eatoire sur un groupe de Lie}, Ast\'erisque, vol. \textbf{74}, 47--98, 1980.

\bibitem{guivarc'h2} Y. Guivarc'h, M. Keane, P. Roynette, \textit{Marches al\'eatoires sur les groupes de Lie}, LNM Vol. \textbf{624}, Springer-Verlag, 1977.

\bibitem{GLeP:04} Y.~Guivarc'h, \'E. Le~Page, \textit{Simplicit\'e de spectres de Lyapunov et propri\'et\'e d'isolation spectrale pour une famille d'op\'erateurs de transfert sur l'espace projectif.}, 181--259, Berlin: de Gruyter, 2004.


\bibitem{HMM:13} P.~Ha\"{\i}ssinsky, P.~Mathieu, S.~M\"uller \textit{Renewal theory for random walks on  surface groups}, preprint, 2013.


\bibitem{humphreys} J.E. Humphreys, \textit{Reflection Groups and Coxeter Groups}, Cambridge Studies in Advanced Mathematics, 29, 1990.

\bibitem{IM} N. Iwahori, H. Matsumoto, \textit{On some Bruhat decomposition and the structure of the Hecke rings of $p$-adic Chevalley groups}, Publications Math\'{e}matiques de l'IH\'{E}S (25): 5--48, 1965.



\bibitem{kac} V. Kac, \textit{Infinite dimensional dimensional Lie algebras}, third ed., Cambridge University Press, 1990. 


\bibitem{kaimanovich} V. Kaimanovich, \textit{Lyapunov exponents, symmetric spaces, 
and a multiplicative ergodic theorem for semisimple Lie groups}. ) Zap. Nauchn. Sem. Leningrad. Otdel. 
Mat. Inst. Steklov. (LOMI) \textbf{164}, 29--46, 196--197, 1987; Russian, English translation in
J. Soviet Math. \textbf{47}, 2387--2398, 1989.

\bibitem{La:93} S.~P. Lalley, \textit{Finite range random walk on free groups and homogeneous trees}, Ann. Probab., 21(4):2087--2130, 1993.


\bibitem{Le:01} F. Ledrappier, \textit{Some asymptotic properties of random walks on free groups}, CRM Proceedings and 
Lectures Notes, 28, 2001. 

\bibitem{LePage} E.~Le~Page, \textit{Th\'{e}or\`{e}mes de la limite centrale pour certains produits de matrices
  al\'{e}atoires}, C.R. Acad. Sci. Paris S\'{e}r. 1 Math, 292(6):379--382, 1981.


\bibitem{LV} M. Lindlbauer, M. Voit, \textit{Limit theorems for isotropic random walks on triangle buildings}, J. Aust. Math. Soc, 73,  301--333, 2002.
   
   
\bibitem{NW} T. Nagnibeda, W. Woess, \textit{Random walks on trees with Þnitely many cone types}, J. Theoret. 
Probab., 15(2):383Ð422, 2002. 

\bibitem{P1} J. Parkinson, \textit{Buildings and Hecke algebras}, J. Algebra, 297:1--49, 2006. 

\bibitem{P3} J. Parkinson, \textit{Isotropic random walks on affine buildings}, Annales de l'Institut Fourier 57, No.2, 379--419, 2007.

\bibitem{PS} J. Parkinson, B. Schapira, \textit{A local limit theorem for random walks on the chambers of $\tilde{A}_2$ buildings}, Progress in Probability, 64, Birkh\"{a}user, 15--53, 2011.



\bibitem{ronanconstruction} M. Ronan, \textit{A construction of buildings with no rank 3 residues of spherical type.} Buildings and the Geometry of Diagrams, Lecture Notes in Mathematics, 1181, 242--248, 1986.


\bibitem{ST} S.~Sawyer, T.~Steger, \textit{The rate of escape for anisotropic random walks in a tree}, Probability Theory and Related Fields, 76:207--230, 1987.

\bibitem{BS} B. Schapira, \textit{Random walk on a building of type $\tilde{A}_r$ and Brownian motion of the Weyl chamber}, Annales de l'I.H.P. (B) 45, 289--301, 2009.

\bibitem{steinberg} R. Steinberg, \textit{Lecture notes on Chevalley groups}, Tale University, 1967. 

\bibitem{Stroock:73} D. Stroock, S.R.S. Varadhan, \textit{Limit Theorems for random walks on Lie groups}, Indian J. of Stat., Sankhy$\bar{\text{a}}$, Ser. A, 35, 277--294, 1973. 


\bibitem{Ti74} J. Tits, Buildings of spherical type and finite BN-pairs. {\it Lecture Notes in Mathematics} Vol. 386 Springer-Verlag, Berlin-New York, 1974. 

\bibitem{titskac} J. Tits, \textit{Uniqueness and presentation of Kac-Moody groups over fields}, J. Algebra, 105(2): 542--573, 1987



\bibitem{tolli} F. Tolli, \textit{A local limit theorem on certain $p$-adic groups and buildings}, Monatsh. Math., 133, 163--173, 2001.


\bibitem{Tutubalin:65} V.~N. Tutubalin, \textit{Limit theorems for a product of random matrices}, Teor. Verojatnost. i Primenen., 10:19--32, 1965.



\bibitem{HVM} H. Van Maldeghem, \textit{Generalized Polygons}, Monographs in Mathematics, 93, Birkh\"{a}user, Basel, Boston, Berlin, 1998. 

\bibitem{virtser} A. D. Virtser, \textit{Central limit theorem for semi-simple Lie groups}, Theory Probab. Appl. \textbf{15}, 667--687, 1970.

\bibitem{wehn:62} D. Wehn, \textit{Probabilities on Lie groups}, Proc. Nat. Acad. Sci. USA, 48, 791--795, 1962.


\bibitem{woessbook} W. Woess, \textit{Random walks on infinite graphs and groups}. Cambridge Tracts in Mathematics 138, Cambridge University Press, 2000.


      
\end{thebibliography}
\end{document}